\documentclass[12pt, reqno, draft
]{amsart}
\usepackage{amsmath}
\usepackage{amstext}
\usepackage{amsbsy}
\usepackage{amsopn}
\usepackage{upref}
\usepackage{amsthm}
\usepackage{amsfonts}
\usepackage{amssymb}
\usepackage{mathrsfs}
\usepackage{times}
\usepackage{fancybox} 
\usepackage{ascmac}
\allowdisplaybreaks
 \newtheorem{theorem}{Theorem}[section]
 
 \newtheorem{proposition}[theorem]{Proposition}

\theoremstyle{definition}
 
 \theoremstyle{remark} 
 \newtheorem{remark}[theorem]{Remark}
\newcommand{\ep}{\varepsilon}
\newcommand{\p}{\partial}
\newcommand{\RR}{\mathbb{R}}
\newcommand{\TT}{\mathbb{T}}
\renewcommand{\Re}{\operatorname{Re}}
\renewcommand{\Im}{\operatorname{Im}}
\DeclareMathOperator{\diag}{diag}
\numberwithin{equation}{section}
\begin{document}

\title[]
{Local well-posedness of the initial value problem\\
for a fourth-order nonlinear dispersive system\\
 on the real line
}
\author[E.~Onodera]{Eiji Onodera}
\address[Eiji Onodera]{Department of Mathematics and Physics, 
Faculty of Science and Technology, 
Kochi University, 
Kochi 780-8520, 
Japan}
\email{onodera@kochi-u.ac.jp}
%
%
%
\begin{abstract}
This paper investigates the initial value problem for 
a system of one-dimensional fourth-order 
dispersive partial differential-integral equations with nonlinearity
involving derivatives up to second order. 
Examples of the system arise in relation 
with nonlinear science and geometric analysis. 
Applying the energy method based on the idea of a gauge transformation 
and Bona-Smith approximation technique, 
we prove that the initial value problem is time-locally well-posed 
on the real line for initial data in a Sobolev space with high regularity. 
\end{abstract}
%
\maketitle
%
%
\section{Introduction}
\label{section:introduction}
This paper investigates the initial value problem for an $n$-component 
system of fourth-order nonlinear  
dispersive partial differential-integral equations on the real line: 
 \begin{alignat}{2}
 \left(
 \p_t-iM_a\p_x^4-M_b\p_x^3-iM_{\lambda}\p_x^2
 \right)
 Q
  &=F(Q, \p_xQ, \p_x^2Q)
  & \quad
  & \text{in}
    \quad
    \RR\times\RR,
 \label{eq:apde}
 \\
  Q(0,x)
   &=
   Q_{0}(x)
 & \quad
 & \text{in}
   \quad
   \RR, 
 \label{eq:adata}
 \end{alignat}
where $n$ is a positive integer, 
$Q={}^t(Q_1,\ldots,Q_n)(t,x):\RR\times \RR\to \mathbb{C}^n$ 
 is an unknown function,  
 $Q_0={}^t(Q_{01},\ldots,Q_{0n})(x):\RR\to \mathbb{C}^n$ 
 is a given initial function, 
$i=\sqrt{-1}$, 
$a=(a_1,\ldots,a_n)\in (\RR\backslash \left\{0\right\})^n$,  
$M_a=\diag(a_1,\ldots,a_n)$, 
$b=(b_1,\ldots, b_n)\in \RR^n$, 
$M_{b}=\diag(b_1,\ldots,b_n)$, 
$\lambda=(\lambda_1,\ldots, \lambda_n)\in \RR^n$, 
$M_{\lambda}=\diag(\lambda_1,\ldots,\lambda_n)$, 
and
 \begin{align}
   F(Q,\p_xQ,\p_x^2Q)
   &=
   {}^t(F_1(Q,\p_xQ,\p_x^2Q),\ldots,F_n(Q,\p_xQ,\p_x^2Q)) 
 \nonumber
 \end{align}
is a nonlinear expression of 
$Q$, $\p_xQ$, $\p_x^2Q$ and their complex conjugates
$\overline{Q}$, $\overline{\p_xQ}$, $\overline{\p_x^2Q}$. 
It is supposed that 
each of $F_j(Q,\p_xQ,\p_x^2Q)$ for $j\in \left\{1,\ldots,n\right\}$ 
takes the form
  \begin{align}
   F_j(Q,\p_xQ,\p_x^2Q)
   &=
   F_j^1(Q, \p_x^2Q)+F_j^2(Q,\p_xQ)+F_j^3(Q,\p_xQ), 
   \nonumber
   \\
   F_j^3(Q,\p_xQ)(t,x)
          &=
          \sum_{r=1}^{n}
          \left(\int_{-\infty}^x
            F_{j,r}^{3,A}(Q,\p_xQ)(t,y)
            dy\right)
             F_{j,r}^{3,B}(Q)(t,x), 
    \nonumber
   \end{align} 
   and all of the following conditions
      (F1)-(F3) are satisfied:  
   \begin{enumerate}
   \item[(F1)] 
$F_j^1(u,w)$ for each 
 $j\in \left\{1,\ldots,n\right\}$ 
 is a complex-valued polynomial in 
      $u,\overline{u}$, $w,\overline{w}\in \mathbb{C}^{n}$ 
satisfying
      \begin{align}
      &|F_j^1(u, w)|
      \leqslant c_{j}^1
      |u|^2
      |w|
      \quad 
      \text{for any $u, w\in \mathbb{C}^{n}$},
      \label{eq:F1}
      \end{align} 
 where $c_j^1>0$ is a constant which may depend on $j$
  but not on $u,w$.   
      \item[(F2)]
There exist integers $d_1, d_2\geqslant 0$ such that 
$F_j^2(u,v)$  for all $j\in \left\{1,\ldots,n\right\}$ 
is a complex-valued polynomial in
$u,\overline{u},v,\overline{v}\in \mathbb{C}^{n}$ satisfying  
            \begin{equation}
            |F_j^2(u,v)|
            \leqslant 
            c_j^2
            \sum_{p_1=0}^{d_1}
            \sum_{p_2=0}^{d_2}
            |u|^{1+p_1}|v|^{p_2}
            \quad 
            \text{ for any $u,v\in \mathbb{C}^n$}, 
            \label{eq:F2}
            \end{equation}
            where $c_j^2>0$ is a positive constant which may depend on $j$
            but not on $u,v$. 
      \item[(F3)] 
There exist integers $d_3,d_4,d_5\geqslant 0$ such that $F_{j,r}^{3,A}(u,v)$ and $F_{j,r}^{3,B}(u)$  for all 
$j,r\in \left\{1,\ldots,n\right\}$ 
are respectively complex-valued polynomials in 
           $u$,$\overline{u}$,$v$, $\overline{v}\in \mathbb{C}^{n}$ and in $u,\overline{u}\in \mathbb{C}^{n}$   
           satisfying
          \begin{align}
          \hspace{2em}|F_{j,r}^{3,A}(u,v)|
          &\leqslant c_{j,r}
          \left(
          \sum_{p_3=0}^{d_3}|u|^{2+p_3}+\sum_{p_4=0}^{d_4}|v|^{2+p_4}
          \right)
          \ 
          \text{for any $u,v\in \mathbb{C}^n$}, 
          \label{eq:F31}
          \\
          |F_{j,r}^{3,B}(u)|
          &\leqslant 
          c_{j,r}\sum_{p_5=0}^{d_5}
          |u|^{1+p_5}
          \ \ 
          \text{for any $u\in \mathbb{C}^{n}$}, 
          \label{eq:F32}
          \end{align}
where $c_{j,r}>0$  
is a positive constants which may depend on $j,r$
but not on $u,v$. 
   \end{enumerate}
\par 
Examples of \eqref{eq:apde} 
satisfying (F1)-(F3)
with nonlinearity involving derivatives up to second order 
arise in some fields of nonlinear science, 
in which the nonlocal terms satisfying (F3) are not involved, 
that is, 
$F_{j}^{3}(Q,\p_xQ)\equiv 0$ for all 
$j\in \left\{1,\ldots,n\right\}$.   
They include single equations (in the case of $n=1$)
which are related to 
the vortex filament (\cite{fukumoto,FM}), 
continuum models of Heisenberg spin chain systems (\cite{DKA,LPD,PDL})
and alpha-helical proteins (\cite{DL}). 
(See \eqref{eq:4shro} in Section~\ref{section:Background}.)
They also include an $n$-component system to 
study the wave propagation of $n$ distinct ultrashort optical fields in a fiber 
(\cite{WZY}). 
(See \eqref{eq:WZY} in Section~\ref{section:Background}.)
\par 
Examples of \eqref{eq:apde} 
satisfying (F1)-(F2) and (F3) with non-vanishing 
nonlocal terms have their origin in 
geometric dispersive partial differential equations (PDEs) 
having been investigated in \cite{DW2018,DZ2021,onoderamomo}. 
The geometric equations describe the evolution of a map 
 $u(t,\cdot):M \to N$, 
 where $M$ is a Riemannian manifold and 
 $N$ is a K\"ahler (or para-K\"ahler) manifold. 
 It can be also said that they describe a curve flow on $N$ if $M=\RR$.
 Roughly speaking, each of them 
 can be transformed to 
 a system of nonlinear fourth-order dispersive PDEs 
 for complex-valued functions (including the case of a single equation) 
 if $M=\RR$ with Euclidean metric, 
 and  the derived system satisfies the structure of 
  \eqref{eq:apde} with (F1)-(F3) under some geometric assumptions on 
  the K\"ahler manifold $N$. 
The component $n$ of \eqref{eq:apde} in this context   
    is equivalent to the complex dimension of  $N$.
  The transformation can be comprehensively regarded as a kind of 
   the so-called generalized Hasimoto transformation. 
(See \eqref{eq:Ondq2}, \eqref{eq:2241b}, and \eqref{eq:2241a} 
in Section~\ref{section:Background}.)
\par 
Our goal of this paper is to show \eqref{eq:apde}-\eqref{eq:adata} 
is time-locally well-posed for initial data in a Sobolev space 
with high regularity. 
This is an attempt to present a framework that can solve the 
initial value problem for the examples mentioned above  
comprehensively. 
This is also an attempt to interpret the solvable structure of the above 
geometric dispersive PDEs for curve flows with values into $N$ of complex-dimension $n\geqslant 2$, 
in the level of the system \eqref{eq:apde} for $\mathbb{C}^n$-valued functions. 
Applying the energy method based on the idea of a gauge transformation 
and Bona-Smith approximation technique, 
we prove \eqref{eq:apde}-\eqref{eq:adata} 
is time-locally well-posed in Sobolev space $H^m(\RR;\mathbb{C}^n)$ 
for integer $m\geqslant 4$ (Theorem~\ref{theorem:lwp}). 
Time local well-posedness for systems 
 (except for the case of a single equation)
of fourth-order dispersive PDEs for complex-valued functions 
 with nonlinearity involving derivatives up to second order 
 seems to be established for the first time in this paper, 
 with or without nonlocal terms. 
 See Section~\ref{section:results} for other contributions of our results and 
 for related known results.
 \par 
 The strategy and the idea to prove our main results (Theorem~\ref{theorem:lwp}) 
 are outlined in Section~\ref{section:idea}.  
The idea of the gauge transformation   
is to bring out the local smoothing effect of dispersive equations on 
$\RR$ and overcomes the difficulty of the 
loss of derivatives occurred from the nonlinear terms 
$F_j^1(Q,\p_x^2Q)$ and $F_j^2(Q,\p_xQ)$ with the conditions (F2) and (F3). 
The idea is motivated from \cite{CO2} 
which investigated a fourth-order geometric dispersive PDE for curve flows on 
a compact K\"ahler manifold. 
For better readability of the idea and our proof, we additionally illustrate the idea 
with the initial value problem for a linear dispersive 
PDE for complex-valued functions, showing Proposition~\ref{prop:linear}. 
On the other hand, 
the Bona-Smith approximation technique is often useful to derive  
the continuous dependence of the solution with respect to the initial data, 
which is also the case for our problem. 
Not only that, we adopt the technique 
to construct a time-local solution, 
the difficulty of which comes from the presence of nonlocal terms in $F_j^3(Q,\p_xQ)$ with 
the condition (F3). 
See Section~\ref{section:idea} and Remark~\ref{remark:bs} in Section~\ref{section:local} 
for the detail. 
\par 
The framework presented in this paper 
is certainly applicable to all the examples mentioned above comprehensively, 
whereas Theorems~\ref{theorem:lwp}  may still has a room for improvement 
from the viewpoint of mathematical analysis of nonlinear PDEs.     
Our proof heavily relies on the idea of the gauge transformation in \cite{CO2} 
and the local smoothing effect for dispersive PDEs brought out via it  is not sharp  
as is pointed out in \cite{CO2}.  
If we can make full use of the smoothing effect via another method and avoid any obstructions 
due to the presence of nonlocal terms,   
then the assumption on the regularity of the data   
will be improved or the conditions on the nonlinearity will be relaxed. 
Moreover, 
it seems that 
our proof of Theorem~\ref{theorem:lwp} handles   
\eqref{eq:apde}  as if 
it is close to $n$-pieces of single equations for complex-valued functions 
 which can be investigated separately, and does not make use of 
 the structure of \eqref{eq:apde} as a system.  
For example, if we can 
provide a classification of \eqref{eq:apde} 
in terms of the regularity of the Sobolev space  
based on the dispersion coefficients and other coefficients of nonlinear terms, 
it will be more interesting.   
These directions are not pursued in this paper, although they seem to be worth 
investigating.  
\par 
The organization of the present paper is as follows: 
In Section~\ref{section:Background}, we review the background 
of \eqref{eq:apde} satisfying (F1)-(F3) in more detail. 
In Section~\ref{section:results}, 
we state Theorem~\ref{theorem:lwp} and the contribution.
In Section~\ref{section:idea} and Appendix, we illustrate the strategy and the idea of the 
proof of Theorem~\ref{theorem:lwp}.   
In Sections~\ref{section:local}-\ref{section:prooflw}, 
we complete the proof of Theorem~\ref{theorem:lwp}. 
 \section*{Notation used throughout this paper}
\label{section:notation} 
      Different positive constants are sometimes denoted by the same $C$ for simplicity, 
      if there seems to be no confusion. 
      Expressions such as $C=C(\cdot,\ldots,\cdot)$, 
      $C_k=C_k(\cdot,\ldots,\cdot)$, and $D_k=D_k(\cdot,\ldots,\cdot)$ are 
      used to show the dependency on quantities appearing in parenthesis. 
      Other symbols to denote a constant are explained on each occasion.
   For nonnegative integers $j$ and $k$, 
   the set of integers $\ell$ with $j\leqslant \ell\leqslant j+k$
   is denoted by $\left\{j,\ldots,j+k\right\}$. 
   The partial differentiation for functions is written by 
   $\p$ or the subscript, e.g., $\p_xf$, $f_x$.
   \par 
   For any 
   $z={}^{t}(z_1,\ldots,z_n)$ and $w={}^{t}(w_1,\ldots,w_n)$ 
   in $\mathbb{C}^n$, their inner product is defined by 
   $z\cdot w=\displaystyle\sum_{j=1}^{n}z_j\overline{w_j}$, 
   and the norm of $z$ is by 
   $|z|=(z\cdot z)^{1/2}$. 
   (Although the same $|\cdot|$ is often used to denote the absolute value 
   of a complex number in $\mathbb{C}$, 
   the author expects it does not cause great confusion.)  
   Moreover, we set  
      \begin{align}
      g(z)&=|z|^{2} \quad \text{for} \quad z\in \mathbb{C}^n, 
      \label{eq:gz}
      \end{align}
      which will be used for readability 
      of the role of (F1) and our gauge transformation.
   \par 
  The $L^2$-space of $\mathbb{C}^n$-valued functions on $\RR$ is denoted by 
  $L^2(\RR;\mathbb{C}^n)$ being the set of all measurable functions 
    $f={}^{t}(f_1,\ldots,f_n):\RR\to \mathbb{C}^n$ such that 
    $$
    \|f\|_{L^2}:=\left(
    \int_{\RR}|f(x)|^2\,dx
    \right)^{1/2}
    =\left(
    \sum_{j=1}^n
      \int_{\RR}f_j(x)\overline{f_j(x)}\,dx
      \right)^{1/2}
    <\infty.
    $$
  The $L^2$-type  Sobolev space for a nonnegative integer $k$ is denoted by 
  $H^k(\RR;\mathbb{C}^n)$  
   being the set of all measurable functions 
   $f={}^{t}(f_1,\ldots,f_n):\RR\to \mathbb{C}^n$ such that 
  $\p_x^{\ell}f\in L^2(\RR;\mathbb{C}^n)$ for  
  all $\ell\in \left\{0,\ldots,k\right\}$. 
  The norm $\|f\|_{H^k}$ of $f\in H^k(\RR;\mathbb{C}^n)$ 
  is defined by  
  $$ 
   \|f\|_{H^k}:=\left(
   \sum_{\ell=0}^{k}\|\p_x^{\ell}f\|_{L^2}^2
   \right)^{1/2}.
   $$ 
   Moreover, $H^{\infty}(\RR;\mathbb{C}^n)$ denotes the intersection 
   of all $H^k(\RR;\mathbb{C}^n)$ for nonnegative integers $k$. 
Furthermore, $C([t_1,t_2];H^k(\RR;\mathbb{C}^n))$ denotes the 
Banach space of $H^k(\RR;\mathbb{C}^n)$-valued continuous functions on 
the interval $[t_1,t_2]$ with the norm 
$\|Q\|_{C([t_1,t_2];H^k)}:=\displaystyle\sup_{t\in [t_1,t_2]}\|Q(t,\cdot)\|_{H^k}$.
 \section{Background of \eqref{eq:apde}}
 \label{section:Background}
 This section aims at reviewing  
 the background of \eqref{eq:apde}  
 satisfying (F1)-(F3) more concretely. 
 \par 
 First, \eqref{eq:apde} is a multi-component extension of the 
 single nonlinear dispersive partial differential equation (PDE):  
  \begin{align}
  (\p_t-i\nu\p_x^4-i\p_x^2)\psi
  &=
  \mu_1|\psi|^2\psi
  +\mu_2|\psi|^4\psi
  +\mu_3(\p_x\psi)^2\overline{\psi}
  +\mu_4|\p_x\psi|^2\psi
  \nonumber
  \\
  &\quad
  +\mu_5\psi^2\overline{\p_x^2\psi}
  +\mu_6|\psi|^2\p_x^2\psi
  \label{eq:4shro}
  \end{align}
  for $\psi=\psi(t,x):\RR\times \RR\to \mathbb{C}$, 
  where $\mu_j$ for $j\in \left\{1,\ldots,6\right\}$ and $\nu\ne 0$ 
  are real constants. 
  The equation arises in relation with  
  the continuum limit of a one-dimensional isotropic Heisenberg ferromagnetic 
  spin chain systems with nearest neighbor bilinear and bi-quadratic 
  exchange interaction (\cite{LPD,PDL}), 
  the continuum limit of a one-dimensional anisotropic Heisenberg
  ferromagnetic spin chain systems with octupole-dipole interaction (\cite{DKA}),  
  the three-dimensional motion of a vortex filament with elliptical deformation effect 
  of the core in an incompressible viscous fluid (\cite{fukumoto,FM}), 
  and the molecular excitations along the hydrogen bonding spine 
  in an alpha-helical protein with
  higher-order molecular interactions(\cite{DL}). 
  \par 
  Second, \eqref{eq:apde} is a fourth-order extension 
  of the system of nonlinear Schr\"odinger equations which has 
  had much attention to study the 
  interactions of many bodies.  
  The following is the example of \eqref{eq:apde} in this context:  
   \begin{align}
   {\bf q}_t
   &=i\alpha \left(
   \frac{1}{2} 
   {\bf q}_{xx}+{\bf q}{\bf q}^{*}{\bf q}
   \right)
   -\ep \Biggl[
     \dfrac{1}{2}{\bf q}_{xxx}
     +\dfrac{3}{2}({\bf q}_x{\bf q}^{*}{\bf q}
     +{\bf q}{\bf q}^{*}{\bf q}_x)
     \Biggr]
  \nonumber
  \\&\quad
   +i\gamma
   \Biggl[
   \dfrac{1}{2}{\bf q}_{xxxx}
   +{\bf q}({\bf q}_x^{*}{\bf q})_x
   +{\bf q}_x{\bf q}_x^{*}{\bf q}
   \nonumber
   \\
   &\qquad\qquad
   +2(
   {\bf q}_{xx}{\bf q}^{*}{\bf q}
   +
   {\bf q}{\bf q}^{*}{\bf q}_{xx}
   )
   +3\left\{
   {\bf q}_{x}{\bf q}^{*}{\bf q}_x
   +
   {\bf q}({\bf q}^{*}{\bf q})^2
   \right\}
   \Biggr]
   \label{eq:WZY}
   \end{align}
   for ${\bf q}(t,x)={}^t(q_1(t,x),\ldots,q_n(t,x)):\RR\times \RR\to \mathbb{C}^n$, 
   where $\gamma\ne 0$, $\alpha$ and  $\ep$ are real constants, 
   ``$*$" denotes the Hermitian transpose. 
   It is pointed out in \cite{WZY} that   
   \eqref{eq:WZY} investigates the wave propagation 
   of $n$ distinct ultrashort optical fields in a fiber, 
   and models the broadband, ultrashort pulses propagation.     
 \par 
 Third, examples of \eqref{eq:apde} satisfying (F1)-(F2) and (F3) 
 with non-vanishing nonlocal terms have their origin in 
 geometric dispersive PDEs:  
  The equation  
  for the so-called generalized bi-Schr\"odinger flow(GBSF) 
  was introduced by Ding and Wang in \cite{DW2018}, 
  which is formulated for  
  $u=u(t,x):(-T,T)\times M \to N$, 
  where $M$ is a Riemannian manifold and 
  $N$ is a K\"ahler (or para-K\"ahler) manifold. 
  \begin{itemize}
   \item 
   When $M=\RR$ (with Euclidean metric) and $N$ is either of $G_{n_0,k_0}$ or $G_{n_0}^{k_0}$, 
   it is revealed in \cite{DW2018} that 
   the equation for GBSF can be equivalently reduced to 
   a  fourth-order matrix nonlinear dispersive partial differential-integral equation 
   for $q=q(t,x):(-T,T)\times \RR\to \mathcal{M}_{k_0\times (n_0-k_0)}$, 
   where $G_{n_0,k_0}$ (resp. $G_{n_0}^{k_0}$) denotes the complex Grassmannian of compact 
   (resp. noncompact) type as a Hermitian symmetric space and 
   $\mathcal{M}_{k_0\times (n_0-k_0)}$ stands for 
   the space of $k_0\times (n_0-k_0)$ complex-matrices. 
   For example, when $N=G_{n_0,k_0}$ for integers $n_0,k_0$ with 
    $1\leqslant k_0<n_0$ being a K\"ahler manifold of complex dimension $n=k_0(n_0-k_0)$, 
     the results in \cite{DW2018} tell that 
    the equation for $q$ can be formulated by
    \begin{align}
    q_t
    &=
    -i\,\alpha
    \biggl\{
    q_{xx}+2qq^{*}q\biggr\}
    +i\,\beta
    \biggl\{
    q_{xxxx}+4q_{xx}q^{*}q+2qq^{*}_{xx}q+4qq^{*}q_{xx}
    \nonumber
    \\
    &\quad
    +2q_xq^{*}_xq+6q_xq^{*}q_x+2qq^{*}_xq_x+6qq^{*}qq^{*}q
    \biggr\}
    \nonumber
    \\
    &\quad
    -2i\,(\beta+8\gamma)
    \biggl\{
    (qq^{*}q)_{xx}+2qq^{*}qq^{*}q
    \nonumber
    \\
    &\qquad
    +q\left(\int_{-\infty}^x
    q^{*}(qq^{*})_sq\,ds
    \right)
    +
    \left(
    \int_{-\infty}^x
    q(q^{*}q)_sq^{*}\,ds
    \right)q
    \biggr\} 
    \label{eq:Ondq2}
    \end{align} 
       for $q=q(t,x):(-T,T)\times \RR\to \mathcal{M}_{k_0\times (n_0-k_0)}$, 
       where $\beta\neq 0$ and $\alpha$ are real constants.  
     In this setting, 
     for any $j\in \left\{1,\ldots,n\right\}$, 
     there exists a unique pair of integers $j_1\in \left\{1, \ldots, k_0\right\}$ and 
     $j_2\in \left\{1, \ldots, n_0-k_0\right\}$ such that 
     $j=(j_2-1)k_0+j_1$. 
     Hence, if we set $Q_j$ to be the $(j_1,j_2)$-component of $q$ 
     for $j=(j_2-1)k_0+j_1\in \left\{1,\ldots,n\right\}$,   
     then the equation for 
     $Q={}^t(Q_1,\ldots,Q_n)$ turns out to be a specialization of \eqref{eq:apde} 
     with (F1)-(F3).  
   \item When $M=\RR$ and $N$ is a Riemann surface, 
    the equation for GBSF can be reduced to 
    a single partial differential-integral equation for $q=q(t,x):(-T,T)\times \RR\to \mathbb{C}$, 
    which is proved in \cite{DZ2021} by using the generalized Hasimoto transformation. 
  The explicit expression is available in \cite[Theorem~5.1]{DZ2021}.
    Looking at the the relation between $u$ and $q$ explained in the proof, 
    we see that, if $N$ is compact, then the partial derivative of the Gaussian curvature at $u(t,x)$ 
    with respect to $x$ is bounded by $|q(t,x)|$ multiplied by a positive constant, and thus 
    the equation for $q$ satisfies 
    \begin{align}
    \hspace{1.5em}(\p_t-i\beta\p_x^4+i\alpha \p_x^2)q
    &=
    O\left(
    |\p_x^2q||q|^2
    +
    |\p_xq|^2|q|
    +
    |\p_xq||q|^3
    +|q|^3+|q|^5
    \right)
    \nonumber
    \\
    &\quad
    +
    \left(\int_{-\infty}^x
    f_{1}(q,\p_xq)(t,y)\,dy
    \right)
    q, 
    \label{eq:2241b}
    \end{align}
    where  $f_{1}(q,\p_xq)=
      O\left(
      |\p_xq|^2|q|+|\p_xq||q|^2+|q|^3+|q|^5
      \right)$. 
     The equation \eqref{eq:2241b} satisfies the structure of 
     \eqref{eq:apde} with (F1)-(F3) for $n=1$.
    Additionally, the explicit expression stated in \cite[Theorem~5.1]{DZ2021} tells that 
    the nonlocal term remains in \eqref{eq:2241b} unless 
    $N$ has a constant Gaussian curvature. 
   \end{itemize} 
  \par
 In recent study \cite{onoderamomo} by the present author, 
 a similar fourth-order geometric dispersive PDE has been investigated, which reads
 \begin{alignat}{2}
   &u_t
    =
    a\,J_u\nabla_x^3u_x
    +
    \lambda\, J_u\nabla_xu_x
    +
    b\, R(\nabla_xu_x,u_x)J_uu_x
    +
    c\, R(J_uu_x,u_x)\nabla_xu_x
  \label{eq:pde}
  \end{alignat}
  for $u=u(t,x):\RR\times \RR\to N$, where 
  $a\ne 0$, $b$, $c$, $\lambda$ are real constants, 
  $N$ is a general compact K\"ahler manifold of complex dimension $n$
  with complex structure $J$, K\"ahler metric $h$, and with  
  the Levi-Civita connection $\nabla$ and the Riemann curvature tensor $R$ 
  associated to $h$.   
  (See e.g., \cite{onodera4,onoderamomo}, 
  for the details on the geometric setting of terms in \eqref{eq:pde}.) 
 Developing the generalized Hasimoto transformation,  
 the author in \cite{onoderamomo} has shown that 
 \eqref{eq:pde} can be transformed to a system for
 complex-valued functions $Q_1,\ldots,Q_n$. 
 If $n\geqslant 2$, then  
 the equation for each $Q_j$ satisfies 
 \begin{align}
  (\p_t-ia\p_x^4-i\lambda \p_x^2)Q_j
  &=
  O\left(
  |\p_x^2Q||Q|^2
  +
  |\p_xQ|^2|Q|
  +
  |\p_xQ||Q|^3
  +|Q|^3
  \right)
  \nonumber
  \\
  &\quad
  +\sum_{r=1}^{n}
  \left(\int_{-\infty}^x
  f_{j,r}(Q,\p_xQ)(t,y)\,dy
  \right)
  Q_r, 
  \label{eq:2241a}
  \end{align}
  where
  \begin{align}
  f_{j,r}(Q,\p_xQ)
  &=
  O\left(
  |\p_xQ|^2|Q|+|\p_xQ||Q|^2+|\p_xQ||Q|^3+|Q|^3
  \right).
  \label{eq:32241b}
  \end{align}
   The system of \eqref{eq:2241a} for 
   $j\in \left\{1,\ldots,n\right\}$ satisfies the structure of 
      \eqref{eq:apde} with (F1)-(F3). 
  \begin{remark}
  When $N$ is imposed to be locally Hermitian symmetric,   
   \eqref{eq:pde} coincides with 
    the equation for GBSF
    under an assumption on coefficients of the equation, 
    which is proved in \cite{onodera4}.  
    Note that the expression of the right hand side of \eqref{eq:2241a} with \eqref{eq:32241b}
    can change 
    depending on additional assumptions on $N$, 
    where the nonlocal terms are rewritten by the fundamental theorem of calculus. 
    One may notice that the structure of  
    \eqref{eq:2241a} with \eqref{eq:32241b} is slightly different from \eqref{eq:2241b}, 
    even though $G_{n_0,k_0}$ is also locally Hermitian symmetric. 
    However, it is not inconsistent by the above reason, which has been discussed in \cite{onoderamomo}.
  \end{remark}
\section{Main theorem}
\label{section:results}
Our main results is now stated as follows: 
\begin{theorem}
 \label{theorem:lwp}
 Let $m$ be an integer satisfying $m \geqslant 4$. 
 Then the initial value problem 
 \eqref{eq:apde}-\eqref{eq:adata} 
 is time-locally well-posed in $H^m(\RR;\mathbb{C}^n)$, that is, the following assertions hold:
 \begin{enumerate}
 \item[(i)] (Existence and uniqueness.) 
 For any $Q_0\in H^m(\RR;\mathbb{C}^n)$, there exists a time  
 $T=T(\|Q_0\|_{H^4})>0$ and a unique solution 
 $Q\in C([-T,T]; H^{m}(\RR;\mathbb{C}^n))$
 to \eqref{eq:apde}-\eqref{eq:adata}.
 \item[(ii)] (Continuous dependence with respect to the initial data.) 
 Suppose that $T>0$ and $Q\in C([-T,T];H^m(\RR;\mathbb{C}^n))$ are respectively 
 the time and the unique solution 
 to \eqref{eq:apde} with initial data $Q_0$ obtained in the above part $(\mathrm{i})$. 
 Fix  $T^{\prime}\in (0,T)$. 
 Then for any $\eta>0$, there exists $\delta>0$ such that for any 
 $\widetilde{Q_0}\in H^m(\RR;\mathbb{C}^n)$ satisfying 
 $\|Q_0-\widetilde{Q}_0\|_{H^m}<\delta$, 
 the unique solution $\widetilde{Q}$ 
 to \eqref{eq:apde} with initial data $\widetilde{Q}_0$ exists on 
 $[-T^{\prime},T^{\prime}]\times \RR$ and 
 satisfies 
 $\|Q-\widetilde{Q}\|_{C([-T,T^{\prime}];H^m)}<\eta$.
 \end{enumerate}
 \end{theorem}
 We state the contribution of Theorem~\ref{theorem:lwp} and related results. 
\par 
 First, 
to the best of the author's knowledge, 
 no previous results which established well-posedness of 
 \eqref{eq:apde}-\eqref{eq:adata} for $n\geqslant 2$   
 are available with or without nonlocal terms. 
 It might be better to mention that 
 the recent study by Malham \cite{Malham} has succeeded to  
 construct a time-local solution to 
 the following matrix nonlinear fourth-order dispersive PDE: 
   \begin{align}
   q_t&=
   \mu_2q_{xx}+\mu_3q_{xxx}+\mu_4q_{xxxx}
   +2\mu_2 qq^*q
   +3\mu_3\left(
   q_xq^*q+qq^{*}q_x
   \right)
   \nonumber
   \\
   &\quad
   +\mu_4
   \bigl(
   4q_{xx}q^{*}q+2qq^{*}_{xx}q+4qq^{*}q_{xx}
   +2q_xq^{*}_xq
   \nonumber
   \\
   &\qquad \qquad 
   +6q_xq^{*}q_x+2qq^{*}_xq_x+6qq^{*}qq^{*}q
   \bigr), 
   \label{eq:mal}
   \end{align} 
   where $q=q(t,x)$ is a $k_0\times m_0$ complex-matrix-valued function of 
   $t,x\in \RR$, and 
   $\mu_4\neq 0,\mu_2,\mu_3\in \mathbb{C}$ are constants such that 
   $\mu_2,\mu_4\in i\RR$ and $\mu_3\in \RR$. 
   By the same identification as that for  \eqref{eq:Ondq2}, 
   the equation \eqref{eq:mal} can be identified with the 
   example of
   \eqref{eq:apde} where $n=k_0m_0$. 
   Meanwhile, it does not seem that the initial value problem for \eqref{eq:mal}  
   has been investigated in \cite{Malham}. 
 \par 
 Second, it seems that time-local existence results have been expected 
 by the authors in \cite{DW2018} concerning \eqref{eq:Ondq2}. 
 Indeed, it is commented in \cite[p.190]{DW2018}
 as follows: 
 \begin{quotation}
 When $\alpha=1$ and $\beta=\gamma=0$, 
 one sees that Eqs. (62), (63) and (64) return to the standard matrix NLS 
 and matrix NLH respectively. 
 We believe that for the matrix nonlinear Schr\"odinger-like equation (62) on 
 $\mathfrak{u}(n)$, a similar property to the standard matrix NLS that one may obtained by using the geometric energy method (refer to [29, 43]) is the short time existence of solutions.
 \end{quotation}
 The equation (62) in the above quotation corresponds to \eqref{eq:Ondq2}.
 The part of the time-local existence of a solution in Theorem~\ref{theorem:lwp} 
 is not inconsistent with their belief. 
 Further, 
 Theorem~\ref{theorem:lwp} also presents the uniqueness and the continuous dependence 
 with respect to the initial data. 
   \begin{remark}
   \label{remark:gpde}
   The previous study in \cite{CO2} investigated another but similar 
   fourth-order geometric dispersive PDE for curve flows on 
   a compact K\"ahler manifold $N$, 
   showing that the initial value problem possesses 
   a time-local solution $u:[0,T]\times \RR\to N$ 
   for initial data $u_0\in C(\RR;N)$ with $\p_xu_0\in H^m(\RR;TN)$ 
   and $m\geqslant 4$, 
   where  $H^m(\RR;TN)$ is a kind of
   geometric Sobolev space on $\RR$ with the norm $\|\cdot\|_{H^m(\RR;TN)}$ 
   and $T>0$ depends on $\|\p_xu_0\|_{H^4(\RR;TN)}$. 
  The proof in \cite{CO2} was based on the 
  geometric energy method, and the 
  method seemed to be valid also for \eqref{eq:pde}.
   In view of the claim and the relation $\|\p_xu(t)\|_{H^m(\RR;TN)}=\|Q(t)\|_{H^m}$ 
   via the generalized Hasimoto transformation, 
      the author expected that a solvable structure of \eqref{eq:pde} is inherited in some sense 
      to the systems \eqref{eq:Ondq2} and \eqref{eq:2241a}
      with initial data $Q_0\in H^m(\RR;\mathbb{C}^n)$. 
      Theorem~\ref{theorem:lwp} is not inconsistent with the expectation. 
   \end{remark}
  \begin{remark}
  \label{remark:equiv}
  We should comment that Theorem~\ref{theorem:lwp} 
  still does not immediately contribute to solve the initial value 
  problem for \eqref{eq:pde}.  
 To be exact, the generalized Hasimoto transformation does not ensure the 
  equivalence of the initial value problem for \eqref{eq:pde} and 
  that for the system derived after transformation in general, 
  since constructing the inverse of the transformation 
   seems to be a nontrivial task and since the transformations 
   to derive \eqref{eq:Ondq2}, \eqref{eq:2241b}, \eqref{eq:2241a} 
   impose the existence of a fixed edge point $u(t,\infty)\in N$ or 
   $u(t,-\infty)\in N$.   
   However, the author expects that future work as a continuation of 
   this paper on \eqref{eq:apde} 
   will present an important insight on some unsolved problems for \eqref{eq:pde}, 
   such as finding conditions for time-global existence of a solution. 
      \end{remark}  
 \par 
 Third, 
 well-posedness for single fourth-order dispersive PDEs for 
 complex-valued functions 
 with nonlinearity 
 involving derivatives up to second order   
  but without nonlocal terms  
  has already been extensively studied by several authors   
  (\cite{HHW2006,HHW2007,HJ2005,HJ2007,HJ2011,RWZ,segata2003,segata2004,segata}). 
  In this direction, 
  the contribution of Theorem~\ref{theorem:lwp} seems to be rather limited. 
  To see it,  
  we review some known results related to Theorem~\ref{theorem:lwp}:  
  A series of studies by Segata (\cite{segata2003,segata2004}) 
  and Huo and Jia (\cite{HJ2005,HJ2007}) 
  established local well-posedness 
  for \eqref{eq:4shro} in $H^s(\RR)$ with $s>1/2$ 
  (and in $H^{1/2}(\RR)$ under an additional condition on coefficients of the equation), 
  which was proved by applying the smoothing effect via the Fourier restriction norm method. 
Theorem~\ref{theorem:lwp} presents local well-posedness 
  for \eqref{eq:apde} with more general nonlinearities 
  than \eqref{eq:4shro}, but imposes higher regularity on the solution. 
  Segata in \cite{segata} also showed local well-posedness of 
   \eqref{eq:4shro} in $H^s(\TT)$ with $s\geqslant 4$ by analyzing the structure 
   of the nonlinearity in more detail via the so-called modified energy method, 
   where $\TT$ is the one-dimensional flat torus and thus 
   the above smoothing effect on $\RR$ is absent.  
  Hirayama et al. in \cite[Theorem~1.3]{HIT} established local well-posedness in 
  $H^s(\RR)$ with $s\geqslant 1/2$ for a one-dimensional 
  fourth-order dispersive PDE 
  with first- and second-derivative nonlinearities, 
  by developing the method to apply the local smoothing effect, 
  which improved the class of the solution and generalizes the nonlinearities  
  handled in \cite{segata2003,segata2004,HJ2005,HJ2007}
  (see \cite[Remark~1.5.]{HIT}). 
  We note that their generalization of the nonlinearities  
  is partly different from that given by (F1)-(F3) 
  in our paper. 
  For example, restricting to the case $n=1$ where \eqref{eq:apde} is a 
      single equation, 
      the nonzero nonlinear term of $O(|\p_x^2Q|^3)$ 
    satisfies the assumption
    in \cite[Theorem~1.1]{HIT}  
    but does not satisfy (F1) in our paper. 
    On the other hand,  
    the quadratic type nonlinear term of $O(|Q||\p_xQ|)$ is not considered in \cite{HIT} but 
    satisfies (F1)-(F3)
    where $n=j=1$, $F_j^2(u,v)=O(|u||v|)$ and 
    $F_j^1\equiv F_j^3\equiv 0$.
\par
 Apart from the contribution of Theorem~\ref{theorem:lwp}, 
  it might be better to state the following  
  related results: 
 \begin{enumerate}
 \item[(i)]  
 Some multi-component systems of 
  fourth-order nonlinear dispersive PDEs 
  for complex-valued functions have been considered in \cite{GS,tarulli},  
  and local and global existence of a unique solution, and scattering properties 
  have been investigated, 
  where neither nonlinear terms involving derivatives nor 
  nonlocal terms are included in their systems. 
 \item[(ii)]  Similar nonlocal nonlinearities  have already 
  appeared in the study of well-posedness for one- or two-dimensional  
  nonlinear Schr\"odinger equations with physical background
   (e.g., \cite{CK,HO1995,NO,OYY})
  and the Davey-Stewartson system 
  (e.g., \cite{chihara1999,GiSa,hayashi1997,LP}). 
  Moreover, similar nonlinearities also appear in the study of 
  geometric dispersive PDEs for the so-called Schr\"odinger flow 
  and for the third-order analogues via the generalized Hasimoto transformation 
  (e.g., \cite{CSU,onodera0,RRS,SW2011}).
\item[(iii)]  Some single fourth-order dispersive PDEs for complex-valued 
functions with nonlinearities involving 
  derivatives up to third order have been investigated by 
   \cite{HIT,HJ2011,RWZ}, and local well-posedness 
  in Sobolev spaces (\cite{HIT,HJ2011}) 
  and global well-posedness in Sobolev spaces and 
   in modulation spaces (\cite{RWZ}) were established.  
 Moreover, multi-dimensional case has been also investigated in     
   \cite{HHW2007,HJ2011,RWZ}.   
  Meanwhile, all of these results   
  impose the smallness of the initial data, if a third-order derivative is involved in their 
  nonlinearities.  
   \end{enumerate}
\section{Idea of the proof of Theorem~\ref{theorem:lwp}}
\label{section:idea}
Our proof of Theorem~\ref{theorem:lwp} is based on the parabolic regularization 
 and the energy method combined with 
 a gauge transformation.  
 The assumption $m\geqslant 4$ on the regularity of the solution comes from 
 the requirement for the above method to work.
 More concretely, 
 a loss of derivative of order one occurs from the nonlinear terms 
 $F_j^1(Q,\p_x^2Q)$ and $F_j^2(Q,\p_xQ)$ 
 (not from $F_j^3(Q,\p_xQ)$) in general, 
 which prevents 
 the classical $L^2$-based energy method from working. 
 We overcome the difficulty by introducing a gauge transformation of the form  
 \eqref{eq:V_j}-\eqref{eq:Phi} 
 (and analogically \eqref{eq:Z_j}-\eqref{eq:Phia} and 
 \eqref{eq:Z_12}-\eqref{eq:Phi1}), 
 which is a method to bring out the local smoothing effect 
 for dispersive PDEs on $\RR$.
 Roughly speaking, 
 the gauge transformation behaves as a summation of the identity 
  and a pseudodifferential operator of order $-1$, and the commutator with 
  the fourth-order principal part of \eqref{eq:apde}  
  generates a second-order elliptic operator 
  which absorbs the loss of derivative. 
  Additionally, in the actual proof of Theorem~\ref{theorem:lwp}, 
  the gauge transformation we call here acts on images of the 
  partial differentiation $\p_x$, 
  and thus explicit pseudodifferential calculus is not required.  
We choose the strategy   
by following the idea in \cite{CO2} to solve a fourth-order 
geometric dispersive PDE for curve flows 
on a compact K\"ahler manifold. 
Indeed, the form of \eqref{eq:V_j}-\eqref{eq:Phi} looks 
extremely similar to that  in \cite[Eqn.(40)]{CO2}. 
\par 
 Notably, using the above energy method finds 
 another difficulty to show Theorem~\ref{theorem:lwp}, 
 which is due to the presence of nonlocal terms in $F_j^3(Q,\p_xQ)$.   
 More concretely, the energy method with our gauge transformation actually 
 leads to an estimate for the solution to the initial value problem 
 for the parabolic regularized equation uniformly with respect to the coefficient of the 
 added parabolic fourth-order term. 
 This ensures that the family of the parabolic regularized solutions 
 subconverges to a limit  
 weak-star in $L^{\infty}H^m$ 
 and (strongly) in $CH^{m-1}_{\text{loc}}$. 
 However, the convergence seems to be insufficient to show that 
the nonlocal terms for the regularized solution converges to those 
for the limit in the sense of distribution. 
Hence, an additional argument is required to conclude that  
the limit is actually a solution to the original problem. 
(See Remark~\ref{remark:bs} also.)
To avoid the delicate argument, 
following mainly \cite{BS,ET,II,Mietka,segata}, 
we adopt the Bona-Smith type approximation of the 
original equation to construct parabolic regularized solutions, 
which  then ensures a strong convergence to a solution in $CH^m$. 
The use of it also leads to the continuous dependence of the 
solution with respect to the initial data. 
\par 
Additionally, we point out that the well-posedness theory 
for single linear dispersive PDEs  
for complex-valued functions 
is helpful in understanding the idea of our proof and the conditions (F1)-(F3).  
To state it, consider the initial value problem 
of the form 
\begin{alignat}{2}
(\p_t-ia\p_x^4-b\p_x^3)u
  &=
  i\p_x\{\beta_1(t,x)\p_xu\}
  +
  i\p_x\{\beta_2(t,x)\overline{\p_xu}\}
\nonumber
\\
&\quad 
  +
  \gamma_1(t,x)\p_xu+\gamma_2(t,x)\overline{\p_xu}
& \quad
& \text{in}
  \quad
  \mathbb{R}\times\mathbb{R},
\label{eq:pde1}
\\
u(0,x)
  &=
  u_0(x)
& \quad
& \text{in}
  \quad
  \mathbb{R},
\label{eq:data1}
\end{alignat}
where 
$u(t,x)$ is a complex-valued unknown function of $(t,x) \in \mathbb{R}\times\mathbb{R}$, 
$a\ne0$ and $b$ are real constants, 
$\beta_1(t,x)$, $\beta_2(t,x)$, $\gamma_1(t,x)$, 
$\gamma_2(t,x) \in C\bigl(\mathbb{R};\mathscr{B}^\infty(\mathbb{R})\bigr)$ 
are complex-valued, 
$\mathscr{B}^\infty(\mathbb{R})$ 
is the set of all bounded smooth functions on $\mathbb{R}$ 
whose derivatives of any order are all bounded, 
and 
$u_0(x)$ is an initial data in $L^2$. 
Under the setting where $\beta_2(t,x)\equiv \gamma_2(t,x)\equiv 0$ 
and both 
$\beta_1(t,x)$ and $\gamma_1(t,x)$ are independent of $t$, 
the necessary and sufficient condition for $L^2$-well-posedness of 
\eqref{eq:pde1}-\eqref{eq:data1} 
was established by Mizuhara (\cite{Mizuhara})
and Tarama (\cite{Tarama}), where  
more general higher-order linear dispersive PDEs 
were also investigated.  
Restricting our concern to the application to our problem, 
we recall the fact
that 
\eqref{eq:pde1}-\eqref{eq:data1} is $L^2$-well-posed if 
$\beta_1(t,x)$ is real-valued, $\beta_2(t,x)\equiv \gamma_2(t,x)\equiv 0$, 
and there exists a function $\phi(x)\in \mathscr{B}^\infty(\mathbb{R})$ such that 
\begin{equation}
\lvert
\operatorname{Im} \gamma_1(t,x)
\rvert
\leqslant
\phi(x)
\
\text{for any}
\
(t,x)\in\mathbb{R}^2,
\quad
\int_{\mathbb{R}}
\phi(x)
dx
<\infty. 
\label{eq:addd3291}
\end{equation}
The proof  was presented in \cite[Section~2]{CO2} 
for the purpose of illustrating 
the idea of the gauge transformation in \cite[Eqn.(40)]{CO2}
to solve the fourth-order geometric dispersive 
PDE for curve flows on a compact K\"ahler manifold. 
We note that the idea 
also works for \eqref{eq:pde1}-\eqref{eq:data1} with 
$\beta_2(t,x)\not\equiv 0$ or $\gamma_2(t,x)\not\equiv 0$
under the following assumption:  
\begin{proposition}
\label{prop:linear}
Suppose that there exist functions $\phi_A(x), \phi_B(x)\in \mathscr{B}^\infty(\mathbb{R})$ such that  
\begin{align}
&\lvert
\operatorname{Im}\beta_1(t,x)
\rvert
+
\lvert
\beta_2(t,x)
\rvert
\leqslant
\phi_A(x)
\
\text{for any}
\
(t,x)\in\mathbb{R}^2, 
\ \ 
\int_{\mathbb{R}}
\phi_A(x)
\,dx
<\infty,
\label{eq:addd3292}
\\
&
\lvert
\operatorname{Im} \gamma_1(t,x)
\rvert
\leqslant
\phi_B(x)
\
\text{for any}
\
(t,x)\in\mathbb{R}^2,
\ \ 
\int_{\mathbb{R}}
|\phi_B(x)|^2
dx
<\infty.
\label{eq:addd3293}
\end{align}
Then \eqref{eq:pde1}-\eqref{eq:data1} 
is $L^2$-well-posed, that is, for any $u_0\in L^2(\RR;\mathbb{C})$, 
\eqref{eq:pde1}-\eqref{eq:data1} has a unique solution 
$u\in C(\RR;L^2(\RR;\mathbb{C}))$.
\end{proposition}  
The proof is stated in Appendix. 
Although the assumptions 
\eqref{eq:addd3292}-\eqref{eq:addd3293} are still loose 
from the viewpoint of the theory for linear dispersive PDEs,  
they are informative enough to find the way to solve the initial value problem  
for \eqref{eq:apde} including 
\eqref{eq:4shro}, \eqref{eq:WZY}, 
\eqref{eq:Ondq2}, 
\eqref{eq:mal}.  
In fact, we arrived at the conditions (F1) and (F2) 
by observing the above fact 
and the equation in \cite{CO2}((Eqn.(41) with $\ep=0$) 
satisfied by higher-order covariant derivatives of the curve flow, 
and our choice of the gauge transformation to prove  Theorem~\ref{theorem:lwp} is also motivated from that 
of Proposition~\ref{prop:linear} 
where $|Q|^2$ and $|Q|$ play roles as  
$\phi_A$ and $\phi_B$ respectively. 
%
 %
  %
 \section{Uniform estimate for Bona-Smith regularized solutions 
 in $L^{\infty}H^m$.}
  \label{section:local} 
This section aims to obtain uniform estimates for solutions to 
an initial value problem regularized by the 
Bona-Smith approximation.   
Throughout this section, $m$ is supposed be an integer satisfying $m\geqslant 4$.
\par 
To begin with, 
following mainly \cite{BS,ET,II,Mietka,segata},  
we recall the setting of the Bona-smith approximation:
Let $\phi:\RR\to\RR$  be a Schwartz function 
satisfying $0\leqslant \phi(x)\leqslant 1$ on $\RR$ and  
$\phi(x)=1$ on a neighborhood of the origin $x=0$ so that 
$\p_x^k\phi(0)=0$ for all positive integers $k$.
For any $Q_0={}^t(Q_{01},\ldots,Q_{0n})\in H^m(\RR;\mathbb{C}^n)$ and $\ep\in (0,1)$, 
define $Q_0^{\ep}:\RR\to \mathbb{C}^n$ by
$$
\widehat{Q_0^{\ep}}(\xi)
=
\phi(\ep \xi)
\widehat{Q_0}(\xi) 
\quad
(\xi\in \RR), 
$$ 
where $\widehat{Q_0^{\ep}}$ and $\widehat{Q_0}$ denote the 
Fourier transform of $Q_0^{\ep}$ and $Q_0$ respectively.
It follows that $Q_0^{\ep}\in H^{\infty}(\RR;\mathbb{C}^n)$ and 
$Q_0^{\ep}\to Q_0$ in $H^m(\RR;\mathbb{C}^n)$ as $\ep \downarrow 0$. 
Moreover, 
\begin{align}
&\|Q_0^{\ep}\|_{H^m}
\leqslant 
\|Q_0\|_{H^m}, 
\label{eq:bs1}
\\
&\|Q_0^{\ep}\|_{H^{m+\ell}}
\leqslant 
C\ep^{-\ell}
\|Q_0\|_{H^m}
\quad 
(\ell=0,1,2,\ldots),
\label{eq:6092}
\\
&\|Q_0^{\ep}-Q_0\|_{H^{m-\ell}}
\leqslant 
C\ep^{\ell}
\|Q_0\|_{H^m}
\quad 
(\ell=0,1,2,\ldots), 
\label{eq:6093}
\end{align} 
where $C$ is a positive 
constant which depends on $m,k,\phi$, 
but not on $\ep$. 
The set $\left\{Q_{0}^{\ep}\right\}_{\ep\in (0,1)}$ is called a 
Bona-Smith approximation of $Q_0$.   
\par 
For $\ep\in (0,1)$ and $Q_0\in H^m(\RR;\mathbb{C}^n)$,  
we consider the initial value problem for 
the fourth-order parabolic regularized system:   
\begin{alignat}{2}
 \left(
 \p_t+\ep^5 \p_x^4-iM_a\p_x^4-M_b\p_x^3-iM_{\lambda}\p_x^2
 \right)
 Q
  &=F(Q, \p_xQ, \p_x^2Q), 
 \label{eq:bpde}
 \\
  Q(0,x)
   &=
   Q_{0}^{\ep}(x) 
 \label{eq:bdata}
 \end{alignat}
for $Q={}^t(Q_1,\ldots,Q_n):[0,\infty)\times \RR\to \mathbb{C}^n$, 
where $ Q_{0}^{\ep}\in H^{\infty}(\RR;\mathbb{C}^n)$ is given 
by the Bona-Smith approximation of $Q_0$. 
It is not difficult to show that there exists a time 
$T_{\ep}=T(\ep, \|Q_0^{\ep}\|_{H^m})>0$ 
and a unique solution $Q^{\ep}={}^t(Q^{\ep}_1,\ldots,Q^{\ep}_n)\in C([0,T_{\ep}];H^{\infty}(\RR;\mathbb{C}^n))$ to \eqref{eq:bpde}-\eqref{eq:bdata} 
by the standard contraction mapping argument. 
We omit the detail. 
\par 
The goal of this section is to show the following:
  \begin{proposition}
  \label{proposition:preloc}
  Let $m$ be an integer with $m\geqslant 4$. 
  For any $Q_0\in H^m(\RR;\mathbb{C}^n)$, 
  let $\left\{Q^{\ep}\right\}_{\ep\in (0,1)}$ be the 
  family of solutions to \eqref{eq:bpde}-\eqref{eq:bdata}.
  Then, there exists a time $T=T(\|Q_0\|_{H^4})>0$ which 
  is independent of $\ep\in (0,1)$ such that $\left\{Q^{\ep}\right\}_{\ep\in (0,1)}$ is bounded in  $L^{\infty}(0,T;H^m(\RR;\mathbb{C}^n))$.
  \end{proposition}
\begin{proof}[Proof of Proposition~\ref{proposition:preloc}]
For the integer $m\geqslant 4$ and fixed $\ep\in (0,1)$, 
we consider the estimate for 
$\mathcal{E}_m(Q^{\ep})=\mathcal{E}_m(Q^{\ep}(t)):[0,T_{\ep}]\to [0,\infty)$, 
 the square of which is defined by 
\begin{align}
\mathcal{E}_m(Q^{\ep}(t))^2
&:=
\|V^{\ep}(t)\|_{L^2}^2+\|Q^{\ep}(t)\|_{H^{m-1}}^2.
\label{eq:Nml}
\end{align}
Here, $V^{\ep}={}^t(V^{\ep}_1,\ldots,V^{\ep}_n)$ 
is a $\mathbb{C}^n$-valued function defined by
\begin{align}
 &V^{\ep}_j=V^{\ep}_j(t,x)
 :=\p_x^mQ^{\ep}_j(t,x)+\dfrac{L}{4a_j}\Phi^{\ep}(t,x)i\p_x^{m-1}Q^{\ep}_j(t,x)
 \quad (j\in \left\{1,\ldots,n\right\}), 
 \label{eq:V_j}
\\
 &\Phi^{\ep}=\Phi^{\ep}(t,x)
 :=\int_{-\infty}^{x}g(Q^{\ep}(t,y))\,dy
 \left(
 =\int_{-\infty}^{x}|Q^{\ep}(t,y))|^2\,dy
 \right), 
 \label{eq:Phi}
 \end{align} 
 where $L>1$ is a sufficiently large constant which will be decided later 
 independently of $j$ and $\ep$. 
 Moreover, we set $\mathcal{E}_m(Q_0^{\ep})=\mathcal{E}_m(Q^{\ep}(0))$. 
\vspace{0.5em}
\\
\underline{Equivalence of $\mathcal{E}_m(Q^{\ep}(t))$ and $\|Q^{\ep}(t)\|_{H^m}$
on a restricted 
time-interval:}
\\ By definition \eqref{eq:Nml} with \eqref{eq:V_j}-\eqref{eq:Phi} 
and 
$$
|\Phi^{\ep}(t,x)|
\leqslant \int_{\RR}|Q^{\ep}(t,y))|^2\,dy
=\|Q^{\ep}(t)\|_{L^2}^2, 
$$ 
there exist constants 
$C_{1,L}, C_{2,L}>0$ which depend on $L$ and not on $\ep$ and $m$ 
such that 
\begin{align}
&\frac{\|Q^{\ep}(t)\|_{H^m}^2}{C_{1,L}
\left(1+\|Q^{\ep}(t)\|_{L^2}^4\right)}
\leqslant 
\mathcal{E}_m(Q^{\ep}(t))^2
\leqslant 
C_{2,L}
\left(1+
\|Q^{\ep}(t)\|_{L^2}^4
\right)\|Q^{\ep}(t)\|_{H^m}^2
\label{eq:H3Em}
\end{align}
for all $t\in [0,T_{\ep}]$. 
We set  
\begin{align}
T_{\ep}^{\star}
&=\sup
\left\{
T>0
\mid 
\mathcal{E}_4(Q^{\ep}(t))
\leqslant 
2\mathcal{E}_4(Q_0^{\ep})
\ \ 
\text{for all}
\ \  
t\in [0,T]
\right\}.
\label{eq:cuttime}
\end{align}
By \eqref{eq:Nml} and \eqref{eq:H3Em} for $m=4$ 
and \eqref{eq:cuttime}, 
\begin{align}
\|Q^{\ep}(t)\|_{H^4}^2
&\leqslant 
C_{1,L}(1+
\|Q^{\ep}(t)\|_{L^2}^4
)
\mathcal{E}_4(Q^{\ep}(t))^2
\quad 
(\because \eqref{eq:H3Em})
\nonumber
\\
&\leqslant 
C_{1,L}
\left\{
1+
\mathcal{E}_4(Q^{\ep}(t))^4
\right\}
\mathcal{E}_4(Q^{\ep}(t))^2
\quad 
(\because \eqref{eq:Nml})
\nonumber
\\
&\leqslant 
4C_{1,L}
\left\{
1+
16\,\mathcal{E}_4(Q_0^{\ep})^4
\right\}
\mathcal{E}_4(Q_0^{\ep})^2
\quad 
(\because \eqref{eq:cuttime}
\label{eq:H41}
\end{align}  
for $t\in [0,T_{\ep}^{\star}]$. 
In addition,  \eqref{eq:H3Em} for $t=0$ and \eqref{eq:bs1} for $m=4$ imply
\begin{align}
\mathcal{E}_4(Q_0^{\ep})^2
&\leqslant
C_{2,L}\left(1+\|Q_0^{\ep}\|_{H^4}^4\right) \|Q_0^{\ep}\|_{H^4}^2
\leqslant 
C_{2,L}\left\{
1+\|Q_0\|_{H^4}^4
\right\} 
\|Q_0\|_{H^4}^2.
\label{eq:H3F4}
\end{align}
Combining \eqref{eq:H41} and \eqref{eq:H3F4}, 
we have 
\begin{align}
\sup_{t\in [0,T_{\ep}^{\star}]}\|Q^{\ep}(t)\|_{H^4}^2
&\leqslant 
P^0_L(\|Q_0\|_{H^4})  
\label{eq:H3Emn}
\end{align}
and 
\begin{align}
&\frac{\|Q^{\ep}(t)\|_{H^m}^2}{P_{L}^1(\|Q_0\|_{H^4})}
\leqslant 
\mathcal{E}_m(Q^{\ep}(t))^2
\leqslant 
P_{L}^2(\|Q_0\|_{H^4})\|Q^{\ep}(t)\|_{H^m}^2
\label{eq:H3Emm}
\end{align}
for any $t\in [0,T_{\ep}^{\star}]$.
Here, each of $P_L^k(\cdot):[0,\infty)\to (0,\infty)$ for $k=0,1,2$ denotes a positive-valued 
increasing function on $[0,\infty)$ depending on $L$ but not on $m$ and $\ep$. 
\vspace{0.5em}
\\
\underline{The equation satisfied by $V_j^{\ep}$ and 
uniform estimate for $\left\{Q^{\ep}\right\}_{\ep\in (0,1)}$:}
\\
We set $U^{\ep}=\p_x^mQ^{\ep}$, i.e., 
$U^{\ep}={}^t(U^{\ep}_1,\ldots,U^{\ep}_n)$ and 
$U^{\ep}_j(t,x)=\p_x^mQ^{\ep}_j(t,x)$ for $j\in \left\{1,\ldots,n\right\}$. 
In this part, we 
investigate the equation satisfied by $V_j^{\ep}$ 
and then derive estimates for $\left\{Q^{\ep}(t)\right\}_{\ep\in (0,1)}$ in $H^m(\RR)$ 
on a restricted time-interval. 
Hereafter in this part, we use 
$A_m^k(\cdot): [0,\infty)\to (0,\infty)$ 
and $A_{L,m}^k(\cdot): [0,\infty)\to (0,\infty)$ for some integer $k$ 
to denote a positive-valued 
 increasing function which depends on $m$
 but not on $\ep$. 
We use the latter only if the increasing function depends also on $L$.  
\par
Applying $\p_x^m$ to the $j$-th component of \eqref{eq:bpde}, 
we have  
\begin{align}
 &\left\{
  \p_t+(\ep^5 -i\, a_j)\p_x^4-b_j\p_x^3-i \lambda_j \p_x^2
  \right\}
  U_j^{\ep}
  \nonumber
  \\
  &=
  \p_x^m(F_j^1(Q^{\ep},\p_x^2Q^{\ep}))
  +\p_x^m(F_j^2(Q^{\ep},\p_xQ^{\ep}))
  +\p_x^m(F_j^3(Q^{\ep},\p_xQ^{\ep})).
\label{eq:Ujm}
\end{align}
We compute the three terms of the right hand side of \eqref{eq:Ujm} separately. 
\par 
First, recalling 
$F_j^1(u,w)=O\left(
g(u)|w|
\right)$ follows from  \eqref{eq:F1} in (F1) and \eqref{eq:gz}, 
we use the Leibniz rule, 
the Gagliardo-Nirenberg inequality, 
and the Sobolev embedding  to deduce  
\begin{align}
& \p_x^m(F_j^1(Q^{\ep},\p_x^2Q^{\ep}))
\nonumber
\\  
&=
  O(g(Q^{\ep})|\p_x^2U^{\ep}|)
  +O(|\p_x\left\{g(Q^{\ep})\right\}| |\p_xU^{\ep}|)
  +r^1_{\leqslant m}
  \nonumber
  \\
 &=
 O(g(Q^{\ep})|\p_x^2U^{\ep}|)
 +O\left(|\p_xQ^{\ep}|
 |Q^{\ep}||\p_xU^{\ep}|\right)
 +r^1_{\leqslant m}
 \nonumber
 \\
 &=
 O(g(Q^{\ep})|\p_x^2U^{\ep}|)
  +O\left(\|Q^{\ep}(t)\|_{H^2}
 |Q^{\ep}||\p_xU^{\ep}|\right)
  +r^1_{\leqslant m}, 
  \label{eq:F1m}
 \end{align}
 where 
 \begin{align}
& \| r^1_{\leqslant m}(t)\|_{L^2}
\leqslant 
A_{m}^1(\|Q^{\ep}(t)\|_{H^3})\|Q^{\ep}(t)\|_{H^m}. 
\label{eq:re1}
 \end{align}
More precisely, the first term of the right hand side of \eqref{eq:F1m} satisfies 
 \begin{align}
 |O(g(Q^{\ep})|\p_x^2U^{\ep}|)|
 \leqslant \sum_{j=1}^{n}c_j^1\,g(Q^{\ep})|\p_x^2U^{\ep}|, 
 \label{eq:F1ms}
 \end{align}
 where the constants $c_{j}^1$ for $j\in \left\{1,\ldots,n\right\}$ 
 come from (F1) and are 
 independent of $L$. 
 \par
Second, since 
$$
F_j^2(Q^{\ep},\p_xQ^{\ep})
=
O\left(
\sum_{p_1=0}^{d_1}
\sum_{p_2=0}^{d_2}
|Q^{\ep}|^{1+p_1}|\p_xQ^{\ep}|^{p_2}
\right)
$$
follows from \eqref{eq:F2} in (F2),     
a similar computation using the Leibniz rule, the Sobolev embedding, 
and the  Gagliardo-Nirenberg inequality shows the following:  
If $d_2=0$, then 
\begin{align}
\|
\p_x^m\left(
F_j^2(Q^{\ep},\p_xQ^{\ep})
\right)
(t)
\|_{L^2}
&\leqslant 
 A_{m}^2(\|Q^{\ep}(t)\|_{H^1})\|Q^{\ep}(t)\|_{H^m}. 
\label{eq:F2m2}
  \end{align}
If $d_2\geqslant 1$, then 
\begin{align}
&\p_x^m\left(
F_j^2(Q^{\ep},\p_xQ^{\ep})
\right)
\nonumber
\\&=
O\left(
\sum_{p_1=0}^{d_1}
\sum_{p_2=1}^{d_2}
|Q^{\ep}|^{1+p_1}|\p_xQ^{\ep}|^{p_2-1}|\p_x^{m+1}Q^{\ep}|
\right)
+
r^{2}_{\leqslant m}
\nonumber
\\
&=
O\left(
\sum_{p_1=0}^{d_1}
\sum_{p_2=1}^{d_2}
\|Q^{\ep}(t)\|_{H^2}^{p_1+p_2-1}
|Q^{\ep}||\p_x^{m+1}Q^{\ep}|
\right)
+
r^2_{\leqslant m}
\label{eq:F2m1}
\end{align}
where 
\begin{align}
 & \| r^2_{\leqslant m}(t)\|_{L^2}
 \leqslant 
 A_{m}^3(\|Q^{\ep}(t)\|_{H^3})\|Q^{\ep}(t)\|_{H^m}.
\label{eq:re2}
  \end{align}
\par 
Third, it follows that 
\begin{align}
\|
\p_x^m\left(
F_j^3(Q^{\ep},\p_xQ^{\ep})
\right)
(t)
\|_{L^2}
&\leqslant 
 A_{m}^4(\|Q^{\ep}(t)\|_{H^2})\|Q^{\ep}(t)\|_{H^m}.
  \label{eq:F3m}
 \end{align} 
Although it may not be difficult to see the fact, 
we confirm it here to share how to handle 
nonlocal terms.  
For this purpose, 
 we begin with using the Leibniz rule to see
   \begin{align}
  &\p_x^m(F_j^3(Q^{\ep},\p_xQ^{\ep})) 
  \nonumber
  \\
  &=
  \sum_{r=1}^n
  \left(
  \int_{-\infty}^x
  F_{j,r}^{3,A}(Q^{\ep},\p_xQ^{\ep})(t,y)dy
  \right)
  \p_x^{m}(F_{j,r}^{3,B}(Q^{\ep})) 
  \nonumber
  \\
  &\quad
  +
  \sum_{r=1}^n
   \sum_{k=1}^m
   \frac{m!}{k!(m-k)!}
  \left(
  \p_x^k
  \int_{-\infty}^x
  F_{j,r}^{3,A}(Q^{\ep},\p_xQ^{\ep})(t,y)dy
  \right)
  \p_x^{m-k}(F_{j,r}^{3,B}(Q^{\ep})) 
  \nonumber
  \\
  &=
  \sum_{r=1}^n
  \left(
  \int_{-\infty}^x
  F_{j,r}^{3,A}(Q^{\ep},\p_xQ^{\ep})(t,y)dy
  \right)
  \p_x^{m}(F_{j,r}^{3,B}(Q^{\ep})) 
  \nonumber
  \\
  &\quad
  +
  \sum_{r=1}^n
   \sum_{k=1}^m
   \frac{m!}{k!(m-k)!}
  \p_x^{k-1}
  (F_{j,r}^{3,A}(Q^{\ep},\p_xQ^{\ep}))\, 
  \p_x^{m-k}(F_{j,r}^{3,B}(Q^{\ep})). 
  \nonumber
   \end{align}
  Here, by the H\"older inequality and the Sobolev embedding,   
   \begin{align}
    &\left\|
    \left(
    \int_{-\infty}^{\cdot}
    F_{j,r}^{3,A}(Q^{\ep},\p_xQ^{\ep})(t,y)dy
    \right)
    \p_x^{m}(F_{j,r}^{3,B}(Q^{\ep}))(t)
    \right\|_{L^2}
    \nonumber
    \\
    &\leqslant 
    \left\|
     \int_{-\infty}^{\cdot}
     F_{j,r}^{3,A}(Q^{\ep},\p_xQ^{\ep})(t,y)dy
     \right\|_{L^{\infty}}
     \left\|
      \p_x^{m}(F_{j,r}^{3,B}(Q^{\ep}))(t)
      \right\|_{L^2}
   \nonumber
   \\
   &\leqslant 
   \left\|
     F_{j,r}^{3,A}(Q^{\ep},\p_xQ^{\ep})(t)
     \right\|_{L^{1}}
     \left\|
      \p_x^{m}(F_{j,r}^{3,B}(Q^{\ep}))(t)
      \right\|_{L^2}
   \nonumber
    \end{align}
    for $j,r\in \left\{1,\ldots,n\right\}$, 
    and  
    \begin{align}
    &\left\|
    \left(
    \p_x^{k-1}
    (F_{j,r}^{3,A}(Q^{\ep},\p_xQ^{\ep}))\, 
    \p_x^{m-k}(F_{j,r}^{3,B}(Q^{\ep}))
    \right)(t)
    \right\|_{L^2}
    \nonumber
    \\
    &\leqslant 
    \left\|
     \p_x^{k-1}
     (F_{j,r}^{3,A}(Q^{\ep},\p_xQ^{\ep}))(t)
     \right\|_{L^2}
     \left\|
      \p_x^{m-k}(F_{j,r}^{3,B}(Q^{\ep}))(t)
      \right\|_{L^{\infty}}
   \nonumber
   \\
   &\leqslant 
   \left\|
     \p_x^{k-1}
     (F_{j,r}^{3,A}(Q^{\ep},\p_xQ^{\ep}))(t)
     \right\|_{L^2}
     \left\|
      \p_x^{m-k}(F_{j,r}^{3,B}(Q^{\ep}))(t)
      \right\|_{H^1}
   \nonumber 
    \end{align}
   for  $j,r\in \left\{1,\ldots,n\right\}$ and 
   $k\in \left\{1,\ldots,m\right\}$.  
   Furthermore, recalling 
   \begin{align}
   F_{j,r}^{3,A}(Q^{\ep},\p_xQ^{\ep})
   &=
   O\left(
  \sum_{p_3=0}^{d_3} |Q^{\ep}|^{2+p_3}
  +\sum_{p_4=0}^{d_4}|\p_xQ^{\ep}|^{2+p_4}
   \right), 
   \nonumber
   \\
   F_{j,r}^{3,B}(Q^{\ep})
   &=
   O\left(
   \sum_{p_5=0}^{d_5} |Q^{\ep}|^{1+p_5}
   \right), 
   \nonumber
   \end{align}
   which follow from
  \eqref{eq:F31} and \eqref{eq:F32} in  (F3),    
  we estimate in the same way as above to deduce 
   \begin{align}
   &\left\|
      F_{j,r}^{3,A}(Q^{\ep},\p_xQ^{\ep})(t)
      \right\|_{L^{1}}
   \nonumber
   \\
   &\leqslant
   C\left(
   \sum_{p_3=0}^{d_3}
   \|Q^{\ep}(t)\|_{L^{\infty}}^{p_3}
   +
   \sum_{p_4=0}^{d_4}
   \|\p_xQ^{\ep}(t)\|_{L^{\infty}}^{p_4}
   \right)
   \left(
   \|Q^{\ep}(t)\|_{L^{2}}^{2}
    +
    \|\p_xQ^{\ep}(t)\|_{L^{2}}^{2}
   \right)
   \nonumber
   \\
   &\leqslant 
   C\sum_{\ell=2}^{d_3+d_4+2}\|Q^{\ep}(t)\|_{H^2}^{\ell}, 
   \label{eq:4141}
  \\
  & \left\|
       \p_x^{m}(F_{j,r}^{3,B}(Q^{\ep}))(t)
       \right\|_{L^2}
  \leqslant 
   A_m^5(\|Q^{\ep}(t)\|_{H^1})
   \|Q^{\ep}(t)\|_{H^m}, 
  \label{eq:4142}
  \\
  &\left\|
      \p_x^{k-1}
      (F_{j,r}^{3,A}(Q^{\ep},\p_xQ^{\ep}))(t)
      \right\|_{L^2}
  \leqslant  
  A_m^6(\|Q^{\ep}(t)\|_{H^2})
   \|Q^{\ep}(t)\|_{H^k},
  \label{eq:4143}
  \\
  &\left\|
      \p_x^{m-k}(F_{j,r}^{3,B}(Q^{\ep}))(t)
      \right\|_{H^1}
  \leqslant  
  A_m^7(\|Q^{\ep}(t)\|_{H^1})
   \|Q^{\ep}(t)\|_{H^{m-k+1}}.
  \label{eq:4144}
   \end{align}
  It follows from \eqref{eq:4141} and \eqref{eq:4142}  
   \begin{align}
   &\left\|
   \left(
   \int_{-\infty}^{\cdot}
   F_{j,r}^{3,A}(Q^{\ep},\p_xQ^{\ep})(t,y)dy
   \right)
   \p_x^{m}(F_{j,r}^{3,B}(Q^{\ep}))(t)
   \right\|_{L^2}
   \nonumber
   \\
   &\leqslant 
      A_m^8(\|Q^{\ep}(t)\|_{H^2})
       \left\|
        Q^{\ep}(t)
        \right\|_{H^m}.
        \label{eq:233291}
   \end{align}
  It follows from \eqref{eq:4143} and \eqref{eq:4144} 
   \begin{align}
   &\left\|
   \left(
   \p_x^{k-1}
   (F_{j,r}^{3,A}(Q^{\ep},\p_xQ^{\ep}))\, 
   \p_x^{m-k}(F_{j,r}^{3,B}(Q^{\ep}))
   \right)(t)
   \right\|_{L^2}
   \nonumber
   \\
   &\leqslant 
      A_m^9(\|Q^{\ep}(t)\|_{H^2})
          \|Q^{\ep}(t)\|_{H^k} 
          \|Q^{\ep}(t)\|_{H^{m-k+1}}.
  \nonumber
   \end{align} 
  Here, for each $k\in \left\{1,\ldots,m\right\}$,  the Gagliardo-Nirenberg inequality implies 
   \begin{align}
  &\|Q^{\ep}(t)\|_{H^k} 
   \|Q^{\ep}(t)\|_{H^{m-k+1}}
  \nonumber
  \\
  &\leqslant C
   \|Q^{\ep}(t)\|_{H^1}^{\frac{m-k}{m-1}} 
    \|Q^{\ep}(t)\|_{H^{m}}^{\frac{k-1}{m-1}} 
    \|Q^{\ep}(t)\|_{H^1}^{\frac{k-1}{m-1}} 
     \|Q^{\ep}(t)\|_{H^{m}}^{\frac{m-k}{m-1}}
  \nonumber
  \\
  &=C
   \|Q^{\ep}(t)\|_{H^1} 
    \|Q^{\ep}(t)\|_{H^{m}}.   
   \nonumber
  \end{align}
  Hence, we obtain 
  \begin{align}
   &\left\|
   \left(
   \p_x^{k-1}
   (F_{j,r}^{3,A}(Q^{\ep},\p_xQ^{\ep}))\, 
   \p_x^{m-k}(F_{j,r}^{3,B}(Q^{\ep}))
   \right)(t)
   \right\|_{L^2}
   \nonumber
   \\
   &
  \leqslant 
   A_m^{10}(\|Q^{\ep}(t)\|_{H^2})
   \|Q^{\ep}(t)\|_{H^m}.
   \label{eq:233292}
   \end{align}
   The desired estimate \eqref{eq:F3m} immediately follows from 
  \eqref{eq:233291} and \eqref{eq:233292}.
 \par 
 Combining 
 \eqref{eq:Ujm}, \eqref{eq:F1m}(with \eqref{eq:re1}), 
 \eqref{eq:F2m2} or \eqref{eq:F2m1} (with \eqref{eq:re2}) , 
 and \eqref{eq:F3m}, we have
 \begin{align}
  &\left\{
   \p_t+(\ep^5 -i\, a_j)\p_x^4-b_j\p_x^3-i\lambda_j \p_x^2
   \right\}
   U_j^{\ep}
   \nonumber
   \\
   &=
 O(g(Q^{\ep})|\p_x^2U^{\ep}|)
   +O(A_m^{11}(\|Q^{\ep}(t)\|_{H^2})|Q^{\ep}||\p_xU^{\ep}|)
   +r^3_{\leqslant m},  
 \label{eq:414r1}
 \end{align}
 where 
 \begin{align}
  & \| r^3_{\leqslant m}(t)\|_{L^2}
  \leqslant 
  A_{m}^{12}(\|Q^{\ep}(t)\|_{H^3})\|Q^{\ep}(t)\|_{H^m}.
 \label{eq:414r2}
   \end{align}
 \par
 We next compute 
 \begin{align}
 &
 \p_t\left(
 \dfrac{L}{4a_j}\Phi^{\ep}i\p_x^{m-1}Q^{\ep}_j
 \right)
 =
 \dfrac{L}{4a_j}\Phi^{\ep}i\p_t\p_x^{m-1}Q^{\ep}_j
 +
 \dfrac{L}{4a_j}(\p_t\Phi^{\ep})i\p_x^{m-1}Q^{\ep}_j.
 \label{eq:27}
 \end{align}
 By the almost same computation to obtain \eqref{eq:Ujm} 
 with \eqref{eq:F1m}-\eqref{eq:F3m}, 
 we have 
 \begin{align}
 \p_t\p_x^{m-1}Q_j^{\ep}
 &=
\left\{
 (-\ep^5 +ia_j)\p_x^4+b_j\p_x^3+i\lambda_j\p_x^2
 \right\}
 \p_x^{m-1}Q_j^{\ep}
 \nonumber
 \\
 &\quad 
 +O(g(Q^{\ep})|\p_xU^{\ep}|)
 +r^4_{\leqslant m}
 \nonumber
 \\
 &=
 \left\{
  (-\ep^5 +ia_j)\p_x^4+b_j\p_x^3+i\lambda_j\p_x^2
  \right\}
  \p_x^{m-1}Q_j^{\ep}
 \nonumber
 \\
 &\quad
  +O(\|Q^{\ep}(t)\|_{H^1}|Q^{\ep}||\p_xU^{\ep}|)
  +r^4_{\leqslant m},
 \label{eq:271}
\end{align}
where 
\begin{align}
 \| r^4_{\leqslant m}(t)\|_{L^2}
 &\leqslant 
 A_{m}^{13}(\|Q^{\ep}(t)\|_{H^2})\|Q^{\ep}(t)\|_{H^m}.
 \nonumber
  \end{align}
  Recalling $|\Phi^{\ep}(t,x)|\leqslant \|Q^{\ep}(t)\|_{L^2}^2$, 
  and substituting \eqref{eq:271} into the 
  first term of the right hand side of \eqref{eq:27}, we see
 \begin{align}
 &\dfrac{L}{4a_j}\Phi^{\ep}i\p_t\p_x^{m-1}Q^{\ep}_j
 \nonumber
 \\
 &=
 \left\{(-\ep^5+ia_j)\p_x^4+b_j\p_x^3+i\lambda_j\p_x^2\right\}
 \left(
 \dfrac{L}{4a_j}\Phi^{\ep}i\p_x^{m-1}Q^{\ep}_j
 \right)
 \nonumber
 \\
 &\quad
 -(-\ep^5+ia_j)\sum_{k=1}^4\frac{4!}{k!(4-k)!}(\p_x^k\Phi^{\ep})
 \dfrac{L}{4a_j}i\p_x^{4-k+m-1}Q^{\ep}_j
 \nonumber
  \\
  &\quad
  -b_j\sum_{k=1}^3\frac{3!}{k!(3-k)!}(\p_x^k\Phi^{\ep})
   \dfrac{L}{4a_j}i\p_x^{3-k+m-1}Q^{\ep}_j
 \nonumber
 \\
 &\quad
 -i\lambda_j\sum_{k=1}^2\frac{2!}{k!(2-k)!}(\p_x^k\Phi^{\ep})
  \dfrac{L}{4a_j}i\p_x^{2-k+m-1}Q^{\ep}_j
  \nonumber
  \\
  &\quad
  +\dfrac{L}{4a_j}O\left(A_m^{14}(\|Q^{\ep}(t)\|_{H^1}) 
  |Q^{\ep}||\p_xU^{\ep}|\right)
  +r_{\leqslant m}^5, 
  \nonumber
  \end{align}
where 
\begin{align}
  \| r^5_{\leqslant m}(t)\|_{L^2}
   &\leqslant 
   A_{L,m}^1(\|Q^{\ep}(t)\|_{H^2})\|Q^{\ep}(t)\|_{H^m}. 
   \nonumber
 \end{align}
   Moreover, since $\p_x\Phi^{\ep}=g(Q^{\ep})=|Q^{\ep}|^2$, it follows that
 \begin{align}
 &\sum_{k=1}^4\frac{4!}{k!(4-k)!}(\p_x^k\Phi^{\ep})
  \dfrac{L}{4a_j}i\p_x^{4-k+m-1}Q^{\ep}_j
  \nonumber
  \\
  &=
  \frac{Li}{a_j}g(Q^{\ep})\p_x^{m+2}Q_j^{\ep}
  +
  \frac{4!}{2!2!}\frac{Li}{4a_j}\p_x\left\{g(Q^{\ep})\right\}\p_x^{m+1}Q_j^{\ep}
  +\cdots
  \nonumber
  \\
  &=
  \frac{Li}{a_j}\p_x\left\{g(Q^{\ep})\p_xU_j^{\ep}\right\}
  +\frac{Li}{2 a_j}\p_x\left\{g(Q^{\ep})\right\}\p_xU_j^{\ep}
  +r_{\leqslant m}^6
   \nonumber
   \\
   &=
  \frac{Li}{a_j}\p_x\left\{g(Q^{\ep})\p_xU_j^{\ep}\right\}
   +\frac{Li}{a_j}O\left(A_m^{15}(\|Q^{\ep}(t)\|_{H^2})|Q^{\ep}||\p_xU^{\ep}|\right)
   +r_{\leqslant m}^6,  
 \nonumber
 \end{align}
 where 
 \begin{align}
   &\| r^6_{\leqslant m}(t)\|_{L^2}
    \leqslant 
    A_{L,m}^2(\|Q^{\ep}(t)\|_{H^3})\|Q^{\ep}(t)\|_{H^m}.
    \nonumber
  \end{align}
  In the same way as above, 
  \begin{align}
   &\sum_{k=1}^3\frac{3!}{k!(3-k)!}(\p_x^k\Phi^{\ep})
    \dfrac{L}{4a_j}i\p_x^{3-k+m-1}Q^{\ep}_j
    \nonumber
    \\
    &=
\frac{Li}{a_j}O\left(\|Q^{\ep}(t)\|_{H^1}|Q^{\ep}||\p_xU^{\ep}|\right)
     +r_{\leqslant m}^7,  
   \nonumber
   \end{align}
   where 
   \begin{align}
     &\| r^7_{\leqslant m}(t)\|_{L^2}
      \leqslant 
      A_{L,m}^3(\|Q^{\ep}(t)\|_{H^3})\|Q^{\ep}(t)\|_{H^m}.
      \nonumber
    \end{align}
Noting them, we have
\begin{align}
&\dfrac{L}{4a_j}\Phi^{\ep}i\p_t\p_x^{m-1}Q^{\ep}_j
\nonumber
\\
&=
 \left\{(-\ep^5+ia_j)\p_x^4+b_j\p_x^3+i\lambda_j\p_x^2\right\}
 \left(
 \dfrac{L}{4a_j}\Phi^{\ep}i\p_x^{m-1}Q^{\ep}_j
 \right)
 \nonumber
 \\
 &\quad 
+\left(
\frac{\ep^5 i}{a_j}+1
\right)L\p_x\left\{g(Q^{\ep})\p_xU_j^{\ep}\right\}  
+
O(A_{L,m}^4(\|Q^{\ep}(t)\|_{H^2})|Q^{\ep}||\p_xU^{\ep}|)
\nonumber
\\
&\quad +r_{\leqslant m}^8, 
\label{eq:91}
\end{align}
where 
\begin{align}
\| r^8_{\leqslant m}(t)\|_{L^2}
    \leqslant 
    A_{L,m}^5(\|Q^{\ep}(t)\|_{H^3})\|Q^{\ep}(t)\|_{H^m}.
    \nonumber
\end{align}
\\
On the other hand, noting $Q^{\ep}\in C([0,T_{\ep}]; H^{\infty}(\RR;\mathbb{C}^n))$ 
and $\ep\in (0,1)$, 
we use \eqref{eq:bpde} to deduce 
\begin{align}
\left|
\dfrac{L}{4a_j}(\p_t\Phi^{\ep})i\p_x^{m-1}Q^{\ep}_j
\right|
&=
\left|
\dfrac{L}{2a_j}
\Re\left[\int_{-\infty}^{x}
\p_tQ^{\ep}\cdot Q^{\ep}
\,dy
\right]
i\p_x^{m-1}Q^{\ep}_j
\right|
\nonumber
\\
&\leqslant 
\dfrac{L}{2|a_j|}
\|\p_tQ^{\ep}(t)\|_{L^2}
\|Q^{\ep}(t)\|_{L^2}
\,|\p_x^{m-1}Q^{\ep}_j|
\nonumber
\\
&\leqslant 
A_{L,m}^6(\|Q^{\ep}(t)\|_{H^4})
|\p_x^{m-1}Q^{\ep}_j|.
\nonumber
\end{align}
This shows
\begin{align}
\left\|
\left(
\dfrac{L}{4a_j}(\p_t\Phi^{\ep})i\p_x^{m-1}Q^{\ep}_j
\right)(t)
\right\|_{L^2}
&\leqslant 
A_{L,m}^6(\|Q^{\ep}(t)\|_{H^4})
\|Q^{\ep}(t)\|_{H^m}. 
\label{eq:101}
\end{align}
Combining \eqref{eq:91} and \eqref{eq:101}, we obtain 
\begin{align}
 &
 \p_t\left(
 \dfrac{L}{4a_j}\Phi^{\ep}i\p_x^{m-1}Q^{\ep}_j
 \right)
\nonumber
\\
&=
 \left\{(-\ep^5+ia_j)\p_x^4+b_j\p_x^3+i\lambda_j\p_x^2\right\}
 \left(
 \dfrac{L}{4a_j}\Phi^{\ep}i\p_x^{m-1}Q^{\ep}_j
 \right)
 \nonumber
 \\&\quad 
+\left(
\frac{\ep^5 i}{a_j}+1
\right)L\p_x\left\{g(Q^{\ep})\p_xU_j^{\ep}\right\}  
+O(A_{L,m}^4(\|Q^{\ep}(t)\|_{H^2})|Q^{\ep}||\p_xU^{\ep}|)
\nonumber
\\
&\quad +r_{\leqslant m}^9, 
\label{eq:102}
\end{align}
where 
\begin{align}
  &\| r^9_{\leqslant m}(t)\|_{L^2}
   \leqslant 
   A_{L,m}^7(\|Q^{\ep}(t)\|_{H^4})\|Q^{\ep}(t)\|_{H^m}.
   \label{eq:103}
\end{align}
\par 
Consequently, combining 
\eqref{eq:414r1} (with \eqref{eq:414r2})
and  \eqref{eq:102} (with \eqref{eq:103}), 
and then using \eqref{eq:V_j},  
we deduce 
\begin{align}
\p_tV_j^{\ep}
&=
  \left\{(-\ep^5+ia_j)\p_x^4+b_j\p_x^3+i\lambda_j\p_x^2\right\}V_j^{\ep}
  \nonumber
  \\
  &\quad 
  +\left(
  \frac{\ep^5 i}{a_j}+1
  \right)L\p_x\left\{g(Q^{\ep})\p_xU_j^{\ep}\right\}
 +O(g(Q^{\ep})|\p_x^2U^{\ep}|)
   \nonumber
   \\
   &\quad
   +O(A_{L,m}^8(\|Q^{\ep}(t)\|_{H^2})|Q^{\ep}||\p_xU^{\ep}|)
  +r^{3}_{\leqslant m}+r^{9}_{\leqslant m}
  \nonumber
  \\
 &=
  \left\{(-\ep^5+ia_j)\p_x^4+b_j\p_x^3+i\lambda_j\p_x^2\right\}V_j^{\ep}
    \nonumber
    \\
    &\quad
  +\left(
  \frac{\ep^5 i}{a_j}+1
  \right)L\p_x\left\{g(Q^{\ep})\p_xV_j^{\ep}\right\}
 +O(g(Q^{\ep})|\p_x^2V^{\ep}|)
   \nonumber
   \\
   &\quad
   +O(A_{L,m}^9(\|Q^{\ep}(t)\|_{H^2})|Q^{\ep}||\p_xV^{\ep}|)
  +r^{10}_{\leqslant m}, 
\label{eq:111}
\end{align}
where   
\begin{align}
  &\| r^{10}_{\leqslant m}(t)\|_{L^2}
   \leqslant 
   A_{L,m}^{10}(\|Q^{\ep}(t)\|_{H^4})\|Q^{\ep}(t)\|_{H^m}.  
   \label{eq:112}
\end{align}
\par 
In what follows, we estimate  $\mathcal{E}_m(Q^{\ep}(t))^2$ for  
$t\in [0,T_{\ep}^{\star}]$ where $T_{\ep}^{\star}$ is introduced by \eqref{eq:cuttime}. 
Using \eqref{eq:111} 
and the integration by parts, 
we deduce  
\begin{align}
&\dfrac{1}{2}\dfrac{d}{dt}
\|V^{\ep}(t)\|_{L^2}^2
=
\sum_{j=1}^n
\Re\int_{\RR}
\p_tV_j^{\ep}\overline{V_j^{\ep}}\,dx
\nonumber
\\
&\leqslant 
-\ep^5\int_{\RR}
|\p_x^2V^{\ep}|^2\,dx
-L\int_{\RR}
g(Q^{\ep})|\p_xV^{\ep}|^2\,dx
\nonumber
\\
&\quad 
+
\sum_{j=1}^n
\Re\int_{\RR}
O(g(Q^{\ep})|\p_x^2V^{\ep}|)
\overline{V_j^{\ep}}\,dx
\nonumber
\\
&\quad
+
C\,
A_{L,m}^9(\|Q^{\ep}(t)\|_{H^2})
\int_{\RR}
|Q^{\ep}||\p_xV^{\ep}|
|V^{\ep}|\,dx
\nonumber
\\
&\quad 
+\|r^{10}_{\leqslant m}(t)\|_{L^2}
\|V^{\ep}(t)\|_{L^2}. 
\label{eq:H121}
\end{align}
The second term of the right hand side of \eqref{eq:H121} 
comes from the first term of the right hand side of \eqref{eq:F1m} with \eqref{eq:F1ms}.   
Therefore, by integration by parts, we see that there exists positive 
constants $C_{1}^{\star}$ and $C_{2}^{\star}$ which are independent of $L$ such that 
\begin{align}
&\sum_{j=1}^n
\Re\int_{\RR}
O(g(Q^{\ep})|\p_x^2V^{\ep}|)
\overline{V_j^{\ep}}\,dx
\nonumber
\\
&=
\sum_{j,\ell=1}^n
\Re\int_{\RR}
O(g(Q^{\ep}))
\left(\p_x^2V_{\ell}^{\ep}+\overline{\p_x^2V_{\ell}^{\ep}}\right)
\overline{V_j^{\ep}}\,dx
\nonumber
\\
&=
-\sum_{j,\ell=1}^n
\Re\int_{\RR}
O(g(Q^{\ep}))
\left(\p_xV_{\ell}^{\ep}+\overline{\p_xV_{\ell}^{\ep}}\right)
\overline{\p_xV_j^{\ep}}\,dx
\nonumber
\\
&\quad 
-\sum_{j,\ell=1}^n
\Re\int_{\RR}
\p_x\left(O(g(Q^{\ep}))\right)
\left(\p_xV_{\ell}^{\ep}+\overline{\p_xV_{\ell}^{\ep}}\right)
\overline{V_j^{\ep}}\,dx
\nonumber
\\
&\leqslant 
C_1^{\star}
\int_{\RR}
g(Q^{\ep})|\p_xV^{\ep}|^2\,dx
+C_2^{\star}
\|Q^{\ep}(t)\|_{H^2}
\int_{\RR}
|Q^{\ep}||\p_xV^{\ep}|
|V^{\ep}|\,dx.
\nonumber
\end{align}
Furthermore, the Young inequality for products  
and \eqref{eq:H3Emn} 
shows
\begin{align}
&
\|Q^{\ep}(t)\|_{H^2}
\int_{\RR}
|Q^{\ep}||\p_xV^{\ep}|
|V^{\ep}|\,dx
\nonumber
\\
&\leqslant 
\frac{1}{2}\int_{\RR}|Q^{\ep}|^2|\p_xV^{\ep}|^2\,dx
+
\frac{1}{2}\|Q^{\ep}(t)\|_{H^2}^2\int_{\RR}|V^{\ep}|^2\,dx
\nonumber
\\
&\leqslant 
\frac{1}{2}\int_{\RR}g(Q^{\ep})|\p_xV^{\ep}|^2\,dx
+
\frac{1}{2}P_{L}^{0}(\|Q_0\|_{H^4})\mathcal{E}_m(Q^{\ep}(t))^2. 
\nonumber
\end{align}
Thus, we have 
\begin{align}
&\sum_{j=1}^n
\Re\int_{\RR}
O(g(Q^{\ep})|\p_x^2V^{\ep}|)
\overline{V_j^{\ep}}\,dx
\nonumber
\\
&\leqslant 
\left(C_1^{\star}+\frac{C_2^{\star}}{2}\right)\int_{\RR}g(Q^{\ep})|\p_xV^{\ep}|^2\,dx
+
\frac{C_2^{\star}}{2}
P_{L}^{0}(\|Q_0\|_{H^4})
\mathcal{E}_m(Q^{\ep}(t))^2. 
\label{eq:H122}
\end{align}
In the same way, we use the Young inequality and \eqref{eq:H3Emn} 
to deduce 
\begin{align}
&A_{L,m}^9(\|Q^{\ep }(t)\|_{H^2})
\int_{\RR}
|Q^{\ep}||\p_xV^{\ep}|
|V^{\ep}|\,dx
\nonumber
\\
&\leqslant 
\frac{1}{2}\int_{\RR}g(Q^{\ep})|\p_xV^{\ep}|^2\,dx
+
A^{11}_{L,m}(\|Q_0\|_{H^4})\mathcal{E}_m(Q^{\ep}(t))^2. 
\label{eq:H123}
\end{align}
In addition, in view of \eqref{eq:112}, \eqref{eq:H3Emn}, and \eqref{eq:H3Emm}, 
we obtain
\begin{align}
\|r^{10}_{\leqslant m}(t)\|_{L^2}
\|V^{\ep}(t)\|_{L^2}
&\leqslant 
A^{12}_{L,m}(\|Q_0\|_{H^4})\mathcal{E}_m(Q^{\ep}(t))^2. 
\label{eq:H124}
\end{align}
Therefore, combining 
\eqref{eq:H121}-\eqref{eq:H124}, 
we get 
\begin{align}
&\dfrac{1}{2}\dfrac{d}{dt}
\|V^{\ep}(t)\|_{L^2}^2
+\ep^5\int_{\RR}
|\p_x^2V^{\ep}|^2\,dx
\nonumber
\\
&\leqslant 
\left(
-L
+C_1^{\star}+\frac{C_2^{\star}}{2}
+\frac{1}{2}
\right)\int_{\RR}
g(Q^{\ep})|\p_xV^{\ep}|^2\,dx
\nonumber
\\&\quad 
+
A^{13}_{L,m}(\|Q_0\|_{H^4})\mathcal{E}_m(Q^{\ep}(t))^2. 
\label{eq:125}
\end{align}
\par 
On the other hand, permitting loss of one derivative, 
we can easily obtain 
\begin{align}
 &\dfrac{1}{2}\dfrac{d}{dt}
 \|Q^{\ep}(t)\|_{H^{m-1}}^2
 +\ep^5\sum_{k=0}^{m-1}
 \int_{\RR}
  |\p_x^{k+2}Q^{\ep}|^2\,dx
 \nonumber
 \\&\leqslant 
 A^{14}_{L,m}(\|Q_0\|_{H^4})\mathcal{E}_m(Q^{\ep}(t))^2.
 \label{eq:H132}
 \end{align}
\par 
Noting again that $C_1^{\star}$ and $C_1^{\star}$ are independent of $L$, 
we can take $L=L_0$ to satisfy 
$-L_0+C_1^{\star}+(C_2^{\star}/2)+(1/2)<0$. 
By fixing $L=L_0$ and combining \eqref{eq:125} and \eqref{eq:H132}, 
we obtain 
\begin{align}
& \dfrac{d}{dt}
 \mathcal{E}_m(Q^{\ep}(t))^2
 +2\ep^5\left(
 \|\p_x^2V^{\ep}(t)\|_{L^2}^2
 +
 \|\p_x^2Q^{\ep}(t)\|_{H^{m-1}}^2
 \right)
\nonumber
\\
&
\leqslant 
 A^{15}_{L_0,m}(\|Q_0\|_{H^4})\mathcal{E}_m(Q^{\ep}(t))^2. 
 \label{eq:H133}
 \end{align}
Therefore, the Gronwall inequality shows  
\begin{align}
\mathcal{E}_m(Q^{\ep}(t))^2
&\leqslant 
\mathcal{E}_m(Q_0^{\ep})^2
\exp(A^{15}_{L_0,m}(\|Q_0\|_{H^4})t)
\quad 
\text{for}
\quad
t\in [0,T_{\ep}^{\star}].
\label{eq:H151}
\end{align}
This inequality \eqref{eq:H151} for $m=4$ and the definition of $T_{\ep}^{\star}$
implies 
$$4\leqslant \exp(A^{15}_{L_0,4}(\|Q_0\|_{H^4})T_{\ep}^{\star}).$$ 
From this, we obtain 
\begin{equation}
T_{\ep}^{\star}\geqslant 
\frac{\log 4}{A^{15}_{L_0,4}(\|Q_0\|_{H^4})}
=:T>0, 
\label{eq:extime}
\end{equation}
and it follows that  
\begin{align}
\sup_{t\in [0,T]}
\mathcal{E}_m(Q^{\ep}(t))^2
&\leqslant 
\mathcal{E}_m(Q_0^{\ep})^2
\exp(A^{15}_{L_0,m}(\|Q_0\|_{H^4})T)
\nonumber. 
\nonumber
\end{align} 
Furthermore, 
by combining this, \eqref{eq:H3Emm}, \eqref{eq:H3Em} for $t=0$, 
and  \eqref{eq:bs1}, 
we obtain 
\begin{align}
\sup_{t\in [0,T]}
\|Q^{\ep}(t)\|_{H^m}^2
&\leqslant 
C(T,L_0,\|Q_0\|_{H^4})
\mathcal{E}_m(Q_0^{\ep})^2
\nonumber
\\
&\leqslant 
C(T,L_0,\|Q_0\|_{H^4})
\|Q_0^{\ep}\|_{H^m}^2
\label{eq:A153}
\\
&\leqslant 
C(T,L_0,\|Q_0\|_{H^4})
\|Q_0\|_{H^m}^2.
\label{eq:H153}
\end{align}
Since the right hand side of \eqref{eq:H153} is independent of $t$ and $\ep$, 
we conclude $\{Q^{\ep}\}_{\ep\in (0,1)}$ is bounded in 
$L^{\infty}(0,T;H^m(\RR;\mathbb{C}^n))$, 
which completes the proof of Proposition~\ref{proposition:preloc}.
\end{proof}
\begin{remark}
\label{remark:bs}
Once Proposition~\ref{proposition:preloc} is proved, 
the standard compactness argument shows 
there exists a subsequence of $\left\{Q^{\ep}\right\}_{\ep\in (0,1)}$ 
which converges to a limit $Q^{\star}$ weak$^{\star}$ 
in $L^{\infty}(0,T;H^m(\RR;\mathbb{C}^n))$ 
and (strongly) in $C([0,T]; H^{m-1}_{\text{loc}}(\RR;\mathbb{C}^n))$.
However, it is still not straightforward to verify 
that $Q^{\star}$ is actually a solution to \eqref{eq:apde}-\eqref{eq:adata}, 
in that the subconvergence of $\left\{Q^{\ep}\right\}_{\ep\in (0,1)}$ 
in $C([0,T]; H^{m-1}_{\text{loc}}(\RR;\mathbb{C}^n))$ 
seems to be insufficient to ensure the subconvergence of 
the nonlocal term $F_3(Q^{\ep}, \p_xQ^{\ep})$ 
to $F_3(Q^{\star}, \p_xQ^{\star})$ even in the sense of distribution. 
To avoid the argument to justify the above, we choose to take an advantage of 
the Bona-Smith approximation $\left\{Q_0^{\ep}\right\}_{\ep\in (0,1)}$  
satisfying \eqref{eq:bs1}-\eqref{eq:6093},  
which will be demonstrated in Sections \ref{section:BS} and \ref{section:prooflw}. 
\end{remark}
\section{Estimate for the difference of Bona-Smith approximated solutions} 
\label{section:BS} 
Let $m$ be an integer with $m\geqslant 4$. 
For $Q_0={}^t(Q_{01},\ldots,Q_{0n})\in H^m(\RR;\mathbb{C}^n)$, let 
$\left\{Q_{0}^{\ep}\right\}_{\ep\in (0,1)}$ be the Bona-Smith approximation of $Q_0$.
We denote $Q^{\mu}$ and $Q^{\nu}$ by corresponding solutions to 
\eqref{eq:bpde}-\eqref{eq:bdata} for $\ep=\mu$ and $\ep=\nu$ respectively, that is,  
\begin{alignat}{2}
 \left(
 \p_t+\mu^5\p_x^4-iM_a\p_x^4-M_b\p_x^3-iM_{\lambda}\p_x^2
 \right)
 Q^{\mu}
  &=F(Q^{\mu},\p_xQ^{\mu}, \p_x^2Q^{\mu}), 
 \label{eq:mapde}
 \\
  Q^{\mu}(0,x)
   &=
   Q_{0}^{\mu}(x),
 \label{eq:madata}
 \end{alignat}
 \begin{alignat}{2}
  \left(
  \p_t+\nu^5\p_x^4-iM_a\p_x^4-M_b\p_x^3-iM_{\lambda}\p_x^2
  \right)
  Q^{\nu}
   &=F(Q^{\nu},\p_xQ^{\nu}, \p_x^2Q^{\nu}), 
  \label{eq:napde}
  \\
   Q^{\nu}(0,x)
    &=
    Q_{0}^{\nu}(x).
  \label{eq:nadata}
  \end{alignat}
Proposition~\ref{proposition:preloc} which is proved in Section~\ref{section:local} ensures    
both $\left\{Q^{\mu}\right\}_{\mu\in(0,1)}$ and $\left\{Q^{\nu}\right\}_{\nu\in (0,1)}$ 
are uniformly bounded in $L^{\infty}(0,T;H^m(\RR;\mathbb{C}^n))$, where 
$T=T(\|Q_0\|_{H^4})>0$ is decided by \eqref{eq:extime} independently of $\mu$ and $\nu$. 
\par
The goal of this section is to get the following: 
\begin{proposition}
\label{proposition:cauchy}
There exists a constant $C=C(T,\|Q_0\|_{H^m})>1$ such that 
for all $\mu$ and $\nu$ satisfying 
$0<\mu\leqslant \nu<1$, 
\begin{align}
\|Q^{\mu}-Q^{\nu}\|_{C([0,T];H^1)}
&\leqslant 
C
(\nu^{m-1}+\nu^4),
\label{eq:1cauchy}
\\
\|Q^{\mu}-Q^{\nu}\|_{C([0,T];H^m)}
&\leqslant 
C\left(
\nu^{m-3}+\nu
+\|Q_0^{\mu}-Q_0^{\nu}\|_{H^m}
\right).
\label{eq:mcauchy}
\end{align}  
\end{proposition}
\begin{proof}[Proof of Proposition~\ref{proposition:cauchy}]
For $\mu, \nu$ satisfying 
$0<\mu \leqslant \nu<1$, 
we set $W:=Q^{\mu}-Q^{\nu}$, that is, 
$W={}^t(W_1,\ldots,W_n)$ and
$
W_j
=Q_j^{\mu}-Q_j^{\nu}$ 
for $j\in \left\{1,\ldots,n\right\}$.
For $k\in \left\{1,\ldots,m\right\}$, we introduce a $\mathbb{C}^n$-valued function
$Z^{k}={}^t(Z^{k}_1,\ldots,Z^{k}_n)$, 
where 
\begin{align}
 &Z^{k}_j=Z^{k}_j(t,x)
 :=\p_x^kW_j(t,x)+\dfrac{L}{4a_j}\Phi^{\mu}(t,x)i\p_x^{k-1}W_j(t,x)
 \quad (j\in \left\{1,\ldots,n\right\}), 
 \label{eq:Z_j}
\\
 &\Phi^{\mu}=\Phi^{\mu}(t,x)
 :=\int_{-\infty}^{x}g(Q^{\mu}(t,y))\,dy
 \left(=\int_{-\infty}^{x}|Q^{\mu}(t,y)|^2\,dy\right), 
 \label{eq:Phia}
 \end{align} 
and $L>1$ is a sufficiently large constant which will be taken later independently of 
$j$, $\mu$, and $\nu$. 
Furthermore we define $\mathcal{E}^{\mu,\nu}_k(W)
=\mathcal{E}^{\mu,\nu}_k(W(t)):[0,T]\to [0,\infty)$ 
to satisfy 
\begin{align}
\mathcal{E}^{\mu,\nu}_k(W(t))^2
&=
\|Z^k(t)\|_{L^2}^2+\|W(t)\|_{H^{k-1}}^2.
\label{eq:Ekl}
\end{align}
We shall estimate $\mathcal{E}^{\mu,\nu}_1(W(t))$ and $\mathcal{E}^{\mu,\nu}_m(W(t))$ for $t\in [0,T]$
to get \eqref{eq:1cauchy} and \eqref{eq:mcauchy}. 
Roughly speaking, 
these estimates can be derived in the same way as we estimate   
\eqref{eq:Nml} in the previous section. 
The main point we need to care is that 
the estimate for the time-derivative of $\|Z^m(t)\|_{L^2}^2$
involves some terms including 
$\|\p_x^{m+j}Q^{\nu}(t)\|_{L^2}$ for $j\in \left\{1,\ldots,4\right\}$ 
which grow as $\nu\downarrow 0$ in relation with  \eqref{eq:6092}. 
To compensate the growth, we apply the decay properties of \eqref{eq:1cauchy} 
and the factor $\nu^5$ in \eqref{eq:napde}.
\par 
Before going to the detail, 
we here collect some estimates on $[0,T]$ and notation
used later. 
First, since the estimates
for the solution $Q^{\ep}$ to \eqref{eq:bpde}-\eqref{eq:bdata} 
in the previous section 
also hold for $Q^{\mu}$ (and $Q^{\nu}$) on $[0,T]$,  
it follows from  
\eqref{eq:A153} and \eqref{eq:H153} for $Q^{\mu}$ and $Q_0^{\mu}$, 
\begin{align}
\|Q^{\mu}\|_{C([0,T]:H^m)}
&\leqslant 
C_1(T, \|Q_0\|_{H^4})
\|Q_0^{\mu}\|_{H^m} 
\label{eq:B51}
\\
&\leqslant 
C_2(T,\|Q_0\|_{H^m}), 
\label{eq:B52}
\end{align} 
where $C_1(T, \|Q_0\|_{H^4})$ and $C_2(T,\|Q_0\|_{H^m})$ are positive 
constants depending also on $L_0$ (in the previous Section) but not on $\mu$.
Second, 
by a similar argument to obtain \eqref{eq:H3Em} 
and by \eqref{eq:B52}, there exists a positive constant 
$C_3(L,T,\|Q_0\|_{H^m})$ depending also on $L_0$ but not on $\mu$ such that 
 \begin{align}
 &\frac{\|W(t)\|_{H^k}^2}{C_3(L,T,\|Q_0\|_{H^m})}
 \leqslant 
 \mathcal{E}^{\mu,\nu}_k(W(t))^2
 \leqslant 
 C_3(L,T,\|Q_0\|_{H^m})\|W(t)\|_{H^k}^2
 \label{eq:H4Em}
 \end{align}
for any $t\in [0,T]$.
(Although constants $C_k(\cdot,\ldots,\cdot)$ 
appearing here and hereafter in this part may 
also depend on $L_0$, 
we omit to write it for simplicity. 
 By noting $L_0$ is a fixed constant to ensure Proposition~\ref{proposition:preloc}, 
 any confusion will not occur.)  
Moreover, in what follows in this part, 
 we use $B_k(\cdot)$ and $B_{L,k}(\cdot)$ for 
 an integer $k$ to denote 
 a positive-valued increasing function on $[0,\infty)$. 
 We use the latter only if the increasing function depends also on $L$.  
\vspace{0.5em}
\\
\underline{Proof of \eqref{eq:1cauchy}:} 
\\
We estimate $\mathcal{E}^{\mu,\nu}_1(W(t))$ for $t\in [0,T]$. 
Since $Q^{\mu}$ and $Q^{\nu}$ satisfy \eqref{eq:mapde} and \eqref{eq:napde} respectively, 
\begin{align}
&\left\{
\p_t+(\mu^5-ia_j)\p_x^4-b_j\p_x^3-i\lambda_j\p_x^2
\right\}
\p_xW_j
\nonumber
\\
&\qquad =
(\nu^5-\mu^5)\p_x^4(\p_xQ^{\nu}_j)+I^{(1)}+I^{(2)}+I^{(3)}, 
\nonumber
\end{align}
where 
\begin{align}
I^{(1)}
&:=
\p_x\left(
F_j^1(Q^{\mu},\p_x^2Q^{\mu})
\right)
-
\p_x\left(
F_j^1(Q^{\nu},\p_x^2Q^{\nu})
\right), 
\nonumber
\\
I^{(2)}
&:=
\p_x\left(
F_j^2(Q^{\mu},\p_xQ^{\mu})
\right)
-
\p_x\left(
F_j^2(Q^{\nu},\p_xQ^{\nu})
\right), 
\nonumber
\\
I^{(3)}
&:=
\p_x\left(
F_j^3(Q^{\mu},\p_xQ^{\mu})
\right)
-
\p_x\left(
F_j^3(Q^{\nu},\p_xQ^{\nu})
\right).
\nonumber
\end{align} 
Since $F_j^1$ satisfies the condition (F1),  
\begin{align}
\p_x\left(
F_j^1(Q^{\mu},\p_x^2Q^{\mu})
\right)
&=
O\left(
g(Q^{\mu})|\p_x^3Q^{\mu}|
\right)
+
O\left(
|\p_x\left\{g(Q^{\mu})\right\}|
|\p_x^2Q^{\mu}|
\right),
\nonumber
\\
\p_x\left(
F_j^1(Q^{\nu},\p_x^2Q^{\nu})
\right)
&=
O\left(
g(Q^{\nu})|\p_x^3Q^{\nu}|
\right)
+
O\left(
|\p_x\left\{g(Q^{\nu})\right\}|
|\p_x^2Q^{\nu}|
\right). 
\nonumber
\end{align}
By taking the difference between both sides, 
we deduce 
\begin{align}
I^{(1)}
&=
O\left(g(Q^{\mu})|\p_x^3W|\right)
+
O\left(
|\p_x\left\{g(Q^{\mu})\right\}||\p_x^2W|
\right)
+r_1+r_2
\nonumber
\\
&=
O\left(g(Q^{\mu})|\p_x^3W|\right)
+
O\left(
\|Q^{\mu}(t)\|_{H^2}|Q^{\mu}||\p_x^2W|
\right)
+r_1+r_2,
\nonumber
\end{align}
where 
\begin{align}
r_1&=O
\left(
(|Q^{\mu}|+|Q^{\nu}|)
|Q^{\mu}-Q^{\nu}|
|\p_x^3Q^{\nu}|
\right), 
\nonumber
\\
r_2&=
O\left(
(|Q^{\mu}|+|Q^{\nu}|)
|\p_xQ^{\mu}-\p_xQ^{\nu}|
|\p_x^2Q^{\nu}|
\right)
\nonumber
\\
&\quad
+
O\left(
(|\p_xQ^{\mu}|+|\p_xQ^{\nu}|)|Q^{\mu}-Q^{\nu}|
|\p_x^2Q^{\nu}|
\right). 
\nonumber
\end{align}
It is easy to deduce 
\begin{align}
\|r_1(t)\|_{L^2}
&\leqslant C
(\|Q^{\mu}(t)\|_{L^{\infty}}+\|Q^{\nu}(t)\|_{L^{\infty}})
\|(Q^{\mu}-Q^{\nu})(t)\|_{L^{\infty}}
\|\p_x^3Q^{\nu}(t)\|_{L^2}
\nonumber
\\
&
\leqslant 
B_1(\|Q^{\mu}(t)\|_{H^1}+\|Q^{\nu}(t)\|_{H^3})
\|W(t)\|_{H^1}, 
\nonumber
\end{align}
\begin{align}
\|r_2(t)\|_{L^2}
&\leqslant 
C
(\|Q^{\mu}(t)\|_{L^{\infty}}+\|Q^{\nu}(t)\|_{L^{\infty}})
\|(\p_xQ^{\mu}-\p_xQ^{\nu})(t)\|_{L^2}
\|\p_x^2Q^{\nu}(t)\|_{L^{\infty}}
\nonumber
\\
&\quad 
+C
(\|\p_xQ^{\mu}(t)\|_{L^{2}}
+\|\p_xQ^{\nu}(t)\|_{L^{2}})
\|(Q^{\mu}-Q^{\nu})(t)\|_{L^{\infty}}
\|\p_x^2Q^{\nu}(t)\|_{L^{\infty}}
\nonumber
\\
&\leqslant 
B_2(\|Q^{\mu}(t)\|_{H^1}+\|Q^{\nu}(t)\|_{H^3})
\|W(t)\|_{H^1}.
\nonumber 
\end{align} 
Since $F_j^2$ satisfies (F2), 
the following holds for 
both $\ep=\mu$ and $\ep=\nu$: 
\\
If $d_2=0$, then 
$$
\p_x\left(
F_j^2(Q^{\ep},\p_xQ^{\ep})
\right)
=
O\left(
\sum_{p_1=0}^{d_1}
|Q^{\ep}|^{p_1}|\p_xQ^{\ep}|
\right). 
$$
If $d_2\geqslant 1$, then  
\begin{align}
\p_x\left(
F_j^2(Q^{\ep},\p_xQ^{\ep})
\right)
&=
O\left(
\sum_{p_1=0}^{d_1}
\sum_{p_2=1}^{d_2}
|Q^{\ep}|^{1+p_1}|\p_xQ^{\ep}|^{p_2-1}|\p_x^2Q^{\ep}|
\right)
\nonumber
\\
&\quad
+
O\left(
\sum_{p_1=0}^{d_1}
\sum_{p_2=0}^{d_2}
|Q^{\ep}|^{p_1}|\p_xQ^{\ep}|^{p_2+1}
\right)
\nonumber
\\
&\quad
+
O\left(
\sum_{p_1=0}^{d_1}
|Q^{\ep}|^{p_1}|\p_xQ^{\ep}|
\right). 
\nonumber
\end{align}
In both cases, it follows that 
\begin{align}
I^{(2)}
&=
O\left(
B_3(
\|Q^{\mu}(t)\|_{L^{\infty}}
+
\|\p_xQ^{\mu}(t)\|_{L^{\infty}}
)
|Q^{\mu}|
|\p_x^2W|
\right)
+r_3
\nonumber
\\
&=
O\left(
B_3(
\|Q^{\mu}(t)\|_{H^2}
)
|Q^{\mu}|
|\p_x^2W|
\right)
+r_3, 
\nonumber
\end{align}
where 
\begin{align}
\|r_3(t)\|_{L^2}
&\leqslant 
B_4(\|Q^{\mu}(t)\|_{H^2}+\|Q^{\nu}(t)\|_{H^3})
\|W(t)\|_{H^1}.
\nonumber
\end{align}
Since $F_{j,r}^{3,A}$ and $F_{j,r}^{3,B}$ satisfy (F3), 
the following holds for both $\ep=\mu$ and $\ep=\nu$: 
\begin{align}
&\p_x\left(
F_j^3(Q^{\ep},\p_xQ^{\ep})
\right)
\nonumber
\\
&=
\sum_{r=1}^{n}
F_{j,r}^{3,A}(Q^{\ep},\p_xQ^{\ep})
F_{j,r}^{3,B}(Q^{\ep})
\nonumber
\\
&\quad
+\sum_{r=1}^{n}
\left(
\int_{-\infty}^{x}
F_{j,r}^{3,A}(Q^{\ep},\p_xQ^{\ep})(t,y)dy
\right)
\p_x\left(
F_{j,r}^{3,B}(Q^{\ep})
\right)
\nonumber
\\
&=
O\left(
\sum_{p_3=0}^{d_3}
\sum_{p_5=0}^{d_5}
|Q^{\ep}|^{3+p_3+p_5}
+
\sum_{p_4=0}^{d_4}
\sum_{p_5=0}^{d_5}
|Q^{\ep}|^{1+p_5}|\p_xQ^{\ep}|^{2+p_4}
\right)
\nonumber
\\
&\quad
+\left(
\int_{-\infty}^{x}
O\left(
\sum_{p_3=0}^{d_3}
|Q^{\ep}|^{2+p_3}
+
\sum_{p_4=0}^{d_4}
|\p_xQ^{\ep}|^{2+p_4}
\right)(t,y)dy
\right)
\nonumber
\\
&\qquad \quad \times
O\left(
\sum_{p_5=0}^{d_5}
|Q^{\ep}|^{p_5}|\p_xQ^{\ep}|
\right).
\nonumber
\end{align}
Therefore, it is now easy to obtain 
\begin{align}
\|I^{(3)}(t)\|_{L^2}
&\leqslant 
B_5(\|Q^{\mu}(t)\|_{H^2}+
\|Q^{\nu}(t)\|_{H^2}
)\|W(t)\|_{H^1}.
\nonumber
\end{align}
Gathering them, we obtain 
\begin{align}
&\left\{
\p_t+(\mu^5-ia_j)\p_x^4-b_j\p_x^3-i\lambda_j\p_x^2
\right\}
\p_xW_j
\nonumber
\\
&=
O\left(g(Q^{\mu})|\p_x^3W|\right)
+
O\left(
B_6(\|Q^{\mu}(t)\|_{H^2})
|Q^{\mu}||\p_x^2W|
\right)
\nonumber
\\
&\quad 
+(\nu^5-\mu^5)\p_x^4(\p_xQ^{\nu}_j)
+r_4, 
\label{eq:8311}
\end{align}
where 
\begin{align}
\|r_4(t)\|_{L^2}
\leqslant 
B_7(\|Q^{\mu}(t)\|_{H^2}+\|Q^{\nu}(t)\|_{H^3})
\|W(t)\|_{H^1}.
\nonumber
\end{align}
We next compute the right hand side of 
 \begin{align}
 &
 \p_t\left(
 \dfrac{L}{4a_j}\Phi^{\mu}iW_j
 \right)
 =
 \dfrac{L}{4a_j}\Phi^{\mu}i\p_tW_j
 +
 \dfrac{L}{4a_j}(\p_t\Phi^{\mu})iW_j.
 \label{eq:8312}
 \end{align}
The argument is almost the same as that to obtain 
\eqref{eq:91}-\eqref{eq:101} and \eqref{eq:102}-\eqref{eq:103}. 
First, it is not difficult  to show 
\begin{align}
\p_tW_j
&=
\left\{
(-\mu^5+ia_j)\p_x^4+b_j\p_x^3+i\lambda_j\p_x^2
\right\}W_j
\nonumber
\\
&\quad 
+
O\left(g(Q^{\mu})|\p_x^2W|\right)
+
(\nu^5-\mu^5)\p_x^4Q_j^{\nu}
+
r_5, 
\nonumber
\end{align}
where
\begin{align}
&\|r_5(t)\|_{L^2}
\leqslant 
B_8(\|Q^{\mu}(t)\|_{H^2}+\|Q^{\nu}(t)\|_{H^3})
\|W(t)\|_{H^1}.
\nonumber
\end{align}
Substituting this and using 
$\|\Phi^{\mu}(t)\|_{L^{\infty}}
\leqslant 
\|Q^{\mu}(t)\|_{L^2}^2$, we deduce 
\begin{align}
&\dfrac{L}{4a_j}\Phi^{\mu}i\p_tW_j
=
 \left\{(-\mu^5+ia_j)\p_x^4+b_j\p_x^3+i\lambda_j\p_x^2
 \right\}
 \left(
 \dfrac{L}{4a_j}\Phi^{\mu}iW_j
 \right)
\nonumber
\\
&\qquad \qquad   \qquad \quad
+\left(
\dfrac{\mu^5 i}{a_j}+1
\right)
L\p_x\left\{g(Q^{\mu})
\p_x^2W_j\right\}
\nonumber
\\
&\qquad \qquad   \qquad \quad
+O(B_{L,1}(
\|Q^{\mu}(t)\|_{H^2})
|Q^{\mu}||\p_x^2W|)
\nonumber
\\
&\qquad \qquad   \qquad \quad
+(\nu^5-\mu^5)\dfrac{L}{4a_j}\,
O\left(\|Q^{\mu}(t)\|_{L^2}^2\right)\p_x^4Q^{\nu}_j
   +r_6, 
\label{eq:8313}
\end{align}
where
\begin{align}
  &\|r_6(t)\|_{L^2}
   \leqslant 
 B_{L,2}(\|Q^{\mu}(t)\|_{H^3}+\|Q^{\nu}(t)\|_{H^3})
 \|W(t)\|_{H^1}.
   \nonumber
\end{align}
Second, 
the same computation to obtain \eqref{eq:101} shows    
\begin{align}
\left\|
\left(
\dfrac{L}{4a_j}(\p_t\Phi^{\mu})iW_j
\right)(t)
\right\|_{L^2}
&\leqslant 
B_{L,3}(\|Q^{\mu}(t)\|_{H^4})\|W(t)\|_{H^1}. 
\label{eq:8314}
\end{align}
Applying \eqref{eq:8313} and \eqref{eq:8314} to \eqref{eq:8312}, and combining this with 
\eqref{eq:8311}, we obtain 
\begin{align}
\p_tZ^1_j
&=
 \left\{(-\mu^5+ia_j)\p_x^4+b_j\p_x^3+i\lambda_j\p_x^2
 \right\}
Z^1_j
\nonumber
\\
&\quad
+
O\left(g(Q^{\mu})|\p_x^3W|\right)
+\left(
\dfrac{\mu^5 i}{a_j}+1
\right)L\p_x\left\{g(Q^{\mu})
\p_x^2W_j\right\}
\nonumber
\\
&\quad 
+
O\left(
B_{L,4}(\|Q^{\mu}(t)\|_{H^2})
|Q^{\mu}||\p_x^2W|
\right)
\nonumber
\\
&\quad
+
(\nu^5-\mu^5)
\left\{
\p_x^4(\p_xQ_j^{\nu})
+
\dfrac{L}{4a_j}
O\left(\|Q^{\mu}(t)\|_{L^2}^2\right)\p_x^4Q^{\nu}_j
\right\}
+r_7, 
\nonumber
\end{align}
where
\begin{align}
  \|r_7(t)\|_{L^2}
   &\leqslant 
 B_{L,5}(\|Q^{\mu}(t)\|_{H^4}+\|Q^{\nu}(t)\|_{H^3})
 \|W(t)\|_{H^1}.
\nonumber
\end{align}
Furthermore, using 
$\p_xW_j=Z_j^1-L\,O(\|Q^{\mu}(t)\|_{L^2}^2|W_j|)$ 
which follows from \eqref{eq:Z_j} for $k=1$, 
we obtain 
\begin{align}
\p_tZ^1_j
&=
\left\{(-\mu^5+ia_j)\p_x^4+b_j\p_x^3+i\lambda_j\p_x^2
 \right\}
Z^1_j
\nonumber
\\
&\quad 
+
O\left(g(Q^{\mu})|\p_x^2Z^1|\right)
+\left(
\dfrac{\mu^5 i}{a_j}+1
\right)L\p_x\left\{g(Q^{\mu})
\p_xZ_j^1\right\}
\nonumber
\\
&\quad
+
O\left(
B_{L,6}(\|Q^{\mu}(t)\|_{H^2})
|Q^{\mu}||\p_xZ^1|
\right)
\nonumber
\\
&\quad
+
(\nu^5-\mu^5)
\left\{
\p_x^4(\p_xQ_j^{\nu})
+
\dfrac{L}{4a_j}
O\left(\|Q^{\mu}(t)\|_{L^2}^2\right)\p_x^4Q^{\nu}_j
\right\}
+r_8, 
\label{eq:8315}
\end{align}
where 
\begin{align} 
  \|r_8(t)\|_{L^2}
   &\leqslant 
 B_{L,7}(\|Q^{\mu}(t)\|_{H^4}+\|Q^{\nu}(t)\|_{H^3})
 \|W(t)\|_{H^1}.
 \label{eq:8316}
\end{align}
\par 
By the almost same way to obtain \eqref{eq:125},  
we use \eqref{eq:8315} with \eqref{eq:8316} to derive 
\begin{align}
&\frac{1}{2}\frac{d}{dt}
\|Z^1(t)\|_{L^2}^2
+
\mu^5\|\p_x^2Z^1(t)\|_{L^2}^2 
\nonumber
\\
&
\leqslant 
(-L+C_{\star})\int_{\RR}
g(Q^{\mu})|\p_xZ^1|^2\,dx
+J_1+J_2, 
\label{eq:9020}
\end{align}
where 
\begin{align}
J_1&:=
(\nu^5-\mu^5)
\sum_{j=1}^{n}
\Re
\int_{\RR}
\left\{
\p_x^4(\p_xQ_j^{\nu})
+
\dfrac{L}{4a_j}
O\left(\|Q^{\mu}(t)\|_{L^2}^2\right)\p_x^4Q^{\nu}_j
\right\}
\overline{Z_j^1}
\,dx, 
\nonumber
\\
J_2&:=
B_{L,7}(\|Q^{\mu}(t)\|_{H^4}+\|Q^{\nu}(t)\|_{H^3})
\|W(t)\|_{H^1}\|Z^1(t)\|_{L^2}, 
\nonumber
\end{align}
and $C_{\star}>0$ is a positive constant which is independent of $L$. 
Recall here that \eqref{eq:B52} ensures $\|Q^{\mu}\|_{C([0,T];H^4)}$ 
and $\|Q^{\nu}\|_{C([0,T];H^4)}$  
are bounded by a positive constant depending 
on $T$ and $\|Q_0\|_{H^m}$ but not on $\mu$ and $\nu$. 
From this and \eqref{eq:H4Em}, it is easy to have   
\begin{align}
J_2&\leqslant 
C_4(L,T,\|Q_0\|_{H^m})
\mathcal{E}^{\mu,\nu}_1(W(t))^2. 
\label{eq:2304011}
\end{align}
In addition, it follows that 
\begin{align}
J_1&\leqslant 
C_5(L,T,\|Q_0\|_{H^m})
(\nu^5-\mu^5)\|Q^{\nu}(t)\|_{H^5}\|Z^1(t)\|_{L^2}.
\nonumber
\end{align}
Here, applying \eqref{eq:B51} for $m=5$, we have  
$$
\|Q^{\nu}\|_{C([0,T];H^5)}
\leqslant 
C_6(T,\|Q_0\|_{H^4})
\|Q_0^{\nu}\|_{H^5}.
$$
Moreover, applying \eqref{eq:6092} for $\ep=\nu$, $m=4$ and $\ell=1$, 
we have  
$$
\|Q_0^{\nu}\|_{H^5}
\leqslant 
C\nu^{-1}\|Q_0\|_{H^4}.
$$ 
Combining them, we obtain 
\begin{align}
J_1&\leqslant 
C_7(L,T,\|Q_0\|_{H^m})
(\nu^5-\mu^5)\nu^{-1}\|Z^1(t)\|_{L^2}
\nonumber
\\
&
\leqslant 
C_7(L,T,\|Q_0\|_{H^m})
\nu^4\mathcal{E}^{\mu,\nu}_1(W(t)).
\nonumber
\end{align}
Consequently, going back to \eqref{eq:9020} and 
taking $L=L_1>1$ so that $-L_1+C_{\star}< 0$, we have  
\begin{align}
\frac{1}{2}\frac{d}{dt}
\|Z^1(t)\|_{L^2}^2
&\leqslant 
C_8(L_1,T,\|Q_0\|_{H^m})
\left(\mathcal{E}^{\mu,\nu}_1(W(t))^2
+\nu^4\mathcal{E}^{\mu,\nu}_1(W(t))
\right).
\nonumber
\end{align} 
On the other hand, it is now not difficult  to show 
\begin{align}
\frac{1}{2}\frac{d}{dt}
\|W(t)\|_{L^2}^2
&\leqslant 
C_{9}(L_1,T,\|Q_0\|_{H^m})
\left(\mathcal{E}^{\mu,\nu}_1(W(t))^2
+\nu^5\mathcal{E}^{\mu,\nu}_1(W(t))
\right). 
\nonumber
\end{align}
Combining them, we have 
\begin{align}
\frac{d}{dt}
\mathcal{E}^{\mu,\nu}_1(W(t))^2
&\leqslant 
C_{10}(L_1,T,\|Q_0\|_{H^m})
\left(\mathcal{E}^{\mu,\nu}_1(W(t))^2
+\nu^4\mathcal{E}^{\mu,\nu}_1(W(t))
\right). 
\label{eq:8317}
\end{align} 
The Gronwall inequality for \eqref{eq:8317} shows
\begin{align}
\mathcal{E}^{\mu,\nu}_1(W(t))
&\leqslant
C_{11}(L_1,T,\|Q_0\|_{H^m})
(\mathcal{E}^{\mu,\nu}_1(W(0))+\nu^4), 
\nonumber
\end{align}
and hence the equivalence \eqref{eq:H4Em} shows 
\begin{align}
\|W(t)\|_{H^1}
&\leqslant
C_{12}(L_1,T,\|Q_0\|_{H^m})
(\|W(0)\|_{H^1}+\nu^4)
\nonumber
\\
&=
C_{12}(L_1,T,\|Q_0\|_{H^m})
(\|Q^{\mu}_0-Q^{\nu}_0\|_{H^1}+\nu^4)
\label{eq:9031}
\end{align}
for any $t\in [0,T]$. 
This combined with the triangle inequality 
$\|Q^{\mu}_0-Q^{\nu}_0\|_{H^1}
\leqslant 
\|Q^{\mu}_0-Q_0\|_{H^1}
+\|Q_0-Q^{\nu}_0\|_{H^1}$ 
and \eqref{eq:6093} (where $\ell=m-1$) for $0<\mu\leqslant \nu<1$
implies 
\begin{align}
\|W\|_{C([0,T];H^1)}
&
\leqslant
C_{13}(L_1,T,\|Q_0\|_{H^m})(\mu^{m-1}+\nu^{m-1}+\nu^4)
\nonumber
\\
&
\leqslant 
2C_{13}(L_1,T,\|Q_0\|_{H^m})(\nu^{m-1}+\nu^4), 
\label{eq:90312}
\end{align}
which is the desired \eqref{eq:1cauchy}.
\vspace{0.5em}
\\
\underline{Proof of \eqref{eq:mcauchy}:}
\\
We estimate $\mathcal{E}^{\mu,\nu}_m(W(t))$ for $t\in [0,T]$. 
Recall again that 
\eqref{eq:B52} shows the existence of a positive constant $D_1=D_1(T,\|Q_0\|_{H^m})$ which is independent of 
$\mu$ and $\nu$ such that 
\begin{equation}
\|Q^{\mu}\|_{C([0,T];H^m)}
+
\|Q^{\nu}\|_{C([0,T];H^m)}
\leqslant 
D_1(T,\|Q_0\|_{H^m}).
\label{eq:9010} 
\end{equation}
The fact will be used hereafter to show \eqref{eq:mcauchy} 
sometimes without any comments. 
Other constants which are independent of  $\mu$ and $\nu$ 
will be denoted by $D_k=D_k(\cdot,\ldots,\cdot)$ for some integer $k=2,3,\ldots$. 
In addition, we use $s_{m,k}$ for an integer $k$ to denote 
a function of $(t,x)$ satisfying 
\begin{align}
\|s_{m,k}(t)\|_{L^2}
&\leqslant 
D_k(T,\|Q_0\|_{H^m})\|W(t)\|_{H^m} 
\quad 
\text{for any $t\in [0,T]$.}
\label{eq:9012}
\end{align} 
We use $s_{m,L,k}$ instead of $s_{m,k}$ 
only when the above $D_k$ depends also on $L$.  
\par  
By taking the difference between  
\eqref{eq:Ujm} for $Q^{\mu}$ and that for $Q^{\nu}$, 
\begin{align}
\p_t\p_x^mW_j
&=
\p_t\p_x^mQ_j^{\mu}-\p_t\p_x^mQ_j^{\nu}
\nonumber
\\
&=
\left\{
(-\mu^5+ia_j)\p_x^4+b_j\p_x^3+i\lambda_j\p_x^2
\right\}\p_x^mW_j
\nonumber
\\
&\quad
+(\nu^5-\mu^5)\p_x^4(\p_x^mQ_j^{\nu})
+I_m^{(1)}
+
I_m^{(2)}
+
I_m^{(3)}, 
\nonumber
\end{align}
where 
\begin{align}
I_m^{(1)}
&:=
\p_x^m\left(
F_j^1(Q^{\mu},\p_x^2Q^{\mu})
\right)
-
\p_x^m\left(
F_j^1(Q^{\nu},\p_x^2Q^{\nu})
\right), 
\nonumber
\\
I_m^{(2)}
&:=
\p_x^m\left(
F_j^2(Q^{\mu},\p_xQ^{\mu})
\right)
-
\p_x^m\left(
F_j^2(Q^{\nu},\p_xQ^{\nu})
\right), 
\nonumber
\\
I_m^{(3)}
&:=
\p_x^m\left(
F_j^3(Q^{\mu},\p_xQ^{\mu})
\right)
-
\p_x^m\left(
F_j^3(Q^{\nu},\p_xQ^{\nu})
\right).
\nonumber
\end{align} 
Noting \eqref{eq:F1m} with \eqref{eq:re1}, 
and using \eqref{eq:9010}, 
we deduce 
\begin{align}
I_m^{(1)}
&=
O\left(
g(Q^{\mu})|\p_x^2(\p_x^mW)|
\right)
+O\left(
|Q^{\mu}||\p_x(\p_x^mW)|
\right)
\nonumber
\\
&\quad
+O\left(
|W||\p_x^2(\p_x^mQ^{\nu})|
\right)
+
O\left(
(|\p_xW|+|W|)|\p_x(\p_x^mQ^{\nu})|
\right)
+s_{m,1}.
\nonumber
\end{align}
In the same way as above, it is not difficult to deduce  
\begin{align}
I_m^{(2)}
&=
O\left(
|Q^{\mu}||\p_x(\p_x^mW)|
\right)
+
O\left(
|W|
\p_x(\p_x^mQ^{\nu})|
\right)
+s_{m,2},
\nonumber
\\
I_m^{(3)}
&=s_{m,3}.
\nonumber
\end{align}
Combining them, we have 
\begin{align}
\p_t\p_x^mW_j
&=
\left\{
(-\mu^5+ia_j)\p_x^4+b_j\p_x^3+i\lambda_j\p_x^2
\right\}\p_x^mW_j
\nonumber
\\
&\quad 
+(\nu^5-\mu^5)\p_x^{m+4}Q_j^{\nu}
+
O\left(
g(Q^{\mu})|\p_x^{m+2}W|
\right)
\nonumber
\\
&\quad
+O\left(
|Q^{\mu}||\p_x^{m+1}W|
\right)
+O\left(
|W||\p_x^{m+2}Q^{\nu}|
\right)
\nonumber
\\
&\quad
+
O\left(
(|\p_xW|+|W|)|\p_x^{m+1}Q^{\nu}|
\right)
+s_{m,1}+s_{m,2}+s_{m,3}.
\label{eq:9011}
\end{align}
In the same way as above,  
we obtain  
\begin{align}
\p_t\p_x^{m-1}W_j
&=
\left\{
(-\mu^5+ia_j)\p_x^4+b_j\p_x^3+i\lambda_j\p_x^2
\right\}
\p_x^{m-1}W_j
\nonumber
\\
&\quad 
+(\nu^5-\mu^5)\p_x^{m+3}Q_j^{\nu}
+
O\left(
|Q^{\mu}||\p_x^{m+1}W|
\right)
\nonumber
\\
&\quad
+O\left(
|W||\p_x^{m+1}Q^{\nu}|
\right)
+s_{m,4}.
\nonumber
\end{align}
Hence, by
the almost 
same computation
to obtain \eqref{eq:91}
and \eqref{eq:8313}, 
we derive  
\begin{align}
&\dfrac{L}{4a_j}\Phi^{\mu}i\p_t\p_x^{m-1}W_j
\nonumber
\\
&=
\left\{
(-\mu^5+ia_j)\p_x^4+b_j\p_x^3+i\lambda_j\p_x^2
\right\}
 \left(
 \dfrac{L}{4a_j}\Phi^{\mu}i
 \p_x^{m-1}W_j
 \right)
 \nonumber
 \\
 &\quad
+\left(
\dfrac{\mu^5 i}{a_j}
+1
\right)
L\p_x\left\{g(Q^{\mu})
\p_x^{m+1}W_j\right\}
\nonumber
\\
&\quad
+L\,O(|Q^{\mu}||\p_x^{m+1}W|)
+L\,O(|W||\p_x^{m+1}Q^{\nu}|)
\nonumber
\\
&\quad
+(\nu^5-\mu^5)\dfrac{L}{4a_j}
O\left(\|Q^{\mu}(t)\|_{L^2}^2\right)
\p_x^{m+3}Q_j^{\nu}
   +s_{m,L,5}.
   \label{eq:9013}
\end{align}
The same computation to obtain 
\eqref{eq:101} and \eqref{eq:8314} shows 
\begin{align}
\left\|
\left(
\dfrac{L}{4a_j}(\p_t\Phi^{\mu})i\p_x^{m-1}W_j
\right)(t)
\right\|_{L^2}
&\leqslant 
C\, L\|W(t)\|_{H^m}. 
\label{eq:9014}
\end{align}
Combining \eqref{eq:9011}, 
\eqref{eq:9013}, and \eqref{eq:9014}, 
and using \eqref{eq:Z_j} for $k=m$, 
we deduce
\begin{align}
\p_tZ^m_j
&=
\left\{
(-\mu^5+ia_j)\p_x^4+b_j\p_x^3+i\lambda_j\p_x^2
\right\}
Z^m_j
+
O\left(
g(Q^{\mu})|\p_x^2Z^m |
\right)
\nonumber
\\
&\quad
+\left(
\dfrac{\mu^5 i}{a_j}
+1
\right)L\p_x\left\{g(Q^{\mu})
\p_xZ^m_j
\right\}
\nonumber
\\
&\quad
+(1+L)O\left(
|Q^{\mu}||\p_xZ^m|
\right)
+O\left(
|W||\p_x^{m+2}Q^{\nu}|
\right)
\nonumber
\\
&\quad
+(1+L)
O\left(
(|\p_xW|+|W|)|\p_x^{m+1}Q^{\nu}|
\right)
\nonumber
\\
&\quad
+(\nu^5-\mu^5)
\left\{
\p_x^{m+4}Q_j^{\nu}+
\dfrac{L}{4a_j}
O\left(\|Q^{\mu}(t)\|_{L^2}^2\right)
\p_x^{m+3}Q_j^{\nu}
\right\}
+s_{m,L,6}.  
\nonumber
\end{align}
Therefore, in the same way as we obtain \eqref{eq:125}, 
we use the integration by part, the Young inequality and  
\eqref{eq:H4Em} for $k=m$ to deduce  
\begin{align}
&\frac{1}{2}\dfrac{d}{dt}
\|Z^m(t)\|_{L^2}^2
\nonumber
\\
&\leqslant 
(C^{*}-L)
\int_{\RR}
g(Q^{\mu})|\p_xZ^m|^2\,dx
+D_{2}(L,T,\|Q_0\|_{H^m})\mathcal{E}^{\mu,\nu}_m(W(t))^2
\nonumber
\\
&\quad
+D_3(T,\|Q_0\|_{H^m})
\|W(t)\|_{L^{\infty}}
\|\p_x^{m+2}Q^{\nu}(t)\|_{L^2}
\|Z^m(t)\|_{L^2}
\nonumber
\\
&\quad
+D_4(L,T,\|Q_0\|_{H^m})
\|\p_xW(t)\|_{L^{2}}
\|\p_x^{m+1}Q^{\nu}(t)\|_{L^{\infty}}
\|Z^m(t)\|_{L^2}
\nonumber
\\
&\quad
+D_5(L,T,\|Q_0\|_{H^m})
\|W(t)\|_{L^{\infty}}
\|\p_x^{m+1}Q^{\nu}(t)\|_{L^2}
\|Z^m(t)\|_{L^2}
\nonumber
\\
&\quad
+(\nu^5-\mu^5)
\sum_{j=1}^{n}
\Re\int_{\RR}
\left\{
\p_x^{m+4}Q_j^{\nu}+
\dfrac{L}{4a_j}
O\left(\|Q^{\mu}(t)\|_{L^2}^2\right)
\p_x^{m+3}Q_j^{\nu}
\right\}
\overline{Z_j^m}
dx
\nonumber
\\
&\leqslant 
(C^{*}-L)
\int_{\RR}
g(Q^{\mu})|\p_xZ^m|^2\,dx
+D_{2}(L,T,\|Q_0\|_{H^m})\mathcal{E}^{\mu,\nu}_m(W(t))^2
\nonumber
\\
&\quad
+D_6(L,T,\|Q_0\|_{H^m})
\|W(t)\|_{H^1}
\|Q^{\nu}(t)\|_{H^{m+2}}
\|Z^m(t)\|_{L^2}
\nonumber
\\
&\quad
+D_7(L,T,\|Q_0\|_{H^m})
(\nu^5-\mu^5)
\|Q^{\nu}(t)\|_{H^{m+4}}
\|Z^m(t)\|_{L^2}.
\label{eq:9016}
\end{align} 
By the same reason as that we choose $L_0$ and $L_1$, 
we can take the constant $C^{*}>0$ independently of $L$ 
and hence can take a positive constant $L=L_2$ 
so that $C^{*}-L_2<0$. 
Furthermore, as the estimate \eqref{eq:B51} holds even when $m$ is replaced with $m+j$ for 
$j=1,2,\ldots$,  
\begin{align}
\|Q^{\nu}\|_{C([0,T];H^{m+j})}
&\leqslant 
D_8(T,\|Q_0\|_{H^4})
\|Q_0^{\nu}\|_{H^{m+j}}
\quad 
(j=1,2,\ldots).
\nonumber
\end{align}
Noting \eqref{eq:6092}, we see  
the left hand side of the above grows up as $\nu\downarrow 0$, that is, 
\begin{align}
\|Q^{\nu}\|_{C([0,T];H^{m+j})}
&\leqslant 
D_9(T,\|Q_0\|_{H^m})
\nu^{-j}
\quad 
(j=1,2,\ldots).
\label{eq:90162}
\end{align}
Combining \eqref{eq:90162} for $j=2$ and  
\eqref{eq:90312},  
we deduce
\begin{align}
&\|W(t)\|_{H^1}
\|Q^{\nu}(t)\|_{H^{m+2}}
\leqslant 
D_{10}(L_1, T,\|Q_0\|_{H^m})
(\nu^{(m-1)-2}+\nu^{4-2}).
\label{eq:9017}
\end{align}
In the same way, we apply \eqref{eq:90162} for $j=4$ to obtain 
\begin{align}
(\nu^5-\mu^5)
\|Q^{\nu}(t)\|_{H^{m+4}}
&\leqslant 
D_{11}(T,\|Q_0\|_{H^m})
(\nu^5-\mu^5)
\nu^{-4}
\nonumber
\\
&\leqslant 
D_{11}(T,\|Q_0\|_{H^m})
\nu.
\label{eq:90172}
\end{align}
Combining \eqref{eq:9016} with the above choice of $L=L_2$, 
\eqref{eq:9017}-\eqref{eq:90172}, 
\eqref{eq:H4Em} for $k=m$, 
and noting $0<\nu<1$,  
we get 
\begin{align}
\frac{d}{dt}\|Z^m(t)\|_{L^2}^2
\leqslant 
D_{12}
\left\{
\mathcal{E}^{\mu,\nu}_m(W(t))^2
+
(\nu^{m-3}+\nu)
\mathcal{E}^{\mu,\nu}_m(W(t))
\right\}
\nonumber
\end{align}
where $D_2=D_{12}(L_1,L_2,T,\|Q_0\|_{H^m})$. 
On the other hand, it is now easy to obtain  
\begin{align}
\frac{d}{dt}\|W(t)\|_{H^{m-1}}^2
\leqslant 
D_{13}
\left\{
\mathcal{E}^{\mu,\nu}_m(W(t))^2
+
(\nu^{m-2}+\nu^2)
\mathcal{E}^{\mu,\nu}_m(W(t))
\right\}
\nonumber 
\end{align}
where $D_{13}=D_{13}(L_1,L_2,T,\|Q_0\|_{H^m})$. 
The above two inequalities and $0<\nu<1$ shows  
\begin{align}
\frac{d}{dt}
\mathcal{E}^{\mu,\nu}_m(W(t))^2
\leqslant 
D_{14}
\left\{
\mathcal{E}^{\mu,\nu}_m(W(t))^2
+
(\nu^{m-3}+\nu)
\mathcal{E}^{\mu,\nu}_m(W(t))
\right\}
\nonumber 
\end{align}
where $D_{14}=D_{14}(L_1,L_2,T,\|Q_0\|_{H^m})$. 
The Gronwall inequality 
and \eqref{eq:H4Em}
shows 
$$
\|W\|_{C([0,T];H^m)}
\leqslant 
D_{15}(L_1,L_2,T,\|Q_0\|_{H^m})
\left(
\nu^{m-3}
 +\nu
+\|W(0)\|_{H^m}
\right), 
$$
which is the desired \eqref{eq:mcauchy}.
\end{proof}
\section{Proof of Theorem~\ref{theorem:lwp}}
\label{section:prooflw}
This section completes the proof of Theorem~\ref{theorem:lwp}.
\begin{proof}[Proof of Theorem~\ref{theorem:lwp}]
Let $m$ be an integer satisfying $m\geqslant 4$, 
and let $Q_0\in H^m(\RR;\mathbb{C}^n)$. 
From the time-reversibility of \eqref{eq:apde}, 
it suffices to solve \eqref{eq:apde}-\eqref{eq:adata} 
in positive time-direction. 
\vspace{0.5em}
\\
\underline{Local existence of a solution in $CH^m$:}
\\
Let $\left\{Q_0^{\alpha}\right\}_{\alpha\in (0,1)}$ be the 
Bona-Smith approximation of $Q_0$ introduced in 
Section~\ref{section:local}. 
For $\mu$ and $\nu$ with $0<\mu\leqslant \nu<1$, 
let $Q^{\mu}$ and $Q^{\nu}$ satisfy \eqref{eq:mapde}-\eqref{eq:madata} and 
\eqref{eq:napde}-\eqref{eq:nadata} respectively.
Let $T=T(\|Q_0\|_{H^4})>0$ be given by \eqref{eq:extime} independently of $\mu$ and $\nu$.
Combining \eqref{eq:mcauchy} in Proposition~\ref{proposition:cauchy} 
with the triangle inequality, 
the convergence $Q_0^{\alpha}\to Q_0$ in $H^m$ as $\alpha\downarrow 0$, 
and $m\geqslant 4$, 
we deduce 
\begin{align}
&\|Q^{\mu}-Q^{\nu}\|_{C([0,T];H^m)}
\nonumber
\\
&\leqslant 
C(T,\|Q_0\|_{H^m})\left(
\nu^{m-3}+\nu
+\|Q_0^{\mu}-Q_0\|_{H^m}
+\|Q_0-Q_0^{\nu}\|_{H^m}
\right)
\nonumber
\\
&\to 0 \quad (\mu, \nu\downarrow 0).
\nonumber
\end{align}
This shows that 
$\left\{Q^{\mu}\right\}_{\mu\in (0,1)}$
is Cauchy in $C([0,T];H^m(\RR;\mathbb{C}^n))$, 
and thus there exists its limit $Q:=\displaystyle\lim_{\mu \downarrow 0}Q^{\mu}$ in $C([0,T];H^m(\RR;\mathbb{C}^n))$. 
By the strong convergence, 
it is not difficult to prove that $Q$ is 
actually a solution to \eqref{eq:apde}-\eqref{eq:adata}. 
If we may add something, 
the proof of it is reduced to the justification of  
\begin{align}
&F_j(Q^{\mu},\p_xQ^{\mu}, \p_x^2Q^{\mu})\to F_j(Q,\p_xQ, \p_x^2Q)
\quad 
\text{as}
\quad 
\mu\downarrow 0
\nonumber
\end{align}
for each $j\in \left\{1,\ldots,n\right\}$ 
in the sense of distribution on $(0,T)\times \RR$. 
We omit the detail but 
demonstrate only the proof of 
 \begin{align}
 &F_j^3(Q^{\mu},\p_xQ^{\mu})\to F_j^3(Q,\p_xQ)
 \quad 
 \text{as}
 \quad 
 \mu\downarrow 0
 \label{eq:appb}
 \end{align}
 for readers who are interested in how to handle the nonlocal terms. 
  In fact, we can prove it in the sense of uniformly convergence on $[0,T]\times \RR$ 
  as follows: 
  By a simple calculation and the triangle inequality, 
  \begin{align}
  &
  \left|
  F_j^3(Q^{\mu},\p_xQ^{\mu})-F_j^3(Q,\p_xQ)
  \right|(t,x)
  \nonumber
  \\
  &\leqslant 
  \sum_{r=1}^{n}
            \left(\int_{\RR}
              \left|
              F_{j,r}^{3,A}(Q^{\mu},\p_xQ^{\mu})
              -
              F_{j,r}^{3,A}(Q,\p_xQ)
              \right|
              (t,y)
              dy\right)
               \left|F_{j,r}^{3,B}(Q^{\mu})(t,x)
               \right|
 \nonumber
 \\
 &\quad+
 \sum_{r=1}^{n}
            \left(\int_{\RR}
            \left|
              F_{j,r}^{3,A}(Q,\p_xQ)
              \right|
              (t,y)
              dy\right)
               \left|\left(F_{j,r}^{3,B}(Q^{\mu})-F_{j,r}^{3,B}(Q)\right)(t,x)
               \right|.
 \nonumber
  \end{align}
  Since $F_{j,r}^{3,A}$ and $F_{j,r}^{3,B}$ satisfy (F3) with 
  \eqref{eq:F31} and \eqref{eq:F32}, 
  it follows from the 
  Schwarz inequality and 
  the Sobolev embedding 
  \begin{align}
 &\int_{\RR}
             \left|
               F_{j,r}^{3,A}(Q,\p_xQ)
               \right|
               (t,y)
               dy
 \nonumber
 \\
 &\leqslant
 C\left(
 \sum_{p_3=0}^{d_3}\|Q(t)\|_{L^{\infty}}^{p_3}
 +
 \sum_{p_4=0}^{d_4}\|\p_xQ(t)\|_{L^{\infty}}^{p_4}
 \right)
 \int_{\RR}
 \left(
 |Q|^2+|\p_xQ|^2
 \right)
 (t,y)
 \,dy
 \nonumber
 \\
 &\leqslant 
 C
 \sum_{\ell=2}^{d_3+d_4+2}
 \|Q\|_{C([0,T];H^2)}^{\ell}, 
 \nonumber
  \end{align} 
  \begin{align}
  &\int_{\RR}
               \left|
               F_{j,r}^{3,A}(Q^{\mu},\p_xQ^{\mu})
               -
               F_{j,r}^{3,A}(Q,\p_xQ)
               \right|
               (t,y)
               dy
  \nonumber
  \\
  &\leqslant 
  C\sum_{p_3=0}^{d_3}
  \int_{\RR}
  \left(
  (
  |Q^{\mu}|^{1+p_3}
  +|Q|^{1+p_3}
  )
  |Q^{\mu}-Q|
  \right)
  (t,y)
  \,dy
  \nonumber
  \\
  &\quad 
  +C\sum_{p_4=0}^{d_4}
  \int_{\RR}
   \left(
   (
   |\p_xQ^{\mu}|^{1+p_4}
   +|\p_xQ|^{1+p_4}
   )
   |\p_xQ^{\mu}-\p_xQ|
   \right)
   (t,y)
   \,dy
   \nonumber
  \\
  &\leqslant 
  C
  \sum_{p_3=0}^{d_3}
  \left(
  \|Q^{\mu}(t)\|_{L^{\infty}}^{p_3}
  +
  \|Q(t)\|_{L^{\infty}}^{p_3}
  \right)
  \left(
  \|Q^{\mu}(t)\|_{L^2}
  +
  \|Q(t)\|_{L^2}
  \right)
  \|(Q^{\mu}-Q)(t)\|_{L^2}
  \nonumber
  \\
  &\quad 
  +
  C\sum_{p_4=0}^{d_4}
  \left(
   \|\p_xQ^{\mu}(t)\|_{L^{\infty}}^{p_4}
   +
   \|\p_xQ(t)\|_{L^{\infty}}^{p_4}
   \right)
   \nonumber
   \\
   &\qquad \qquad \qquad \times 
   \left(
   \|\p_xQ^{\mu}(t)\|_{L^2}
   +
   \|\p_xQ(t)\|_{L^2}
   \right)
   \|(\p_xQ^{\mu}-\p_xQ)(t)\|_{L^2}
   \nonumber
   \\
    &\leqslant 
    C\sum_{p_3=0}^{d_3}\left(
    \|Q^{\mu}(t)\|_{H^1}^{1+p_3}
    +
    \|Q(t)\|_{H^1}^{1+p_3}
    \right)
    \|(Q^{\mu}-Q)(t)\|_{L^2}
    \nonumber
    \\
    &\quad 
    +
    C\sum_{p_4=0}^{d_4}\left(
     \|\p_xQ^{\mu}(t)\|_{H^1}^{1+p_4}
     +
     \|\p_xQ(t)\|_{H^1}^{1+p_4}
     \right)
     \|(\p_xQ^{\mu}-\p_xQ)(t)\|_{L^2}
     \nonumber
 \\
 &\leqslant 
 C\sum_{\ell=1}^{d_3+d_4+1}
 \left(
 \sup_{\mu\in (0,1)}\|Q^{\mu}\|_{C([0,T];H^2)}^{\ell}
    +
    \|Q\|_{C([0,T];H^2)}^{\ell}
 \right)
 \|Q^{\mu}-Q\|_{C([0,T];H^1)}, 
 \nonumber
  \end{align}
  \begin{align}
  \left|
  F_{j,r}^{3,B}(Q^{\mu})(t,x)
  \right|
  &\leqslant 
  C
 \sum_{p_5=0}^{d_5} 
  \|Q^{\mu}(t)\|_{L^{\infty}}^{1+p_5}
  \leqslant 
  C\sum_{p_5=0}^{d_5} 
  \sup_{\mu\in (0,1)}\|Q^{\mu}\|_{C([0,T];H^1)}^{1+p_5}, 
  \nonumber
  \end{align}
  \begin{align}
 &\left|
 \left(F_{j,r}^{3,B}(Q^{\mu})-F_{j,r}^{3,B}(Q)\right)(t,x)
 \right|
 \nonumber
 \\
 &\leqslant 
 C\sum_{p_5=0}^{d_5}\left(
     \|Q^{\mu}(t)\|_{L^{\infty}}^{p_5}
     +
     \|Q(t)\|_{L^{\infty}}^{p_5}
     \right)
     \|(Q^{\mu}-Q)(t)\|_{L^{\infty}}
 \nonumber
 \\
 &\leqslant 
 C\sum_{p_5=0}^{d_5}
 \left(
 \sup_{\mu\in (0,1)}\|Q^{\mu}\|_{C([0,T];H^1)}^{p_5}
 +
 \|Q\|_{C([0,T];H^1)}^{p_5}
 \right)
 \|Q^{\mu}-Q\|_{C([0,T];H^1)}. 
 \nonumber
  \end{align}
  Combining them, we obtain 
  \begin{align}
   &
  \sup_{(t,x)\in [0,T]\times \RR} \left|
   F_j^3(Q^{\mu},\p_xQ^{\mu})-F_j^3(Q,\p_xQ)
   \right|(t,x)
   \nonumber
   \\
   &\leqslant 
 C
 \sum_{\ell=1}^{d_3+d_4+1}
 \sum_{p_5=0}^{d_5}
 M_{T}^{1+p_5}
 \left(M_{T}^{\ell}
 +\|Q\|_{C([0,T];H^2)}^{\ell}
 \right)
 \|Q^{\mu}-Q\|_{C([0,T];H^1)}
 \nonumber
 \\&\quad
 +
 C\sum_{\ell=2}^{d_3+d_4+2}
 \sum_{p_5=0}^{d_5}
 \|Q\|_{C([0,T];H^2)}^{\ell}
 \left(
 M_{T}^{p_5}
 +\|Q\|_{C([0,T];H^1)}^{p_5}
 \right)
 \|Q^{\mu}-Q\|_{C([0,T];H^1)}, 
 \nonumber
  \end{align}
  where $M_{T}:=\displaystyle\sup_{\mu\in (0,1)}\|Q^{\mu}\|_{C([0,T];H^2)}$. 
 Since $\left\{Q^{\mu}\right\}_{\mu\in (0,1)}$ converges to $Q$ 
 and is bounded in $C([0,T];H^m(\RR;\mathbb{C}^n))$, it follows that 
 $\|Q^{\mu}-Q\|_{C([0,T];H^1)}\to 0$ as $\mu\downarrow 0$ and  $M_{T}<\infty$. 
 This implies the desired convergence
 $$
 \sup_{(t,x)\in [0,T]\times \RR} \left|
   F_j^3(Q^{\mu},\p_xQ^{\mu})-F_j^3(Q,\p_xQ)
   \right|(t,x)\to 0
   \quad
   \text{as}
   \quad \mu\downarrow 0.
 $$
 \vspace{-0.3em}
\\
\underline{Uniqueness of the solution:}
\\ 
Let $Q^1, Q^2\in C([0,T];H^4(\RR;\mathbb{C}^n))
$
be solutions to \eqref{eq:apde} with 
$Q^1(0,x)=Q^2(0,x)$.
Then $Q^1, Q^2\in C^1([0,T];L^2(\RR;\mathbb{C}^n))$. 
Set 
$\widetilde{W}={}^t(\widetilde{W_1},\ldots,\widetilde{W_n}):=Q^1-Q^2$.
It suffices to show $\widetilde{W}=0$. 
For this purpose,  
we introduce 
$\widetilde{Z^1}={}^t(\widetilde{Z^{1}_1},\ldots,\widetilde{Z^{1}_n})
$ 
and $\mathcal{E}(\widetilde{W}(t))$, 
where
\begin{align}
 &\widetilde{Z^{1}_j}=\widetilde{Z^{1}_j}(t,x)
 :=\p_x\widetilde{W_j}(t,x)+\dfrac{L}{4a_j}\Phi^{1}(t,x)i\widetilde{W_j}(t,x)
 \quad (j\in \left\{1,\ldots,n\right\}), 
 \label{eq:Z_12}
\\
 &\Phi^{1}=\Phi^{1}(t,x)
 :=\int_{-\infty}^{x}g(Q^{1}(t,y))\,dy
 \left(
 =\int_{-\infty}^{x}|Q^{1}(t,y)|^2\,dy
 \right), 
 \label{eq:Phi1}
 \\
 &\mathcal{E}(\widetilde{W}(t))^2
 :=
 \|\widetilde{Z^1}(t)\|_{L^2}^2+\|\widetilde{W}(t)\|_{L^2}^2, 
 \nonumber
 \end{align} 
 and $L>1$ is again a constant which will be taken later. 
 The argument below is formally the same as that we obtain \eqref{eq:8315} and \eqref{eq:8317} 
 under the setting $\mu=\nu=0$ and the modification of $(Q^{\mu}, Q^{\nu})$ with 
 $(Q^1,Q^2)$. 
 We can make it rigorous by taking the regularity of $Q^1$ and $Q^2$ into account: 
 Since 
 \begin{align}
 &\widetilde{Z^{1}}\in 
 C([0,T];H^3(\RR;\mathbb{C}^n))
 \cap C^1([0,T];H^{-1}(\RR;\mathbb{C}^n)),
 \label{eq:328}
 \end{align}
the following holds in the sense of distribution on $(0,T)$: 
\begin{align}
\frac{d}{dt}
\|\widetilde{Z^1}(t)\|_{L^2}^2
&=
2\operatorname{Re}
\sum_{j=1}^n
\left\langle
\p_t \widetilde{Z^1_j}(t), \widetilde{Z^1_j}(t)
\right\rangle_{H^{-1},H^1}, 
\label{eq:du3281}
\end{align} 
where $\left\langle\cdot, \cdot\right\rangle_{H^{-1},H^1}$ 
denotes the duality paring for $H^{-1}(\RR;\mathbb{C})$ and 
$H^{1}(\RR;\mathbb{C})$. 
Since  $Q^1\in C([0,T];H^4(\RR;\mathbb{C}^n))
\cap C^1([0,T];L^2(\RR;\mathbb{C}^n))$,  
$$
\p_t\Phi^{1}(t,x)
=
2\operatorname{Re}
\int_{-\infty}^x
\dfrac{\p Q^{1}}{\p t}(t,y)\cdot  Q^{1}(t,y)\,dy
$$   
holds for any $(t,x)\in (0,T)\times \RR$.
Moreover, it follows that 
\begin{align}
 &
 \p_t\left(
 \dfrac{L}{4a_j}\Phi^{1}i\widetilde{W}_j
 \right)
 =
 \dfrac{L}{4a_j}\Phi^{1}i\p_t\widetilde{W}_j
 +
 \dfrac{L}{4a_j}(\p_t\Phi^{1})i\widetilde{W}_j
 \quad 
 \text{in}
 \quad 
 C([0,T];L^2(\RR;\mathbb{C}))
\nonumber
 \end{align}
 for $j\in \left\{1,\ldots,n\right\}$. 
Noting them and \eqref{eq:328}, 
we deduce 
\begin{align}
\p_t\widetilde{Z^{1}_j}
&=
\left(ia_j\p_x^4+b_j\p_x^3+i\lambda_j\p_x^2
\right)\widetilde{Z^{1}_j}
+
R_j
\quad 
\text{in}
\quad 
C([0,T];H^{-1}(\RR;\mathbb{C}))
\label{eq:du3282}
\end{align}
for $j\in \left\{1,\ldots,n\right\}$, 
where 
\begin{align}
&R_j
=O\left(g(Q^{1})|\p_x^2\widetilde{Z^1}|\right)
+L\p_x\left\{g(Q^{1})
\p_x\widetilde{Z_j^1}\right\}
+
O\left(
|Q^{1}||\p_x\widetilde{Z^1}|
\right)
+r_j, 
\nonumber
\\ 
&\|r_j(t)\|_{L^2}
\leqslant 
C(\|Q^{1}\|_{C([0,T];H^4)}+\|Q^{2}\|_{C([0,T];H^3)})
 \|\widetilde{W}(t)\|_{H^1}.
\nonumber
\end{align}
In fact, \eqref{eq:328} shows $R_j\in C([0,T];L^2(\RR;\mathbb{C}))$, 
and thus $\langle R_j(t), \widetilde{Z^{1}_j}(t)\rangle_{H^{-1},H^1}$ is just their 
$L^2$-product.  
Noting them and using \eqref{eq:du3281}-\eqref{eq:du3282}, 
we can take a sufficiently large $L$ so that 
\begin{align}
\frac{d}{dt}
\|\widetilde{Z^1}(t)\|_{L^2}^2
&\leqslant 
A_{L}(\|Q^{1}(t)\|_{H^4}+\|Q^{2}(t)\|_{H^3})
\mathcal{E}(\widetilde{W}(t))^2,  
\nonumber
\end{align}
where $A_L(\cdot)$ is a positive-valued increasing function on $[0,\infty)$ which depends on $L$.
This estimate combined with that for the time-derivative of
$\|\widetilde{W}(t)\|_{L^2}^2$ implies 
\begin{align}
\frac{d}{dt}
\mathcal{E}(\widetilde{W}(t))^2
&\leqslant 
C(L, \|Q^{1}\|_{C([0,T];H^4)}+\|Q^{2}\|_{C([0,T];H^3)})
\mathcal{E}(\widetilde{W}(t))^2.
\nonumber
\end{align}
Hence, the Gronwall inequality and $Q^1(0,x)=Q^2(0,x)$  
shows $\mathcal{E}(\widetilde{W}(t))=0$ for any $t\in [0,T]$. 
This implies $\widetilde{W}=0$ on $[0,T]\times \RR$, which is the desired result. 
\vspace{0.5em}
\\
\underline{Continuous dependence:}
\\
Let $Q\in C([0,T(\|Q_0\|_{H^4})];H^m(\RR;\mathbb{C}^n))$ be the unique solution 
to \eqref{eq:apde} with $Q(0,\cdot)=Q_0\in H^m(\RR;\mathbb{C}^n)$
constructed above. 
Fix $T^{\prime}\in (0,T(\|Q_0\|_{H^4})$.
Let $\eta>0$ be any given. 
We take $\delta>0$ (which will be retaken 
sufficiently small later) and 
$\widetilde{Q_0}\in H^m(\RR;\mathbb{C}^n)$ to satisfy 
$\|Q_0-\widetilde{Q}_0\|_{H^m}<\delta$.  
We denote 
the solution to \eqref{eq:apde} with $\widetilde{Q}(0,\cdot)=\widetilde{Q}_0$ 
by $\widetilde{Q}\in C([0,T(\|\widetilde{Q}_0\|_{H^4})];H^m(\RR;\mathbb{C}^n))$.
Moreover, let $Q_{0}^{\alpha}$ and 
$\widetilde{Q}_{0}^{\alpha}$ for each $\alpha\in (0,1)$ be defined to form   
Bona-Smith approximations of $Q_0$ and $\widetilde{Q}_0$ respectively, 
and let 
$Q^{\alpha}$ 
and $\widetilde{Q}^{\alpha}$
be regularized solutions to \eqref{eq:bpde} (for $\ep=\alpha$) with  
$Q^{\alpha}(0,\cdot)=Q_{0}^{\alpha}$ and $\widetilde{Q}^{\alpha}(0,\cdot)=\widetilde{Q}_{0}^{\alpha}$ respectively. 
In view of the lower-semicontinuity for $T=T(\|Q_0\|_{H^4})$ given by \eqref{eq:extime}
with respect to $Q_0$, there exists a sufficiently small $0<\delta_1<1$ such that 
$Q, \widetilde{Q}, Q^{\alpha}, \widetilde{Q}^{\alpha}$ exist commonly at least 
on $[0,T^{\prime}]$ if $\delta$ satisfies $0<\delta<\delta_1$. In what follows, we fix such $\delta_1$ and 
assume $0<\delta<\delta_1<1$. 
\par 
We estimate $\|Q-\widetilde{Q}\|_{C([0,T^{\prime}];H^m)}$.  
By the triangle inequality, 
\begin{align}
\|Q-\widetilde{Q}\|_{C([0,T^{\prime}];H^m)}
&\leqslant 
\|Q-Q^{\alpha}\|_{C([0,T^{\prime}];H^m)}
+
\|Q^{\alpha}-\widetilde{Q}^{\alpha}\|_{C([0,T^{\prime}];H^m)}
\nonumber
\\
&\quad
+
\|\widetilde{Q}^{\alpha}-\widetilde{Q}\|_{C([0,T^{\prime}];H^m)}.
\nonumber
\end{align}
\begin{proposition}
\label{proposition:cdprop} 
Let $\alpha\in (0,1)$. 
There exists a constant $C=C(T,\|Q_0\|_{H^m})>1$ 
which depends on $T$ and $\|Q_0\|_{H^m}$ but is independent of
$\alpha$ such that 
\begin{align}
&\|Q-Q^{\alpha}\|_{C([0,T^{\prime}];H^m)}
\leqslant 
C(\alpha^{m-3}+\alpha+\|Q_0-Q_0^{\alpha}\|_{H^m}),
\label{eq:6100}
\\
&\|\widetilde{Q}^{\alpha}-\widetilde{Q}\|_{C([0,T^{\prime}];H^m)}
\leqslant 
C(\alpha^{m-3}+\alpha+\|\widetilde{Q}_0^{\alpha}-\widetilde{Q}_0\|_{H^m}),
\label{eq:6101}
\\
&\|Q^{\alpha}-\widetilde{Q}^{\alpha}\|_{C([0,T^{\prime}];H^m)}
\leqslant 
C(\alpha^{m-3}+\alpha^{-2}\|Q_0-\widetilde{Q}_0\|_{H^1}
+\|Q_0^{\alpha}-\widetilde{Q}_0^{\alpha}\|_{H^m}).
\label{eq:6102}
\end{align}
\end{proposition}
\begin{proof}[Proof of Proposition~\ref{proposition:cdprop}]
Let $\mu, \nu$ be positive parameters satisfying  $0<\mu\leqslant \nu<1$. 
By \eqref{eq:mcauchy}, 
the following holds: 
\begin{align}
\|Q^{\mu}-Q^{\nu}\|_{C([0,T^{\prime}];H^m)}
&\leqslant 
C(T,\|Q_0\|_{H^m})\left(
\nu^{m-3}+\nu
+\|Q_0^{\mu}-Q_0^{\nu}\|_{H^m}
\right).
\label{eq:2mcauchy}
\end{align}
\par
The estimate \eqref{eq:6100} is obtained by 
fixing $\nu=\alpha$ and by 
passing the limit  as $\mu\downarrow 0$ in \eqref{eq:2mcauchy}, 
where we use 
$Q^{\mu}_0\to Q_0$ in $H^m(\RR;\mathbb{C}^n)$ and  
$Q^{\mu}\to Q$ in $C([0,T^{\prime}];H^m(\RR;\mathbb{C}^n))$ 
as $\mu\downarrow 0$. 
\par
The estimate \eqref{eq:6101} is obtained in the same manner: 
If we show \eqref{eq:6100} for 
$\widetilde{Q}$, $\widetilde{Q}^{\alpha}$, $\widetilde{Q}_0$, $\widetilde{Q}_0^{\alpha}$
in place for $Q,Q^{\alpha},Q_0,Q_0^{\alpha}$ respectively, then it reads 
$$
\|\widetilde{Q}^{\alpha}-\widetilde{Q}\|_{C([0,T^{\prime}];H^m)}
\leqslant 
C(\alpha^{m-3}+\alpha+\|\widetilde{Q}_0^{\alpha}-\widetilde{Q}_0\|_{H^m}),
$$
where $C=C(T, \|\widetilde{Q}_0\|_{H^m})>1$. 
Recalling $\|Q_0-\widetilde{Q}_0\|_{H^m}<\delta\leqslant 1$, 
we can retake a larger constant $C$ which depends on $T$ 
and $\|Q_0\|_{H^m}$.
\par 
The estimate \eqref{eq:6102} follows from a similar argument to obtain 
\eqref{eq:mcauchy} and \eqref{eq:6100} with slight modification. 
The difference of \eqref{eq:6100} and \eqref{eq:6102}
in their right hand side  comes from
the estimate for
$Q^{\alpha}-\widetilde{Q}^{\alpha}$ in $CH^1$:  
To be more precise, 
a similar argument to obtain \eqref{eq:9031} 
(but without handling the terms with coefficient $\nu^5-\mu^5$)
yields
$$
\|Q^{\alpha}-\widetilde{Q}^{\alpha}\|_{C([0,T^{\prime}];H^1)}
\leqslant 
C_1(T, \|Q_0\|_{H^m}, \|\widetilde{Q}_0\|_{H^m})
\|Q^{\alpha}_0-\widetilde{Q}^{\alpha}_0\|_{H^1}.
$$ 
From the triangle inequality and \eqref{eq:6093} with $\ell=m-1$, 
it follows that  
\begin{align}
\|Q^{\alpha}_0-\widetilde{Q}^{\alpha}_0\|_{H^1}
&\leqslant 
\|Q^{\alpha}_0-Q_0\|_{H^1}
+\|Q_0-\widetilde{Q}_0\|_{H^1}
+\|\widetilde{Q}_0-\widetilde{Q}^{\alpha}_0\|_{H^1}
\nonumber
\\
&\leqslant 
C_2\alpha^{m-1}(\|Q_0\|_{H^m}+\|\widetilde{Q}_0\|_{H^m})
+\|Q_0-\widetilde{Q}_0\|_{H^1}, 
\nonumber
\end{align}
where the constant $C_2>0$ is also independent of $\alpha$. 
Combining them and $\|\widetilde{Q}_0\|_{H^m}\leqslant \|Q_0\|_{H^m}+1$, we see that there exists a positive constant 
$C_3=C_3(T, \|Q_0\|_{H^m})$ which is independent of $\alpha$ such that
 \begin{align}
 \|Q^{\alpha}-\widetilde{Q}^{\alpha}\|_{C([0,T^{\prime}];H^1)}
 &\leqslant 
 C_3(T, \|Q_0\|_{H^m})
 (\alpha^{m-1}+\|Q_0-\widetilde{Q}_0\|_{H^1}).
 \label{eq:add1120}
 \end{align}
It is then straightforward to derive \eqref{eq:6102} by using \eqref{eq:add1120} 
in the same way as we obtain 
\eqref{eq:mcauchy} by using \eqref{eq:1cauchy} (or \eqref{eq:90312}), where
the key procedure involves the following estimate 
\begin{align}
&
\|(Q^{\alpha}-\widetilde{Q}^{\alpha})(t)\|_{H^1}
\|\widetilde{Q}^{\alpha}(t)\|_{H^{m+2}}
\nonumber
\\
&\leqslant 
C_4(T, \|Q_0\|_{H^m}, \|\widetilde{Q_0}\|_{H^m})
 (\alpha^{m-1}+\|Q_0-\widetilde{Q}_0\|_{H^1})\alpha^{-2}
 \nonumber
 \\
 &\leqslant 
 C_5(T, \|Q_0\|_{H^m}) (\alpha^{m-3}+\alpha^{-2}\|Q_0-\widetilde{Q}_0\|_{H^1}), 
 \nonumber
\end{align}
which corresponds to the part \eqref{eq:9017} to obtain \eqref{eq:mcauchy}. 
The difference between  the  above and \eqref{eq:9017} 
affects the right hand side of \eqref{eq:6102}. 
We omit the detail for the other parts.
 \end{proof}
\par 
Furthermore, by the triangle inequality and \eqref{eq:bs1}, 
\begin{align}
\|Q_0^{\alpha}-\widetilde{Q}_0^{\alpha}\|_{H^m}
&=
\|(Q_0-\widetilde{Q}_0)^{\alpha}\|_{H^m}
\leqslant 
\|Q_0-\widetilde{Q}_0\|_{H^m}, 
\label{eq:9032}
\\
\|\widetilde{Q}_0^{\alpha}-\widetilde{Q}_0\|_{H^m}
&\leqslant 
\|\widetilde{Q}_0^{\alpha}-Q_0^{\alpha}\|_{H^m}
+
\|Q_0^{\alpha}-Q_0\|_{H^m}
+
\|Q_0-\widetilde{Q}_0\|_{H^m}
\nonumber
\\
&\leqslant 
\|Q_0^{\alpha}-Q_0\|_{H^m}
+
2\|Q_0-\widetilde{Q}_0\|_{H^m}.
\label{eq:9033}
\end{align}
Gathering \eqref{eq:6100}-\eqref{eq:6102} and \eqref{eq:9032}-\eqref{eq:9033}, 
we deduce 
\begin{align}
&\|Q-\widetilde{Q}\|_{C([0,T^{\prime}];H^m)}
\nonumber
\\
&\leqslant 
C(T,\|Q_0\|_{H^m})
\nonumber
\\
&\quad \times 
\left\{
3\alpha^{m-3}+2\alpha+(3+\alpha^{-2})\|Q_0-\widetilde{Q}_0\|_{H^m}
+2\|Q_0-Q_0^{\alpha}\|_{H^m}
\right\}.
\nonumber
\end{align}
Since $m\geqslant 4$ and   
$Q_0^{\alpha}\to Q_0$ in $H^m(\RR;\mathbb{C}^n)$ as $\alpha \downarrow 0$, 
we can take a sufficiently small $0<\alpha_0<1$
such that
$$
C(T,\|Q_0\|_{H^m})(3\alpha^{m-3}+2\alpha)< \eta, \quad
2C(T,\|Q_0\|_{H^m})\|Q_0-Q_0^{\alpha}\|_{H^m}< \eta
$$
for any $\alpha\in (0,\alpha_0]$. 
By fixing $\alpha=\alpha_0$, we have 
\begin{align}
&\|Q-\widetilde{Q}\|_{C([0,T^{\prime}];H^m)}
<
2\eta
+
C(T,\|Q_0\|_{H^m})(3+(\alpha_0)^{-2})\|Q_0-\widetilde{Q}_0\|_{H^m}.
\nonumber
\end{align}
Then we take a $\delta_2\in (0,\delta_1)$ so that 
$$
C(T,\|Q_0\|_{H^m})(3+(\alpha_0)^{-2})\delta_2<\eta.
$$
This shows $\|Q-\widetilde{Q}\|_{C([0,T^{\prime}];H^m)}<3\eta$ for any $\delta\in (0,\delta_2)$. 
Note that $\delta_2>0$ is decided to depend on $\eta$ and $Q_0$ 
but not on $\widetilde{Q}_0$, since so is $\alpha_0$.
This completes the proof of the continuous dependence.
\end{proof}
%
%
\appendix
\section{Proof of Proposition~\ref{prop:linear}}
 \label{section:Appendix}
We state the proof of Proposition~\ref{prop:linear}. 
Our proof follows that of Proposition~2.1 in \cite{CO2}, 
and mostly use the same notation in \cite[Section~2]{CO2} 
for readability to see points to change.
\begin{proof}[Proof of Proposition~\ref{prop:linear}]
We give only the outline of the energy estimates. 
\par 
We introduce a pseudodifferential operator 
$\Lambda=I+\tilde{\Lambda}$ of order zero. 
Here,   
$I$ is the identity operator and the symbol of $\tilde{\Lambda}$ is 
given by 
$$
\tilde{\lambda}(x,\xi)
=
\Phi(x)\frac{\varphi(\xi)}{4a\xi},
$$
where 
$$
\Phi(x)
=
L
\int_0^x
\phi(y)
dy, 
\quad 
\phi=\phi_A+|\phi_B|^2, 
$$
$L>3$ is a constant, 
$\varphi(\xi)\in C^{\infty}(\RR)$ is taken to be 
a real-valued even function which satisfies 
$$
\varphi(\xi)=1\ \  (|\xi|\geqslant r+1), \quad 
\varphi(\xi)=0\ \ (|\xi|\leqslant r), 
$$ 
and $r>0$ is a sufficiently large constant so that  
$\Lambda$ is an automorphism on $L^2(\RR;\mathbb{C})$. 
Compared with the setting of $\Lambda$ used in \cite[Section~2]{CO2}, 
the definition of $\Phi(x)$ is slightly changed by considering 
\eqref{eq:addd3292} and \eqref{eq:addd3293}, and   
$\varphi(\xi)$ is explicitly mentioned to be a real-valued even function, 
which implies  
$\overline{\tilde{\lambda}(x,-\xi)}=-\tilde{\lambda}(x,\xi)$ 
and hence $\overline{\tilde{\Lambda}v}=-\tilde{\Lambda}\overline{v}$.
\par 
Let $u$ be a solution to \eqref{eq:pde1}, 
and set $v=\Lambda u$. Moreover, set $D_x=-i\p_x$.
Here we denote by $\mathscr{L}$ 
the set of all $L^2$-bounded operators on $\RR$. 
In what follows, different positive constants are denoted by the same $C$, 
and different operators in $C(\RR;\mathscr{L})$  
are denoted by the same $P(t)$.
Then, we deduce 
\begin{align}
\Lambda \p_tu&=\p_tv, 
\nonumber
\\
\Lambda ia\p_x^4u
&=
ia\p_x^4v
+ia\left[\tilde{\Lambda}, D_x^4\right]v
-ia\left[\tilde{\Lambda},D_x^4\right]\tilde{\Lambda}v+P(t)v, 
\label{eq:c4}
\\
\Lambda b\p_x^3u
&=
b\p_x^3v-ib
\left[\tilde{\Lambda},D_x^3\right]v
+P(t)v, 
\label{eq:c3}
\\
\Lambda 
i\p_x\left\{
\beta_1(t,x)\p_xu
\right\}
&=
i\p_x\left\{
\beta_1(t,x)\p_xv
\right\}
+P(t)v, 
\nonumber
\\
\Lambda 
i\p_x\left\{
\beta_2(t,x)\overline{\p_xu}
\right\}
&=
i\p_x\left\{
\beta_2(t,x)
\overline{\p_xv}
\right\}
-2i\beta_2(t,x)D_x^2\tilde{\Lambda}\overline{v}
+P(t)\overline{v}, 
\label{eq:c22}
\\
\Lambda \gamma_1(t,x)\p_xu
&=
\gamma_1(t,x)\p_xv+P(t)v, 
\nonumber 
\\
\Lambda \gamma_2(t,x)\overline{\p_xu}
&=
\gamma_2(t,x)\overline{\p_xv}+P(t)\overline{v}, 
\nonumber 
\end{align}
We here check only \eqref{eq:c3} and \eqref{eq:c22}, 
because they are not handled  in \cite[Section~2]{CO2} 
and because the effect of $\overline{\tilde{\Lambda}v}\neq \tilde{\Lambda}\overline{v}$ appears in \eqref{eq:c22}. 
The equality \eqref{eq:c4} is shown in \cite[Section~2]{CO2} 
and other relations are not difficult to be checked by the same argument. 
For \eqref{eq:c3}, we deduce  
\begin{align}
\Lambda b\p_x^3u
&=
-ib(I+\tilde{\Lambda}) 
D_x^3(I-\tilde{\Lambda}+\tilde{\Lambda}^2-\tilde{\Lambda}^3+\cdots)v
\nonumber
\\
&=
-ibD_x^3(I-\tilde{\Lambda}+\tilde{\Lambda}^2-\tilde{\Lambda}^3+\cdots)v
\nonumber
\\
&\quad 
-ib\tilde{\Lambda}D_x^3(I-\tilde{\Lambda}+\tilde{\Lambda}^2-\tilde{\Lambda}^3+\cdots)v
\nonumber
\\&=
-ibD_x^3v+ibD_x^3\tilde{\Lambda}(I-\tilde{\Lambda}+\tilde{\Lambda}^2-\tilde{\Lambda}^3+\cdots)v
\nonumber
\\
&\quad 
-ib\tilde{\Lambda}D_x^3(I-\tilde{\Lambda}+\tilde{\Lambda}^2-\tilde{\Lambda}^3+\cdots)v
\nonumber
\\&=
b\p_x^3v-ib
\left[\tilde{\Lambda},D_x^3\right]
(I-\tilde{\Lambda}+\tilde{\Lambda}^2-\tilde{\Lambda}^3+\cdots)v
\nonumber
\\
&=
b\p_x^3v-ib
\left[\tilde{\Lambda},D_x^3\right]v
+P(t)v. 
\nonumber
\end{align}
For \eqref{eq:c22}, noting $\overline{\tilde{\Lambda}v}=-\tilde{\Lambda}\overline{v}$, 
we deduce 
\begin{align}
\Lambda i\beta_2(t,x)
\p_x^2\overline{u}
&=
-i(I+\tilde{\Lambda})\beta_2(t,x)D_x^2
\overline{
(I-\tilde{\Lambda}+\tilde{\Lambda}^2-\tilde{\Lambda}^3+\cdots)v
}
\nonumber
\\
&=
-i(I+\tilde{\Lambda})\beta_2(t,x)D_x^2
(I+\tilde{\Lambda}+\tilde{\Lambda}^2+\tilde{\Lambda}^3+\cdots)
\overline{v}
\nonumber
\\
&=
-i\beta_2(t,x)D_x^2
(I+\tilde{\Lambda}+\tilde{\Lambda}^2+\tilde{\Lambda}^3+\cdots)
\overline{v}
\nonumber
\\
&\quad 
-i\tilde{\Lambda}\beta_2(t,x)D_x^2
(I+\tilde{\Lambda}+\tilde{\Lambda}^2+\tilde{\Lambda}^3+\cdots)
\overline{v}
\nonumber
\\
&=
-i\beta_2(t,x)D_x^2\overline{v}
-i\beta_2(t,x)D_x^2\tilde{\Lambda}
(I+\tilde{\Lambda}+\tilde{\Lambda}^2+\tilde{\Lambda}^3+\cdots)
\overline{v}
\nonumber
\\
&\quad 
-i\tilde{\Lambda}\beta_2(t,x)D_x^2
(I+\tilde{\Lambda}+\tilde{\Lambda}^2+\tilde{\Lambda}^3+\cdots)
\overline{v}
\nonumber
\\
&=
i\beta_2(t,x)\p_x^2\overline{v}
-i
\left[
\tilde{\Lambda}, 
\beta_2(t,x)D_x^2
\right]
(I+\tilde{\Lambda}+\tilde{\Lambda}^2+\tilde{\Lambda}^3+\cdots)
\overline{v}
\nonumber
\\
&\quad 
-2i\beta_2(t,x)D_x^2\tilde{\Lambda}
(I+\tilde{\Lambda}+\tilde{\Lambda}^2+\tilde{\Lambda}^3+\cdots)
\overline{v}
\nonumber
\\
&=
i\beta_2(t,x)\p_x^2\overline{v}
-2i\beta_2(t,x)D_x^2\tilde{\Lambda}\overline{v}
+P(t)\overline{v}
\nonumber
\end{align}
and 
\begin{align}
\Lambda 
i(\p_x\beta_2)(t,x)
\p_x\overline{u}
&=
i(\p_x\beta_2)(t,x)
\p_x\overline{v}
+P(t)\overline{v}. 
\nonumber
\end{align}
Combining them, we have \eqref{eq:c22}. 
\par
Furthermore, by elementary pseudodifferential calculus, we have 
\begin{align}
ia\left[
\tilde{\Lambda}, D_x^4
\right]
&=
-\Phi^{\prime}(x)D_x^2+\dfrac{3}{2}i\Phi^{\prime\prime}(x)D_x
+P(t)
\nonumber
\\
&=
L\phi(x)\p_x^2+\dfrac{3L}{2}\phi^{\prime}(x)\p_x+P(t), 
\nonumber
\\
ia\left[
\tilde{\Lambda}, D_x^4
\right]\tilde{\Lambda}
&=
-\dfrac{\Phi^{\prime}(x)}{4a}\Phi(x)D_x
+P(t)
=iL\dfrac{\phi(x)}{4a}\Phi(x)\p_x+P(t), 
\nonumber
\\
ib\left[
\tilde{\Lambda}, D_x^3
\right]
&=
ib\dfrac{3i}{4a}\Phi^{\prime}(x)D_x+P(t)
=i\dfrac{3b}{4a}L\phi(x)\p_x+P(t), 
\nonumber
\\
2i\beta_2(t,x)D_x^2\tilde{\Lambda}
&=
2i\dfrac{1}{4a}\beta_2(t,x)\Phi(x)D_x+P(t)
=
\dfrac{1}{2a}\beta_2(t,x)\Phi(x)\p_x+P(t). 
\nonumber
\end{align}
Combining them, we obtain 
\begin{align*}
  \p_tv
& =
  i a\p_x^4v +b\p_x^3v
  +
  L\p_x\{\phi(x)\p_xv\}
  +
  i \p_x\{\beta_1(t,x)\p_xv\}
  \\
  &\quad
  +
   i \p_x\{\beta_2(t,x)\overline{\p_xv}\}
 +
  \left\{\operatorname{Re}\gamma_1(t,x)+\frac{L}{2}\phi^\prime(x)
  \right\}
  \p_xv
\nonumber
\\&\quad
  +
  i 
  \left\{
  \operatorname{Im}\gamma_1(t,x)-L\frac{\phi(x)\Phi(x)}{4a}
  -\dfrac{3b}{4a}L\phi(x)
  \right\}
  \p_xv
\\
&\quad + 
 \left\{
 \gamma_2(t,x)-\dfrac{1}{2a}\beta_2(t,x)\Phi(x)
 \right\}
 \overline{\p_xv}
  +
  P(t)v+P(t)\overline{v}. 
\end{align*}
\par 
Using this and the integration by parts leads to 
\begin{align}
  \frac{d}{dt}
  \int_{\mathbb{R}}
  \lvert{v}\rvert^2
  dx
&=
  -
  2L
  \int_{\mathbb{R}}
  \phi(x)
  \lvert{\p_xv}\rvert^2
  dx
  -2\Re\int_{\RR}
  i\beta_1(t,x)|\p_xv|^2dx
\nonumber
\\
&\quad
  -2\Re\int_{\RR}
  i\beta_2(t,x)\overline{\p_xv}^2dx
  \nonumber
  \\
  &\quad 
  +
  2\Re
  \int_{\RR}
  \left\{\operatorname{Re}\gamma_1(t,x)+\frac{L}{2}\phi^\prime(x)\right\}
  \p_xv \,\overline{v}
  dx
 \nonumber
\\
&\quad
+
  2\Re
  \int_{\RR}
  i 
   \left\{
   \operatorname{Im}\gamma_1(t,x)-L\frac{\phi(x)\Phi(x)}{4a}
   -\dfrac{3b}{4a}L\phi(x)
   \right\}
   \p_xv\,\overline{v}
  dx
  \nonumber
\\
&\quad
+2\Re
  \int_{\RR}
  \left\{
  \gamma_2(t,x)-\dfrac{1}{2a}\beta_2(t,x)\Phi(x)
  \right\}
  \overline{\p_xv}\,\overline{v}
  dx
  \nonumber
  \\
  &\quad
 +
  2
  \operatorname{Re}
  \int_{\mathbb{R}}
  \{P(t)v\}\overline{v}
  dx+
  2
    \operatorname{Re}
    \int_{\mathbb{R}}
    \{P(t)\overline{v}\}\overline{v}
    dx. 
  \label{eq:aa19}
\end{align} 
By assumption \eqref{eq:addd3292}, we have 
\begin{align*}
-2\Re\int_{\RR}
  i\beta_1(t,x)|\p_xv|^2dx
  &=
2\Im\int_{\RR}
  \beta_1(t,x)|\p_xv|^2dx
\leqslant 
2\int_{\RR}
  \phi_A(x)|\p_xv|^2dx, 
\\
-2\Re\int_{\RR}
  i\beta_2(t,x)\overline{\p_xv}^2dx
  &
  \leqslant 2
  \int_{\RR}|\beta_2(t,x)||\p_xv|^2dx
  \leqslant 
  2\int_{\RR}
    \phi_A(x)|\p_xv|^2dx. 
\end{align*}
By assumptions \eqref{eq:addd3292}-\eqref{eq:addd3293} and 
the definition of $\phi$ and $\Phi$,  
$$
\operatorname{Im}\gamma_1(t,x)-L\frac{\phi(x)\Phi(x)}{4a}
   -\dfrac{3b}{4a}L\phi(x)
   =O(\phi_B(x))+O(\phi(x)).
$$
This combined with the Young inequality yields 
\begin{align}
&\left|
2\Re
  \int_{\RR}
  i 
   \left\{
   \operatorname{Im}\gamma_1(t,x)
   -L\frac{\phi(x)\Phi(x)}{4a}
      -\dfrac{3b}{4a}L\phi(x)
   \right\}
   \p_xv\,\overline{v}
  dx
  \right|
\nonumber
\\
&\leqslant 
\int_{\RR}|\phi_B(x)|^2|\p_xv|^2dx
+
C\int_{\RR}|v|^2dx
\nonumber
\\
&\quad
+\int_{\RR}
\phi(x)|\p_xv|^2\,dx+
C\int_{\RR}\phi(x)|v|^2dx
\nonumber
\\
& \leqslant
 2 \int_{\mathbb{R}}
  \phi(x)
  \lvert{\p_xv}\rvert^2
  dx
  +
  C
  \int_{\mathbb{R}}
  \lvert{v}\rvert^2
  dx.
\nonumber
\end{align}
In addition, by using the integration by parts, 
the fourth and sixth terms of the right hand side of \eqref{eq:aa19} 
are bounded by 
$C\int_{\RR}|v|^2dx$. 
Hence, for any $T>0$ there exists a constant $C_T$ such that 
$$
\frac{d}{dt}
\int_{\mathbb{R}}
\lvert{v}\rvert^2
dx
\leqslant
-(2L-6)
\int_{\mathbb{R}}
\phi(x)
\lvert{\p_xv}\rvert^2
dx
+
C_T
\int_{\mathbb{R}}
\lvert{v}\rvert^2
dx
\leqslant 
C_T
\int_{\mathbb{R}}
\lvert{v}\rvert^2
dx
$$
for $t\in[0,T]$. This implies that 
$$
\int_{\RR}|v(t,x)|^2\,dx
\leqslant 
\left(\int_{\RR}|v(0,x)|^2\,dx\right) \exp(C_Tt)
$$  
for $t\in [0,T]$. 
The same inequality holds for the negative direction of $t$. 
Using these energy estimates, we can prove 
Proposition~\ref{prop:linear}.  
We omit the other parts.
\end{proof}
\section*{Acknowledgments}
This work was supported by 
JSPS Grant-in-Aid for Scientific Research (C) 
Grant Numbers JP20K03703 and JP24K06813. 
The author would like to express his sincere gratitude to Hiroyuki Chihara for 
conducting a collaborative research to publish \cite{CO2}, 
the valuable experience of which has helped the author to perform this work.  


\begin{thebibliography}{00}
  
\bibitem{BS}
J.~L.~Bona and R.~Smith, 
The initial value problem for the Korteweg-de Vries equation, 
 Philos. Trans. R. Soc. Lond., Ser. A,  
\textbf{278} (1975), 555--601.

\bibitem{CSU}
N.-H.~Chang, J.~Shatah and K.~Uhlenbeck,
Schr\"odinger maps, 
Comm.\ Pure Appl.\ Math.,  
\textbf{53}  (2000), 590--602. 
 
 \bibitem{chihara1999}
 H.~Chihara, 
The initial value problem for the elliptic-hyperbolic Davey-Stewartson equation, 
J.~Math.~ Kyoto Univ., 
 {\bf 39} (1999), 41--66. 
 
 \bibitem{CO2} 
 H.~Chihara and E.~Onodera,  
A fourth-order dispersive flow into K\"ahler manifolds, 
Z.\ Anal.\ Anwend., \textbf{34} (2015), 221--249. 

\bibitem{CK}
 R.~Cipolatti and O.~Kavian, 
On a nonlinear Schr\"odinger equation modelling ultra-short laser pulses 
 with a large noncompact global attractor, 
Discrete Contin. Dyn. Syst.,
 \textbf{17} (2007), 121--132.
 
 \bibitem{DKA}
 M.~Daniel, L.~Kavitha and R.~Amuda, 
Soliton spin excitations in an anisotropic Heisenberg ferromagnet 
 with octupole-dipole interaction, 
 Phys.\ Rev.\ B, \textbf{59} (1999), 13774.
 
 \bibitem{DL}
 M.~Daniel and M.~M.~Latha, 
Soliton in discrete and continuum alpha helical proteins with higher-order excitations, 
 Phys. A, \textbf{240} (1997), 526--546. 

\bibitem{DW2018}
Q.~Ding and Y.D.~Wang, 
Vortex filament on symmetric Lie algebras and 
generalized bi-Schr\"odinger flows, 
Math.\ Z., \textbf{290} (2018), 167--193. 

\bibitem{DZ2021}
Q.~Ding and S.~Zhong, 
On the vortex filament in 3-spaces and its generalizations,  
Sci. China Math., 
\textbf{64} (2021), 1331--1348. 

\bibitem{ET}
 M.~B.~Erdogan and N.~Tzirakis, 
Dispersive Partial Differential Equations; Wellposed-ness and Applications, 
 Cambridge Student Texts, 86, Cambridge University Press, 2016.
 
 \bibitem{fukumoto}
 Y.~Fukumoto, 
Three-dimensional motion of a vortex filament and its relation 
 to the localized induction hierarchy, 
Eur.\ Phys.\ J. B, {\bf 29} (2002), 167--171.
 
 \bibitem{FM}
 Y.~Fukumoto and T.~K.~Moffatt, 
Motion and expansion of a viscous vortex ring.
 Part 1. A higher-order asymptotic formula for the velocity, 
J.\ Fluid.\ Mech., {\bf 417} (2000), 1--45.
 
 \bibitem{GS}
 R.~Ghanmi and T.~Saanouni, 
Defocusing fourth-order coupled nonlinear Schr\"odinger equations, 
Electron. J. Differential Equations,  
No.96, (2016), 1--24. 
 
 \bibitem{GiSa}
 J.-M.~Ghidaglia and J.-C.~Saut, 
On the initial value problem for Davey-Stewartson systems,
Nonlinearity,  
 {\bf 3} (1990), 475--506. 
 
 \bibitem{HHW2006}
C.~Hao, L.~Hsiao and B.~Wang, 
Well-posedness for the fourth order nonlinear Schr\"odinger equations, 
J.\ Math.\ Anal.\ Appl., 
{\bf 320} (2006), 246--265. 

\bibitem{HHW2007}
C.~Hao, L.~Hsiao and B.~Wang, 
Well-posedness of Cauchy problem for the fourth order nonlinear Schr\"odinger equations in multi-dimensional spaces,  
J.\ Math.\ Anal.\ Appl.,   
{\bf 328} (2007), 58--83.  

\bibitem{HO1995}
N.~Hayashi and T.~ Ozawa, 
Schr\"odinger equations with nonlinearity of integral type, 
Discrete Contin. Dyn. Syst.,  
{\bf 1} (1995), 475--484. 

\bibitem{hayashi1997}
N.~Hayashi, 
Local existence in time of solutions to the elliptic-hyperbolic Davey-Stewartson
system without smallness condition on the data, 
J.~Anal.~Math., 
{\bf 73} (1997), 133--164. 
 
 \bibitem{HIT} 
 H.~Hirayama, M.~Ikeda and T.~Tanaka,
 Well-posedness for the fourth-order Schr\"odinger equation 
 with third order derivative nonlinearities,  
NoDEA Nonlinear Differential Equations Appl.,  
 {\bf 28}  Paper No. 46 (2021), 72 pp.
 
 \bibitem{HJ2005}
 Z.~Huo and Y.~Jia, 
The Cauchy problem for the fourth-order nonlinear Schr\"odinger equation 
 related to the vortex filament,  
J.\ Differential Equations, 
 {\bf  214} (2005), 1--35. 
 
 \bibitem{HJ2007}
  Z.~Huo and Y.~Jia, 
 A refined well-posedness for the fourth-order nonlinear Schr\"odinger equation 
 related to the vortex filament, 
Comm.\ Partial Differential Equations, 
{\bf 32} (2007), 1493--1510.

\bibitem{HJ2011}
Z.~Huo and Y.~Jia, 
Well-posedness for the fourth-order nonlinear derivative Schr\"odinger 
equation in higher dimension,  
J.\ Math.\ Pures Appl.,  
{\bf 96} (2011), 190--206. 

\bibitem{II}
R.~Iorio and V.~M.~Iorio, 
Fourier Analysis and Partial Differential Equations, 
Cambridge Stud. Adv. Math., 70, Cambridge University Press, 2001. 

\bibitem{LPD}
M.~Lakshmanan, K.~Porsezian and M.~Daniel,  
Effect of discreteness on the continuum limit of the Heisenberg
spin chain, 
Phys.\ Lett.\ A, \textbf{133} (1988), 483--488. 

\bibitem{LP}
F.~Linares and G.~Ponce, 
Introduction to nonlinear dispersive equations, 
2nd ed, Springer Verlag, New York, 2015. 

\bibitem{Malham}
S.~J.~A. Malham, 
Integrability of local and non-local non-commutative fourth-order quintic non-linear 
Schr\"odinger equations, 
IMA J. Appl. Math., 
\textbf{87} (2022), 231--259.  

\bibitem{Mietka}
C.~Mietka, 
On the well-posedness of a quasi-linear Korteweg-de Vries equation, 
Ann. Math. Blaise Pascal, 
\textbf{24} (2017), 83--114. 

\bibitem{Mizuhara}
R.~Mizuhara, 
The initial value problem for third and fourth order dispersive 
equations in one space dimension, 
Funkcial. Ekvac.,  
\textbf{49} (2006), 1--38. 

\bibitem{NO}
 K.~Nakamura and T.~Ozawa, 
 Finite charge solutions to cubic Schr\"odinger equations with a nonlocal nonlinearity in one space dimension,  
 Discrete Contin. Dyn. Syst., 
 \textbf{33} (2013), 789--801. 
 
\bibitem{onodera0} 
E.~Onodera, 
Generalized Hasimoto transform of one-dimensional 
dispersive flows into compact Riemann surfaces, 
SIGMA Symmetry Integrability Geom.\ Methods Appl., 
\textbf{4}  article No. 044 (2008), 10 pp.  

\bibitem{onodera4}
E.~Onodera, 
Local existence of a fourth-order dispersive 
curve flow on locally Hermitian symmetric spaces and its application, 
Differential\ Geom.\ Appl., 
\textbf{67} 101560 (2019), 26 pp. 

\bibitem{onoderamomo}
E.~Onodera, 
Structure of a fourth-order dispersive flow equation 
through the generalized Hasimoto transformation, 
preprint, 2024, arXiv: 2405.00412. 

 \bibitem{OYY}
T.~ Ozawa, K.~Yamauchi and Y.~Yamazaki, 
Analytic smoothing effect for solutions to Schr\"odinger equations with 
nonlinearity of integral type,  
Osaka J. Math., 
\textbf{42} (2005), 737--750. 

\bibitem{PDL} 
K.~Porsezian, M.~Daniel and M.~Lakshmanan,  
On the integrability aspects of the one-dimensional classical 
continuum isotropic biquadratic Heisenberg spin chain, 
J. Math.\ Phys.,  \textbf{33} (1992), 1--10. 

\bibitem{RRS}
I.~Rodnianski, Y.~A.~Rubinstein and G.~Staffilani,   
On the global well-posedness of the one-dimensional 
Schr\"odinger map flow,  
Analysis and PDE, \textbf{2} (2009), 187--209. 

\bibitem{RWZ} 
M.~Ruzhansky, B.~Wang and H.~Zhang,
Global well-posedness and scattering for the fourth order nonlinear 
Schr\"odinger equations with small data in modulation and Sobolev spaces,  
J.\ Math.\ Pures Appl.,   
{\bf 105} (2016), 31--65. 

\bibitem{segata2003}
J.~Segata, 
Well-posedness for the fourth-order nonlinear Schr\"odinger-type 
equation related to the vortex filament,  
Differential Integral Equations,  
{\bf 16} (2003), 841--864.

\bibitem{segata2004}
J.~Segata, 
Remark on well-posedness for the fourth order nonlinear Schr\"odinger 
type equation, 
Proc.\ Amer.\ Math.\ Soc., 
{\bf 132} (2004), 3559--3568.  
 
 \bibitem{segata}
 J.~Segata, 
 Refined energy inequality with application to well-posedness 
 for the fourth order nonlinear Schr\"odinger type equation 
 on torus, 
 J. Differential Equations,  \textbf{252} (2012), 5994--6011.   
 
 \bibitem{SW2011}
 X.~W.~Sun and Y.~D.~Wang,  
 KdV geometric flows on K\"ahler manifolds, 
 Internat.\ J.\ Math.,  \textbf{22} (2011), 1439--1500. 
 
 \bibitem{Tarama}
 S.~Tarama, 
 $L^2$-well-posed Cauchy problem for fourth-order dispersive 
 equations on the line, 
 Electron. J. Differential Equations,    
 No. 168 (2011), 11 pp. 
 
 \bibitem{tarulli}
 M.~Tarulli, 
 $H^2$-scattering for systems of weakly coupled fourth-order NLS 
 equations in low space dimensions,  
 Potential Anal., 
 {\bf 51} (2019), 291--313. 
 
 \bibitem{WZY}
 W.~Weng, G.~Zhang and Z.~Yan, 
 Strong and weak interactions of rational vector rogue waves and solitons 
 to any $n$-component nonlinear Schr\"odinger system 
 with higher-order effects,  
 Proc.~A.,
 \textbf{478} no. 2257, Paper No. 20210670 (2022), 23 pp. 

\end{thebibliography}
\end{document}